\documentclass[reqno]{amsart}
\usepackage{amsmath}
\usepackage{amsfonts}%
\usepackage{amssymb}
\usepackage{mathrsfs}
\usepackage[backref]{hyperref}

\usepackage{enumerate}
\usepackage{cite}
\usepackage{color}
\allowdisplaybreaks
\date{\today}

 \numberwithin{equation}{section}

\newcommand{\dv}{\mathrm{div}\,}

\newtheorem{Theorem}{Theorem}[section]
\newtheorem{Lemma}{Lemma}[section]
\newtheorem{Proposition}{Proposition}[section]

\theoremstyle{definition}

\newtheorem{Remark}{Remark}[section]

\begin{document}
\title[ Navier-Stokes-Vlasov-Fokker-Planck Equations]
 {Strong solutions to the compressible Navier-Stokes-Vlasov-Fokker-Planck equations:
 Global existence  near the equilibrium  and large time behavior}

 \author[F.-C. Li]{Fucai Li}
\address{Department of Mathematics, Nanjing University, Nanjing
 210093, P. R. China}
 \email{fli@nju.edu.cn}

  \author[Y.-M. Mu ]{Yanmin Mu$^*$}
\address{School of Applied Mathematics, Nanjing University of Finance \& Economics, Nanjing
 210046, P. R. China}
 \email{yminmu@126.com}
\thanks{$^*$Corresponding author}

 \author[D.-H. Wang]{Dehua Wang}
\address{Department of Mathematics,   University of Pittsburgh,  Pittsburgh, PA 15260}
 \email{dwang@math.pitt.edu}

\begin{abstract}
A kinetic-fluid model describing the evolutions of disperse two-phase flows is considered.
The model  consists of the Vlasov-Fokker-Planck equation for the particles (disperse phase) coupled
 with the compressible Navier-Stokes equations for the fluid (fluid phase) through the friction force.
The friction force depends on the density,  which is different from many previous studies on
 kinetic-fluid models and is more physical in modeling but significantly more difficult in analysis.
 New approach and techniques are introduced to deal with the strong coupling of the fluid and the particles.
 The global well-posedness of strong
 solution in the three-dimensional whole space  is established
 when the initial data is a small perturbation of  some given equilibrium. Moreover, the algebraic rate of convergence of
 solution toward the equilibrium state is obtained.
 For the  periodic domain the same global well-posedness result still holds
  while the convergence rate is  exponential.

 \end{abstract}


\maketitle


\section{Introduction}

\subsection{The model}

Kinetic-fluid models are widely used in the description of the dynamics of  disperse two-phase flows.
In these two-phase flows, the disperse phase is usually considered from the statistical point of view (kinetic equation) while the dense phase is from
the hydrodynamic one (fluid equations).  The kinetic equation is  coupled with the fluid equations through the friction force.
Typical applications of two-phase flows include  the dynamics of sprays
\cite{CBG,BBJM}, diesel engines \cite{FAW,FAWa,RM1,RM2},   pollution  settling processes \cite{BWC}, rain formation \cite{FFS},
wastewater treatment \cite{BBT},
  biomedical flows \cite{BBJM}, combustion theory \cite{FAW}, and so on.

In this paper we are concerned with  the following system of partial differential equations (see \cite{CGL}) of fluid-particle flows:
\begin{align}
&\partial_t F+v\cdot\nabla_{x}F=n\nabla_{v}\cdot[(v-u)F+\nabla_{v}F],\label{v1.1}\\
&\partial_t n +\nabla\cdot(n u)=0,\label{v1.2}\\
&\partial_t(n u)+\nabla\cdot(n u\otimes u)-\Delta
u+\nabla p=n\int_{\mathbb{R}^{3}}(v-u)F\,  {\rm d}v, \label{v1.3}
\end{align}
with the initial data
\begin{align}
(F,n,u)|_{t=0}=(F_{0}(x, v),n_{0}(x), u_{0}(x)). \label{v1.4}
\end{align}
Here, the unknowns  are $F=F(t,x,v)\geq 0$ for $(t,x,v)\in \mathbb{R}^+\times \Omega\times \mathbb{R}^{3}$,
denoting the density distribution
function of particles in the phase space;  and $n=n(t,x)\geq 0, u=u(t,x)\in \mathbb{R}^{3} $
for $(t,x)\in \mathbb{R}^+\times \Omega$, denoting the mass density and
 the velocity field respectively. The pressure function  $p$   depends only on
 $n$ satisfying $p'(\cdot)> 0$. In our present work, we  take $p(n)=c_0 n^{\gamma}$ with the constants $\gamma\geq 1$ and $c_0>0$ for simplicity.  The spatial domain is
 $\Omega=\mathbb{R}^{3}$ or $\mathbb{T}^{3}$ (a periodic domain in $\mathbb{R}^3$). Compared with the model introduced in \cite{CGL},
 here we have normalized the physical constants to be one
 for simplicity and added the viscous term $-\Delta u$ in the momentum equation
 \eqref{v1.3}.

It is easy to check that for smooth solutions of the compressible Navier-Stokes-Vlasov-Fokker-Planck system \eqref{v1.1}-\eqref{v1.3}, the following  quantities are conserved/ dissipated:\\
$\bullet$\; mass conservation:
$$\frac{\rm d}{{\rm d}t}\iint_{\Omega \times \mathbb{R}^3} F \, {\rm d}x  {\rm d}v=0,
\quad \frac{\rm d}{{\rm d}t}\int_{\Omega } n \, {\rm d}x =0,$$
$\bullet$\; momentum conservation:
$$\frac{\rm d}{{\rm d}t}\Big\{\int_{\Omega} n u\, {\rm d}x  +\iint_{\Omega \times\mathbb{R}^3} v F\, {\rm d}x  {\rm d}v\Big\}=0,$$
$\bullet$\;  and energy/entropy dissipation:
\begin{align}
\frac{\rm d}{{\rm d}t}\int_{\Omega}\bigg\{n\Big(\frac{|u|^{2}}{2}+A\Big)&+\int_{\mathbb{R}^3}\Big( F\ln F+\frac{|v|^{2}}{2}F\Big) {\rm d}v\bigg\}\, {\rm d}x
 +\int_{\Omega} |\nabla u|^{2}\, {\rm d}x  \nonumber\\
&\qquad\qquad =-\iint_{\Omega \times \mathbb{R}^3} n\frac{|(v-u)F-\nabla_{v}F|^{2}}{F}\, {\rm d}x {\rm d}v   \label{energy}
\end{align}
with $A=\int^{n}\frac{p(\eta)}{\eta^{2}}{\rm d}\eta$.

Set $$M=M(v)=\frac{1}{(2\pi)^{3/2}}\exp\Big\{-\frac{|v|^{2}}{2}\Big\}.$$
From the energy/entropy dissipation \eqref{energy}, we know that $(F,n,u)\equiv( M,  1,0) $
is an equilibrium of the system \eqref{v1.1}-\eqref{v1.3}.
%
%
%
Thus it is natural to introduce the transforms
\begin{align}\label{perturb}
   F=M+\sqrt{M}f,\quad n=1+\rho
\end{align}
to rewrite the system  \eqref{v1.1}-\eqref{v1.3} as
\begin{align}
&\partial_{t}f+v\cdot \nabla_{x}f +u\cdot \nabla_{v}f-\frac{1}{2}u\cdot vf-u\cdot v\sqrt{M}\nonumber\\
 &\qquad \quad\quad\quad=\mathcal{L}f +\rho\Big(\mathcal{L}f -u\cdot \nabla_{v}f+\frac{1}{2}u\cdot vf+u\cdot v\sqrt{M}\Big),\label{v1.5}\\
 &\partial_{t}\rho+u\cdot \nabla \rho+(1+\rho)\dv u=0,\label{v1.6}\\
 &\partial_{t}u+u\cdot \nabla u+\frac{p'(1+\rho)}{1+\rho}\nabla\rho=\frac{1}{1+\rho}\Delta u-u(1+a)+b.\label{v1.7}
\end{align}
Correspondingly, the   initial data \eqref{v1.4} becomes
\begin{align}
(f,n,u)|_{t=0}=(f_0(x,v), \rho_0(x),u_{0}(x))= \Big(\frac{F_{0}-M}{\sqrt{M}}, n_0(x)-1,u_{0}(x)\Big).\label{v1.8}
\end{align}
In \eqref{v1.5}-\eqref{v1.7}, $\mathcal{L}$ is the linearized Fokker-Planck operator defined by
$$\mathcal{L}f =\frac{1}{\sqrt{M}}\nabla_{v}\cdot\Big[M\nabla_{v}\big(\frac{f}{\sqrt{M}}\big)\Big],$$
and $a=a^{f},b=b^{f}$, depending on $f$, are the moments of $f$ defined by
\begin{align}
a^{f}(t,x)=\int_{\mathbb{R}^{3}}\sqrt{M}f(t,x,v) \,{\rm d}v,\quad b^{f}(t,x)=\int_{\mathbb{R}^{3}}v\sqrt{M}f(t,x,v) \,{\rm d}v. \nonumber
\end{align}

\smallskip
\subsection{Notations}

Let $\nu(v)=1+|v|^{2}$ and denote $|\cdot|_{\nu}$ by
$$|g|_{\nu}^{2}:=\int_{\mathbb{R}^{3}}\big\{|\nabla_{v}g(v)|^{2}+\nu(v)|g(v)|^{2}\big\} \,{\rm d}v,\quad g=g(v).$$
We use $\langle \cdot,\cdot\rangle   $ to denote the inner product over the Hilbert space $L^{2}_{v}$, i.e.,
$$\langle g,h\rangle   :=\int_{\mathbf{R}^{3}}g(v)h(v) \,{\rm d}v,\quad g,h\in L^{2}_{v}.$$
For simplicity, we shall use $\|\cdot\|$ to denote the norm of  $L^{2}_{x}$ or $L^{2}_{x,v}$ when there is no confusion.
Define
$$\|g\|_{\nu}^{2}:=\iint_{\Omega\times\mathbb{R}^{3}}[|\nabla_{v}g(x,v)|^{2}+\nu(v)|g(x,v)|^{2}]\, {\rm d}x  {\rm d}v, \quad g=g(x,v).$$
For multi-indices $\alpha=(\alpha_{1},\alpha_{2},\alpha_{3})$ and $\beta=(\beta_{1},\beta_{2},\beta_{3})$, we denote by
$$\partial^{\alpha}_{\beta}\equiv \partial^{\alpha_{1}}_{x_{1}} \partial^{\alpha_{2}}_{x_{2}} \partial^{\alpha_{3}}_{x_{3}}
 \partial^{\beta_{1}}_{v_{1}}\partial^{\beta_{2}}_{v_{2}}\partial^{\beta_{3}}_{v_{3}}$$
 the partial derivatives with respect to $x=(x_1,x_2,x_3)$ and $v=(v_1,v_2,v_3)$.
 The length of $\alpha$ and $\beta$ are defined as $|\alpha|=\alpha_{1}+\alpha_{2}+\alpha_{3}$ and $ |\beta|=\beta_{1}+\beta_{2}+\beta_{3}$.
 We shall use the following norms:
 \begin{align}
    \|g\|_{H^s}:=\sum_{|\alpha|\leq s}\|\partial^\alpha g\|, \quad
    \|g\|_{H^s_{x,v}}:=\sum_{|\alpha|+|\beta|\leq s}\|\partial^\alpha_\beta g\|. \nonumber
 \end{align}

For  $g(t,x,v)$,
we decompose it as the sum of the fluid part $\mathbf{P}g$ and the particle part $\{\mathbf{I}-\mathbf{P}\}g$\,:
\begin{align}
   g=\mathbf{P}g+\{\mathbf{I}-\mathbf{P}\}g. \label{decomp}
\end{align}
Here the projection operator $\mathbf{P}$ is defined as follows:
$$\mathbf{P}: \  L^{2}\rightarrow \text{Span}\, \big\{\sqrt{M},v_{1}\sqrt{M},v_{2}\sqrt{M},v_{3}\sqrt{M}\big\},$$
and  \begin{align}
\mathbf{P}:=\mathbf{P}_{0}\oplus\mathbf{P}_{1}, \quad  \mathbf{P}_{0}f:&=a^{f}\sqrt{M}, \quad
\mathbf{P}_{1}f:=b^{f}\cdot v \sqrt{M}.\nonumber
\end{align}
This macro-micro decomposition is initiated by Guo \cite{YG} for the Boltzmann equation and later introduced in \cite{DFT} to study the Fokker-Planck type equations.
Notice that the operator $\mathcal{L}$ satisfies
$$-\int_{\mathbb{R}^{3}}g\mathcal{L}g\, {\rm d}v\geq \lambda_{0}|\{\mathbf{I}-\mathbf{P}_{0}\}g|^{2}_{\nu} , \quad \forall\, g=g(v),$$
for  some positive constant $\lambda_{0}>0$.
Meanwhile, $\mathcal{L}g $ can be computed as
$$\mathcal{L}g =\mathcal{L}\{\mathbf{I}-\mathbf{P}\}g+\mathcal{L}\mathbf{P}g=\mathcal{L}\{\mathbf{I}-\mathbf{P}\}g-\mathbf{P}_{1}g. $$
Therefore, we have
\begin{align}
\langle -\mathcal{L}\{\mathbf{I}-\mathbf{P}\}g, g\rangle \geq \lambda_{0}|\{\mathbf{I}-\mathbf{P}\}g|^{2}_{\nu},\quad
\langle -\mathcal{L}g, g \rangle    \geq \lambda_{0}|\{\mathbf{I}-\mathbf{P}\}g|^{2}_{\nu}+|b^{g}|^{2}. \label{v2.1}
\end{align}

For brevity, we define the temporal energy functional and the corresponding dissipation rate for $\big(f(t,x,v),\rho(t,x),u(t,x)\big)$ as the following:

\begin{align}
\mathcal{E}_{0}(t):=\,&\sum_{|\alpha|\leq 3}\sum_{i,j}\int_{\mathbb{R}^{3}}\partial^{\alpha}_{x}
(\partial_{x_{j}}b_{i}+\partial_{_{x_{i}}}b_{j})\partial^{\alpha}_{x}\Gamma_{i,j}(\{\mathbf{I}-\mathbf{P}\}f)\, {\rm d}x \nonumber\\
&-\sum_{|\alpha|\leq 3}\int_{\mathbb{R}^{3}}\partial^{\alpha}_{x}a\partial^{\alpha}_{x}\nabla_{x}\cdot b \, {\rm d}x, \label{v2.16}\\
\mathcal{E}_{1}(t):=\,&\|f\|^{2} +\|\rho\|^{2}+\|u\|^{2}+\sum_{1\leq |\alpha|\leq 4}\bigg\{\|\partial^{\alpha}f\|^{2}
 +\Big\|\frac{\sqrt{p'(1+\rho)}}{1+\rho}\partial^{\alpha}\rho\Big\|^{2}
+\|\partial^{\alpha}u\|^{2}\bigg\}\nonumber\\
&+\tau_{1}\mathcal{E}_{0}(t)+\tau_{2}\sum_{|\alpha|\leq 3}\int_{\mathbb{R}^{3}}\partial^{\alpha}u\cdot\partial^{\alpha}\nabla \rho \, {\rm d}x  ,\label{vz1}\\
\mathcal{D}_{1}(t):=\,&\|\nabla (a,b,\rho,u)\|^{2}_{H^{3}}+\|b-u\|^{2}_{H^{4}}\nonumber\\
&+\sum_{|\alpha|\leq 4}\big(\|\partial^{\alpha}\{\mathbf{I}-\mathbf{P}\}f\|^{2}_{\nu}+\|\partial^{\alpha}\nabla u\|^{2}\big),\label{vz2}\\
\mathcal{E}_{2}(t) :=\,&\sum_{1\leq k\leq4}C_{k}\sum_{\substack{|\beta|=k \\ |\alpha|+|\beta|\leq 4}}\|\partial^{\alpha}_{\beta}\{\mathbf{I}-\mathbf{P}\}f\|^{2},\label{vz3}\\
\mathcal{D}_{2}(t) :=\,&\sum_{\substack{1\leq |\beta|\leq 4 \\ |\alpha|+|\beta|\leq 4}}\|\partial^{\alpha}_{\beta}\{\mathbf{I}-\mathbf{P}\}f\|^{2}_{\nu},\label{vz4}\\
\mathcal{E}(t) :=\,&\mathcal{E}_{1}(t)+\tau_{3}\mathcal{E}_{2}(t),\label{vz5}\\
\mathcal{D}(t) :=\,&\mathcal{D}_{1}(t)+\tau_{3}\mathcal{D}_{2}(t),\label{vz6}
\end{align}
where $\tau_{1},\tau_{2},\tau_{3}, C_{k}(1\leq k\leq 4)$ are suitable constants to be chosen later.
In addition, in torus we know the Poincar\'{e} inequality is true, thus, the total  dissipation rate is slightly different from $\mathcal{D}(t)$. We note that
\begin{align}
\mathcal{D}_{\mathbb{T},1}(t) :=\,&\mathcal{D}_{1}(t)+\tau_{4}(\|a\|^{2}_{L^{2}}+\|\rho\|^{2}_{L^{2}})+\tau_{5}\|b+u\|^{2}_{L^{2}},\nonumber\\
 \mathcal{D}_{\mathbb{T}}(t):=\,&\mathcal{D}_{\mathbb{T},1}(t)+\tau_{3}\mathcal{D}_{2}(t),\nonumber
\end{align}
where $\tau_{4}\,\text{and}\,\tau_{5}$ are sufficiently small to be chosen later.
Throughout  this paper the letter $C$ denotes a positive (generally large) constant and
 $\lambda$ a positive (generally small) canstant, where both $C$ and $\lambda$
 may change from line to line.  The symbol $A\sim B$ means
 $\frac{1}{C} A\leq B\leq C A$ for some constant $C>0$.

\subsection{Main results}
Our aim is to establish the global well-posedness and large-time behavior of strong solutions when the norm of the initial data $\|f_{0}\|_{H^{4}_{x,v}}+\|(\rho_{0},u_{0})\|_{H^{4}}$ is sufficiently small   (near the equilibrium).  We also obtain  the different time-decay rates   depending on the spatial domain $\mathbb{R}^{3}$  or $\mathbb{T}^{3}$. We now state  the   main results as follows.

 \begin{Theorem}\label{vt1.1}
 Let $\Omega=\mathbb{R}^{3}$ and $(f_0,\rho_0,u_0)$ be the initial data. Suppose that
$F_{0}=M+\sqrt{M}f_{0}\geq 0,$ and there exists a constant $\epsilon_{0}>0$ such that
 $\|f_{0}\|_{H^{4}_{x,v}}+\|(\rho_{0},u_{0})\|_{H^{4}}<\epsilon_{0}$.
 Then, the Cauchy problem \eqref{v1.5}-\eqref{v1.8} admits a unique global solution
 $(f,\rho,u)$ satisfying  $F=M+\sqrt{M}f\geq 0$ and
 \begin{gather*}
f\in C([0,\infty);H^{4}(\mathbb{R}^{3}\times\mathbb{R}^{3}));
\quad \rho, \, u\in C([0,\infty);H^{4}(\mathbb{R}^{3})); \label{v1.9} \\
 \sup_{t\geq 0}(\|f(t)\|_{H^{4}_{x,v}}+\|(\rho,u)(t)\|_{H^{4}})\leq C\big(\|f_{0}\|_{H^{4}_{x,v}}+\|(\rho_{0},u_{0})\|_{H^{4}}\big),\label{v1.10}
\end{gather*}
for some constant $C>0$.
Moreover, if we further assume that
\begin{align}\label{intt}
   \|f_{t}(0)\|^{2}_{H^{3}_{x,v}}+\|\rho_{t}(0)\|^{2}_{H^{3}}+\|u_{t}(0)\|^{2}_{H^{3}}< +\infty,
\end{align}then
\begin{align}
\sup_{x\in \mathbb{R}^{3}}\bigg\{\sum_{|\alpha|+|\beta|\leq 1}\|\partial^{\alpha}_{\beta}f\|^{2}_{L^{2}_{v}}\bigg\}
+\|\rho\|^{2}_{L^\infty}+\|u\|^{2}_{L^{\infty}}\leq C(1+t)^{-\frac{1}{2}}\label{v1.11}
\end{align}
for some constant $C>0$ and all $t\geq 0.$
 \end{Theorem}

 \begin{Remark}
In the assumption \eqref{intt}, $f_t(0)$ is indeed defined through the Vlasov-Fokker-Planck  equation  \eqref{v1.5} as follows:
\begin{align*}
f_{t}(0) := & -v\cdot \nabla_{x}f_0 -u_0\cdot \nabla_{v}f_0+\frac{1}{2}u_0\cdot vf_0+u_0\cdot v\sqrt{M}\nonumber\\
 &+\mathcal{L}f_0 +\rho\Big(\mathcal{L}f_0 -u_0\cdot \nabla_{v}f_0+\frac{1}{2}u_0\cdot vf_0+u_0\cdot v\sqrt{M}\Big);
 \end{align*}
 and $\rho_t(0)$ and $ u_t(0)$ are defined similarly. 
\end{Remark}

  \begin{Theorem}\label{vt1.2}
 Let $\Omega=\mathbb{T}^{3}$  and $(f_0,\rho_0,u_0)$ be the initial data.
 Suppose that  $F_{0}=M+\sqrt{M}f_{0}\geq 0,$  and there exists a constant $\epsilon_{0}>0$ such that
 $\|f_{0}\|_{H^{4}_{x,v}}+\|(\rho_{0},u_{0})\|_{H^{4}}<\epsilon_{0}$,
  and
 \begin{align}
 \int_{\mathbb{T}^{3}}a_{0}\, {\rm d}x =0, \quad  \int_{\mathbb{T}^{3}}\rho_{0}\, {\rm d}x =0, \quad \int_{\mathbb{T}^{3}}\left(b_{0}+(1+\rho_{0})u_{0}\right)\, {\rm d}x =0, \nonumber
 \end{align}
 where $a_0=\int_{\mathbb{T}^{3}}\sqrt{M}f_0(x,v) \,{\rm d}v$ and
 $b_0=\int_{\mathbb{T}^{3}}v\sqrt{M}f_0(x,v) \,{\rm d}v. $
Then, the Cauchy problem \eqref{v1.5}-\eqref{v1.8} admits a unique global solution
 $(f,\rho,u)$ satisfying $ F=M+\sqrt{M}f\geq 0$ and
 \begin{gather*}
f\in C([0,\infty);H^{4}(\mathbb{T}^{3}\times\mathbb{R}^{3}));\quad \rho, \, u\in C([0,\infty);H^{4}(\mathbb{T}^{3}));\nonumber\\
\|f(t)\|_{H^{4}_{x,v}}+\|(\rho,u)(t)\|_{H^{4}}
\leq C(\|f_{0}\|_{H^{4}_{x,v}}+\|\rho_{0},u_{0}\|_{H^{4}})e^{-\lambda t}, \label{v1.12}
\end{gather*}
for some constant $\lambda>0$ and any $t\geq 0.$
 \end{Theorem}

 We remark that in the papers \cite{CDM,DL,DFT} for the  related systems the time-decay rates  are optimal in the whole space case.
 Here our  time-decay rate $(1+t)^{-\frac{1}{2}}$ in whole space case for the solution  of  \eqref{v1.5}-\eqref{v1.7} is not optimal.
The main reason is that, due to the strong coupling of the nonlinear terms in \eqref{v1.5}-\eqref{v1.7},   the spectral  analysis cannot be carried out
directly. We thus present another way to obtain the decay rate.  In the periodic case, by the Poincar\'e inequality we  obtain the exponential decay 
which  coincides with
those in \cite{CDM,DL,DFT} in some sense. 

\subsection{Some known results on kinetic-fluid models}

There exist  many versions or variants of kinetic-fluid models, depending on the   physical regimes under consideration,
such as the compressibility of the fluid, viscosity of the fluid, species of particles,   interactions
between the fluid and particles,   motion of the particles, and so on. Below we review some kinetic-fluid models
related to our system \eqref{v1.1}-\eqref{v1.3}. We discuss the case that both the fluid and the particle phases are isothermal.
For the kinetic-fluid models  with energy exchange involved, we refer the reader to \cite{BBF,TSB}.
 Generally speaking, the kinetic-fluid models
can be divided into two categories: incompressible models and compressible models.

\subsubsection{Incompressible kinetic-fluid models}
If we assume that  the fluid is incompressible, we obtain the incompressible kinetic-fluid models.
The mathematical analysis of  incompressible  kinetic-fluid models  has received much attention recently.
In  \cite{KH}, Hamdache established global existence and large-time behavior of solutions
for the Vlasov-Stokes system. 
Boudin, Desvillettes, Grandmont and Moussa \cite{BDGM} proved the global existence of weak solutions to the incompressible Vlasov-Navier-Stokes system on a periodic domain.
Later, this result was extended to a bounded domain by Yu \cite{Yu}.
Goudon, He, Moussa and Zhang \cite{GHMZ} established the global existence of classical solutions near the equilibrium   for the  incompressible Navier-Stokes-Vlasov-Fokker-Planck system, meanwhile
 Carrillo, Duan and Moussa \cite{CDM} studied the corresponding inviscid case. Chae, Kang and Lee \cite{CKL2} obtained the global existence of weak and classical solutions
for the Navier-Stokes-Vlasov-Fokker-Planck equations in a torus.
Benjelloun, Desvillettes and Moussa \cite{BDM} obtained the existence of global weak
solutions to the incompressible Vlasov-Navier-Stokes system with  a fragmentation kernel.
Goudon, Jabin and Vasseur \cite{GJV1,GJV2}  investigated the hydrodynamic limits to the incompressible Vlasov-Navier-Stokes system
by means of some scaling  and convergence methods.

Assume that the fluid is incompressible and  inhomogeneous, Wang and Yu \cite{WY} obtained the global weak solution to the  Navier-Stokes-Vlasov equations,
 while Goudon, Jin, Liu and Yan \cite{GJLY} presented some numerical analysis on this model.

\subsubsection{Compressible kinetic-fluid models}
Mellet and Vasseur \cite{MV1,MV2}  studied the following compressible Navier-Stokes-Vlasov-Fokker-Planck system:
\begin{equation}\label{MV-eq}
\left\{\begin{aligned}
 &\partial_t f+v\cdot\nabla_{x}f+\dv_{v}(F_d f -\nabla_vf)=0, \\
&\partial_t n +\nabla_{x}\cdot(n u)=0,\\
&\partial_t(n u)+\nabla_{x}\cdot(n u\otimes u)-\Delta u+\nabla_{x} p=-\int_{\mathbb{R}^{3}}F_d f\,{\rm d}v,
\end{aligned}  \right.
\end{equation}
where
\begin{align}\label{dragforce}
 F_d=F_0(u-v), \ \ F_0>0  \text{ a constant. }
 \end{align}
  In \cite{MV1} they obtained the global existence of weak solutions of \eqref{MV-eq},  and in \cite{MV2} they studied the
asymptotic analysis of the solutions. In \cite{CKL}, Chae, Kang and  Lee studied
 the existence of the global classical solutions close to an equilibrium  and obtained exponential decay of the system to the system
 \eqref{MV-eq}.

When the viscous term $-\Delta u$ in the system \eqref{MV-eq} is dropped, Duan and Liu \cite{DL} studied the global well-posedenss of small
solution in the perturbation framework.
Carrillo and Goudon \cite{CG}  investigated the dissipative quantities, equilibria and their stability properties.
Morevoer, they also studied some asymptotic problems  and the derivation of macroscopic two-phase models.

As pointed out in \cite{MV2}, the choice of drag force \eqref{dragforce} may not be the most relevant one
from a physical point of view.  It could   be more relevant from a physical
point of view to assume that $F_d$ depends on the density of the fluid, such as
\begin{align}\label{dragforce2}
 F_d = n(u-v).
\end{align}
To our best knowledge, the first rigorous mathematical result concerning the case of the drag force depending on  the density $n$
was obtained by Baranger and  Desvillettes \cite{BD}, where the following inviscid system:
\begin{equation}\label{BD-eq}
\left\{\begin{aligned}
 &\partial_t f+v\cdot\nabla_{x}f+n  \nabla_{v}\cdot( f( u - v))=0, \\
&\partial_t n +\nabla_{x}\cdot(n u)=0,\\
&\partial_t(n u)+\nabla_{x}\cdot(n u\otimes u)+\nabla_{x} p=-n \int_{\mathbb{R}^{3}}f( u - v)\,{\rm d}v
\end{aligned}  \right.
\end{equation}
was considered and the local-in-time  classical solutions were given. In \cite{CGL}, the model \eqref{v1.1}-\eqref{v1.3} was
introduced but no mathematical result  was presented.
In this paper, we study the global existence and large-time  behavior of solutions to the model
\eqref{v1.1}-\eqref{v1.3}.  Our result is the first one on  the global existence  of
the strong solution to the kinetic-fluid model when the
drag force depends on the density.

\subsection{Strategy of the proofs of our main results}


Compared with the studies on the compressible kinetic-fluid models in literature,
the analysis of the system \eqref{v1.1}-\eqref{v1.3} or \eqref{v1.5}-\eqref{v1.7} 
is  more complicated and  difficult in mathematics, as explained below.

Our approach is different from those used to obtain the global existence of classical  solutions in \cite{GHMZ,CKL2,CDM} on the incompressible kinetic-fluid models
and \cite{CKL,DL} on the compressible kinetic-fluid models where the drag force is independent of the density.
Due to the nonlinear terms caused by the strong coupling  of $f$ and $\rho$ in the system \eqref{v1.5}-\eqref{v1.7},
we cannot use the existing results on the
Vlasov-Fokker-Planck system to obtain the regularity of $f$. Thus,
 we have to deal with the derivative of the particle velocity $v$ in our arguments.


Our main difficulties in obtaining the large time behavior of solutions
come again from the strong coupling of $f$ and $\rho$ in the system \eqref{v1.5}-\eqref{v1.7}.
It prevents us  from taking the advantage of the linearized spectral analysis to
gain the rate of convergence of solutions as   in  \cite{GHMZ,CKL2,CDM,CKL,DL}.
To overcome these difficulties, we shall construct some   novel functionals and adopt with modification some techniques  in \cite{DH} to deal with the  compressible Navier-Stokes equations and in \cite{LY} for the Landau equation.
Through the detailed analysis on the strong coupling terms of $f$ and $\rho$ we  obtain the desired estimates.
We believe that the methods developed in this paper can be applied to study the more complicated models in \cite{BBF,TSB},
which is our forthcoming research project.

The rest of paper is organized as follows.
In Section 2,
we shall establish the global existence of classical solutions to the  problem \eqref{v1.5}-\eqref{v1.8}
   in the spatial domain $\Omega=\mathbb{R}^3$ or $\mathbb{T}^3$.
By the fine energy estimates
 we mainly use the local existence of strong solutions and  continuum argument, motivated by
\cite{YG1,YG2} for the Boltzmann and Landau equations.
In Section 3, by means of the energy estimates for the temporal derivative of the system \eqref{v1.5}-\eqref{v1.7} and the Gronwall-type inequality, we eventually obtain the large time behavior.
In Section 4, we shall establish the uniform a priori estimates with the aid of some energy functionals and corresponding dissipation rate. Although the uniform a priori estimates obtained in Section 4 are needed in Sections 2 and 3, we shall present this lengthy part in the last section of the paper for the convenience of readers.

In  the rest of this paper, we   shall  omit the integral domain $\Omega\times \mathbb{R}^3$ or
$\mathbb{R}^3$ in the integrals for simplicity.

\bigskip

\section{Global  existence of the classical solutions}
   In this section, we shall establish the global existence of classical solutions to the  problem \eqref{v1.5}-\eqref{v1.8}
   in the spatial domain $\Omega=\mathbb{R}^3$ or $\mathbb{T}^3$. It is well known that
by the uniform a priori estimates we shall obtain    the global existence of solutions with the help of the local existence as well as
the continuum argument, under the smallness and regularity conditions on the initial data.
Here we first construct the iteration process
   to obtain the unique local solution,  then the global existence of solutions follows from the  continuum argument
and   the uniform a priori estimates   obtained in Section 4.

Now we define iteratively the sequence $(F^{n},\rho^{n},u^{n})_{n=0}^{\infty}$ as the solutions to the system:
\begin{equation*}
  \left\{\begin{aligned}
&\partial_{t}F^{n+1}+v\cdot \nabla_{x}F^{n+1}-(1+\rho^{n})\nabla_{v}\big(v F^{n+1}+\nabla_{v}F^{n+1}\big)\nonumber\\
&\qquad \qquad \qquad =-(1+\rho^{n})u^{n}\cdot\nabla_{v}F^{n+1},\\
&\partial_{t}\rho^{n+1}+u^{n}\cdot \nabla \rho^{n+1}+(1+\rho^{n})\ \dv u^{n}=0,\nonumber\\
&\partial_{t}u^{n+1}-\frac{1}{1+\rho^{n}}\Delta u^{n}=-u^{n}\cdot \nabla u^{n}
+\gamma (1+\rho^{n})^{\gamma-2}\nabla \rho^{n}
+u^{n}(1+a^{n})+b^{n}.\nonumber
\end{aligned}
\right.
\end{equation*}
Setting $F^{n}=M+\sqrt{M}f^{n}$, we can rewrite the above system as
\begin{align}
&\partial_{t}f^{n+1}+v\cdot \nabla_{x}f^{n+1}-(1+\rho^{n})\mathcal{L}f ^{n+1}\nonumber\\
&\qquad\qquad\qquad\qquad\qquad=-(1+\rho^{n})u^{n}\cdot\Big(\nabla_{v}f^{n+1}-\frac{v}{2}f^{n+1}-v\sqrt{M}\Big),\label{v2.31}\\
&\partial_{t}\rho^{n+1}+u^{n}\cdot \nabla \rho^{n+1}+(1+\rho^{n})\ \dv u^{n+1}=0,\label{v2.32} \\
&\partial_{t}u^{n+1}-\frac{1}{1+\rho^{n}}\Delta u^{n+1}=-u^{n}\cdot \nabla u^{n}
+\gamma (1+\rho^{n})^{\gamma-2}\nabla \rho^{n}
+u^{n}(1+a^{n})+b^{n},\label{v2.33}
\end{align}
where $n=0,1,2,\dots$, and $(f^{0},\rho^{0},u^{0})=(f_{0},\rho_{0},u_{0})$ is the starting value of iteration.

We define the solution space $X(0,T;A)$
by
\begin{align}
X(0,T;A) :=\left\{\begin{array}{c}
f\in C([0,T],H^{4}(\Omega\times\mathbb{R}^{3})),\, M+\sqrt{M}f\geq 0;\\
\rho\in C([0,T],H^{4}(\Omega))\bigcap C^{1}([0,T],H^{3}(\Omega));\\
u\in C([0,T],H^{4}(\Omega))\bigcap C^{1}([0,T],H^{2}(\Omega));\\
\sup_{0\leq t\leq T}\{\|f(t)\|_{H^{4}_{x,v}}+\|(\rho,u)\|_{H^{4}}\}\leq A;\\
\rho_{1} =\frac{1}{2}(-1+\inf \rho(0,x))>  -1; \\
\inf_{0\leq t\leq T, x\in \Omega}\rho(t,x)\geq \rho_{1}.
\end{array}
\right\}   \label{v2.34}
\end{align}

The main result of this section reads as follows.

\begin{Theorem}\label{vt2.1}
There exist $A_{0}>  0 $ and $ T^{*}> 0,$ such that if $f_{0}\in H^{4}(\Omega\times\mathbb{R}^{3}),
\rho_{0}\in H^{4}(\Omega),u_{0}\in H^{4}(\Omega)$ with $F_{0}=M+\sqrt{M}f_{0}\geq 0$
and $\mathcal{E}(0)\leq \frac{A_{0}}{2}$,
with $\mathcal{E}(0) \sim \|f_{0}\|^{2}_{H^{4}_{x,v}}+\|(\rho_{0},u_{0})\|^{2}_{H^{4}}$ (see the specific definition of $\mathcal{E}(0)$ in Section 4),
 then for each $n\geq 1$, $(f^{n},\rho^{n},u^{n})$ is well-defined with
\begin{align}
(f^{n},\rho^{n},u^{n})\in X(0,T^{*};A_{0}). \label{v2.35}
\end{align}
Moveover, the following statements hold:
\begin{enumerate}[\rm{(}1\rm{)}]
  \item $(f^{n},\rho^{n},u^{n})_{n\geq 0}$ is a Cauchy sequence in the Banach space $C([0,T^{*}];H^{3}(\Omega\times\mathbb{R}^{3}))\times C\big([0,T^{*}],H^{3}(\Omega)\big)
\times C\big([0,T^{*}],H^{3}(\Omega)\big)$,
    \item  the corresponding limit function denoted by $(f,\rho,u)$ belong to $X(0,T^{*};A_{0})$,

  \item $(f,\rho,u)$ are solutions to the Cauchy problem
\eqref{v1.5}-\eqref{v1.8},

\item  $(f,\rho,u)$ is unique in $X(0,T^{*};A_{0})$ for the problem \eqref{v1.5}-\eqref{v1.8}.

\end{enumerate}
\end{Theorem}


\begin{proof} Let $T^*>0$ be a constant which will be fixed later.
For brevity, we can assume that $(f^{n},\rho^{n},u^{n})$ are smooth enough in order to take the forthcoming calculations, otherwise,
we can consider the following regularized iteration system:
\begin{align}
&\partial_{t}f^{n+1,\epsilon}+v\cdot \nabla_{x}f^{n+1,\epsilon}-(1+\rho^{n,\epsilon})\mathcal{L}f ^{n+1,\epsilon}\nonumber\\
&\qquad\qquad\qquad=-(1+\rho^{n,\epsilon})u^{n,\epsilon}\cdot\Big(\nabla_{v}f^{n+1,\epsilon}-\frac{v}{2}f^{n+1,\epsilon}-v\sqrt{M}\Big),\nonumber\\
&\partial_{t}\rho^{n+1,\epsilon}+u^{n,\epsilon}\cdot \nabla \rho^{n+1,\epsilon}+(1+\rho^{n,\epsilon})\ \dv u^{n+1,\epsilon}=0,\nonumber \\
&\partial_{t}u^{n+1,\epsilon}-\frac{1}{1+\rho^{n,\epsilon}}\Delta u^{n+1,\epsilon}=-u^{n,\epsilon}\cdot \nabla u^{n,\epsilon}\nonumber \\
&\qquad\qquad\qquad+\gamma (1+\rho^{n,\epsilon})^{\gamma-2}\nabla \rho^{n,\epsilon}+u^{n,\epsilon}(1+a^{n,\epsilon})+b^{n,\epsilon},\nonumber\\
&  (f^{n+1,\epsilon},\rho^{n+1,\epsilon},u^{n+1,\epsilon})(0)=(f^{\epsilon}_{0},\rho^{\epsilon}_{0},u^{\epsilon}_{0})\nonumber
\end{align}
for any $\epsilon>0  $ with $(f^{\epsilon}_{0},\rho^{\epsilon}_{0},u^{\epsilon}_{0})$  a smooth approximation of $(f_{0},\rho_{0},u_{0})$ and pass to the limit
by letting $\epsilon\rightarrow 0.$

Applying $\partial^{\alpha}_{x}$ with $|\alpha|\leq 4$ to the equation \eqref{v2.31}, multiplying the result by $\partial^{\alpha}_{x}f^{n+1}$
and then taking integration over $\Omega$, one has
\begin{align}
&\quad \frac{1}{2}\frac{\rm d}{{\rm d}t}\|\partial^{\alpha}f^{n+1}\|^{2}+\int(1+\rho^{n})\langle -\mathcal{L}\partial^{\alpha}f^{n+1},\partial^{\alpha}f^{n+1}\rangle   \, {\rm d}x \nonumber\\
&=\sum_{0\langle  \beta\leq \alpha}C_{\alpha,\beta}\int\partial^{\beta}\rho^{n}\langle \mathcal{L}\partial^{\alpha-\beta}f^{n+1},\partial^{\alpha}f^{n+1}\rangle   \, {\rm d}x \nonumber\\
&\quad-\iint\partial^{\alpha}\Big\{(1+\rho^{n})u^{n}M^{-\frac{1}{2}}\nabla_{v}(M+\sqrt{M}f^{n+1})\Big\}\partial^{\alpha}f^{n+1}\, {\rm d}x  {\rm d}v\nonumber\\
&\leq C(1+\|\rho^{n}\|_{H^{4}})\|u^{n}\|_{H^{4}}\|f^{n+1}\|_{L^{2}_{v}(H^{4})}\nonumber\\
&\quad+C(1+\|\rho^{n}\|_{H^{4}})\|u^{n}\|_{H^{4}}\|f^{n+1}\|_{L^{2}_{v}(H^{4})}\|\partial^{\alpha}f^{n+1}\|_{\nu}\nonumber\\
&\quad+C\|\rho^{n}\|_{H^{4}}\Big(\sum_{|\alpha'|\leq 3}\|\partial^{\alpha'f^{n+1}}\|_{\nu}\Big)\|\partial^{\alpha}f^{n+1}\|_{\nu}.\label{v2333}
\end{align}
Notice that   $$\int(1+\rho^{n})\langle -\mathcal{L}\partial^{\alpha}f^{n+1},\partial^{\alpha}f^{n+1}\rangle   \, {\rm d}x
\geq \lambda \|\{\mathbf{I}-\mathbf{P}_{0}\}\partial^{\alpha}f^{n+1}\|^{2}_{\nu}.$$
 By adding $\|\mathbf{P}_{0}\partial^{\alpha}f^{n+1}\|^{2}_{\nu}$ to both sides
on the   inequality \eqref{v2333}, then, taking summation over $|\alpha|\leq 4$, we have
\begin{align}
&\quad \frac{1}{2}\frac{\rm d}{{\rm d}t}\sum_{|\alpha|\leq 4}\|\partial^{\alpha}f^{n+1}\|^{2}+\lambda\sum_{|\alpha|\leq 4}\|\partial^{\alpha}f^{n+1}\|^{2}_{\nu}\nonumber\\
&\leq C(1+\|\rho^{n}\|_{H^{4}})\|u^{n}\|_{H^{4}}\|f^{n+1}\|_{L^{2}_{v}(H^{4})}\nonumber\\
&\quad +C(1+\|\rho^{n}\|^{2}_{H^{4}})\|u^{n}\|^{2}_{H^{4}}\|f^{n+1}\|^{2}_{L^{2}_{v}(H^{4})}\nonumber\\
&\quad +C\|\rho^{n}\|_{H^{4}}\sum_{|\alpha|\leq 4}\|\partial^{\alpha}f^{n+1}\|^{2}_{\nu}+C\|f^{n+1}\|^{2}_{L^{2}_{v}(H^{4})}.\nonumber
\end{align}
Similarly, for any $0\leq t\leq T\leq T^*$, we obtain that
\begin{align}
&\!\!\!\!\!\!\!\!\!\!\!\!\!\!\!\! \frac{1}{2}\frac{\rm d}{{\rm d}t}\|f^{n+1}\|^{2}_{H^{4}_{x,v}}+\lambda \sum_{|\alpha|+|\beta|\leq 4}\|\partial^{\alpha}_{\beta}f^{n+1}\|^{2}_{\nu}\nonumber\\
&\leq C(1+\|\rho^{n}\|^{2}_{H^{4}})\|f^{n+1}\|^{2}_{H^{4}_{x,v}}+C\|\rho^{n}\|_{H^{4}}\sum_{|\alpha|+|\beta|\leq 4}\|\partial^{\alpha}_{\beta}f^{n+1}\|^{2}_{\nu}\nonumber\\
&\quad+C(1+\|\rho^{n}\|^{2}_{H^{4}})\|u^{n}\|^{2}_{H^{4}}(1+\|f^{n+1}\|^{2}_{H^{4}_{x,v}}).\label{v2.36}
\end{align}

Next, according to \cite{MN}, for the system \eqref{v2.32} and \eqref{v2.33}, there exists a unique solution $(\rho^{n+1},u^{n+1})$ satisfying
$\rho^{n+1}\geq \rho_{1}$, and
$$\rho^{n+1}\in C([0,T],H^{4}(\Omega))\cap C^{1}([0,T],H^{3}(\Omega)),$$
$$u^{n+1}\in C([0,T],H^{4}(\Omega))\cap C^{1}([0,T],H^{2}(\Omega)).$$

Now we  estimate $\frac{\rm d}{{\rm d}t}\|(\rho^{n+1},u^{n+1})\|^{2}_{H^{4}}$.
Applying $\partial^{\alpha}\,(|\alpha|\leq 4)$ to the system \eqref{v2.32} and \eqref{v2.33}, multiplying the results by
$\partial^{\alpha}\rho^{n+1}$ and $ \partial^{\alpha}u^{n+1}$ respectively, and then taking integration and  summation, one has
\begin{align}
&\!\!\!\!\!\!\!\!\!\!\!\!\!\!\!\!\!\!\!\!\!\!\!\!\frac{1}{2}\frac{\rm d}{{\rm d}t}\Big(\|\rho^{n+1}\|^{2}_{H^{4}}+\|u^{n+1}\|^{2}_{H^{4}}\Big)
 +\lambda\sum_{|\alpha|\leq 4}\int|\nabla \partial^{\alpha}u^{n+1}|^{2}{\rm d}x\nonumber\\
\leq\, & C(1+\|\rho^{n}\|^{2}_{H^{4}}+\|u^{n}\|^{2}_{H^{4}})\|\rho^{n+1}\|^{2}_{H^{4}}+C(1+\|\rho^{n}\|^{2}_{H^{4}})\|u^{n+1}\|^{2}_{H^{4}}\nonumber\\
&+C(\|u^{n}\|^{2}_{H^{4}}+\|f^{n}\|^{2}_{H^{4}})(1+\|u^{n}\|^{2}_{H^{4}})+C\|\rho^{n}\|^{2}_{H^{4}}(1+\|\rho^{n}\|^{6}_{H^{4}}).\label{v2.37}
\end{align}
Adding up   \eqref{v2.36} and \eqref{v2.37} gives
\begin{align}
&\!\!\!\!\!\!\!\!\!\!\!\!\!\!\!\!\frac{1}{2}\frac{\rm d}{{\rm d}t}\Big(\|f^{n+1}\|^{2}_{H^{4}_{x,v}}+\|\rho^{n+1}\|^{2}_{H^{4}}+\|u^{n+1}\|^{2}_{H^{4}}\Big)\nonumber\\
&\quad+\lambda
\sum_{|\alpha|+|\beta|\leq 4}\|\partial^{\alpha}_{\beta}f^{n+1}\|^{2}_{\nu}+\lambda\sum_{|\alpha|\leq 4}\|\nabla\partial^{\alpha}u^{n+1}\|^{2}\nonumber\\
&\leq C(1+\|\rho^{n}\|^{2}_{H^{4}})\|f^{n+1}\|^{2}_{H^{4}_{x,v}}+C(1+\|\rho^{n}\|^{2}_{H^{4}})\|u^{n}\|^{2}_{H^{4}}(1+\|f^{n+1}\|^{2}_{H^{4}_{x,v}})\nonumber\\
&\quad+C(1+\|\rho^{n}\|^{2}_{H^{4}}+\|u^{n}\|^{2}_{H^{4}})\|\rho^{n+1}\|^{2}_{H^{4}}+C(1+\|\rho^{n}\|^{2}_{H^{4}})\|u^{n+1}\|^{2}_{H^{4}}\nonumber\\
&\quad+C(1+\|u^{n}\|^{2}_{H^{4}})(\|f^{n}\|^{2}_{H^{4}_{x,v}}+\|u^{n}\|^{2}_{H^{4}})+C\|\rho^{n}\|^{2}_{H^{4}}(1+\|\rho^{n}\|^{6}_{H^{4}})\nonumber\\
&\quad+C\|\rho^{n}\|_{H^{4}}\sum_{|\alpha|+|\beta|\leq 4}\|\partial^{\alpha}_{\beta}f^{n+1}\|^{2}_{\nu}.\label{v2.38}
\end{align}
Using induction, we may assume $A_{n}(T)\leq A_{0}$ and $A_{n}(0)\leq \frac{A_{0}}{2}$ for some $A_0>0$  with
$$A_{n}(T) :=\sup_{0\leq t\leq T}\big\{\|\rho^{n}(t)\|^{2}_{H^{4}}+\|u^{n}(t)\|^{2}_{H^{4}}+\|f^{n}(t)\|^{2}_{H^{4}_{x,v}}\big\}.$$
Integrating \eqref{v2.38} over $[0,T]$ gives
\begin{align}
&\!\!\!\!\!\!\!\!\!\!\!\!\!\!\!\! A_{n+1}(T)+\lambda \int_{0}^{T}\Big\{\sum_{|\alpha|\leq 4}\|\nabla \partial^{\alpha}u^{n+1}\|^{2}
+\sum_{|\alpha|+|\beta|\leq 4}\|\partial^{\alpha}_{\beta}f^{n+1}\|^{2}_{\nu}\Big\}{\rm d} t\nonumber\\
&\leq A_{n+1}(0)+C(1+A^{\frac{1}{2}}_{n}(T)+A^{2}_{n}(T))A_{n+1}(T)T+C(A_{n}(T)+A^{4}_{n}(T))T\nonumber\\
&\quad+CA^{\frac{1}{2}}_{n}(T)\sum_{|\alpha|+|\beta|\leq 4}\|\partial^{\alpha}_{\beta}f^{n+1}\|^{2}_{\nu}\nonumber\\
&\leq\frac{A_{0}}{2}+C(1+A^{2}_{0})TA_{n+1}(T)+C(A_{0}+A^{4}_{0})T\nonumber\\
&\quad+CA^{\frac{1}{2}}_{n}(T)\sum_{|\alpha|+|\beta|\leq 4}\|\partial^{\alpha}_{\beta}f^{n+1}\|^{2}_{\nu}.\label{v2.39}
\end{align}
It follows that, for $T\leq T^{*}$,
$$\big(1-C(1+A^{2}_{0})T\big)A_{n+1}(T)\leq \frac{A_{0}}{2}+C(A_{0}+A^{4}_{0})T.$$
Choosing $T^*$ satisfying  $T^{*}\leq \frac{A_{0}}{2}$, and $A_{0}$  sufficiently small, we conclude that
$$A_{n+1}\leq A_{0}.$$
For the equation of $F^{n+1}$, with the help of the maximum principle, we have
$$F^{n+1}=M+\sqrt{M}f^{n+1}\geq 0.$$

Now we explain that $\|f^{n+1}\|^{2}_{H^{4}_{x,v}}$ is continuous over $0\leq t\leq T^{*}$. In fact,
it  follows from the  inequality:
\begin{align}
 &  \!\!\!\!\!\!\!\!\!\!\!\!\!\!\!\!\Big|\|f^{n+1}(t)\|^{2}_{H^{4}_{x,v}}-\|f^{n+1}(s)\|^{2}_{H^{4}_{x,v}}\Big|\nonumber\\
 =\,&\Big|\int_{s}^{t}\frac{\mathrm{d}}{\mathrm{d}\eta}\|f^{n+1}(\eta)\|^{2}_{H^{4}_{x,v}}\mathrm{d}\eta\Big|\nonumber\\
\leq\, & C A^{\frac{1}{2}}_{0}\sum_{|\alpha|+|\beta|\leq 4} \int_{s}^{t}\|\partial^{\alpha}_{\beta}f^{n+1}\|^{2}_{\nu}\mathrm{d}\eta
 +C(A_{0}+A^{3}_{0})|t-s|, \label{v2.40}
 \end{align}
which can be proved by the same process as    the proof of  \eqref{v2.36}.
Meanwhile, $\|\partial^{\alpha}_{\beta}f^{n+1}\|^{2}_{\nu}$ is integrable over $[0,T^{*}]$.
 Hence, \eqref{v2.35} holds true for $n+1$ and so it does for any $n\geq 0$.

Next, we study the following system:
\begin{align}
&\partial_{t}(f^{n+1}-f^{n})+v\cdot \nabla_{x}(f^{n+1}-f^{n})-L(f^{n+1}-f^{n})\nonumber\\
&\qquad=\rho^{n}\mathcal{L}f ^{n+1}+(1+\rho^{n})u^{n}(v\sqrt{M}+\frac{v}{2}f^{n+1}-\nabla_{v}f^{n+1})\nonumber\\
&\qquad\quad -(1+\rho^{n-1})u^{n-1}(v\sqrt{M}+\frac{v}{2}f^{n}-\nabla_{v}f^{n})-\rho^{n-1}\mathcal{L}f ^{n},\nonumber\\
&\partial_{t}(u^{n+1}-u^{n})-\frac{1}{1+\rho^{n}}\Delta(u^{n+1}-u^{n})\nonumber\\
& \qquad =\Big(\frac{1}{1+\rho^{n}}-\frac{1}{1+\rho^{n-1}}\Big)\Delta u^{n}\nonumber\\
&\qquad\quad -\big((u^{n}-u^{n-1})\cdot \nabla u^{n}+u^{n-1}\cdot \nabla(u^{n}-u^{n-1})\big)
\nonumber\\
&\qquad\quad +\gamma \big((1+\rho^{n})^{\gamma-2}\nabla \rho^{n}-(1+\rho^{n-1})^{\gamma-2}\nabla \rho^{n-1}\big)\nonumber\\
&\qquad\quad +b^{n}-b^{n-1}+(a^{n}-a^{n-1})u^{n-1}+(u^{n}-u^{n-1})(1+a^{n}),\nonumber\\
&\partial_{t}(\rho^{n+1}-\rho^{n})+u^{n}\cdot \nabla (\rho^{n+1}-\rho^{n})\nonumber\\
& \qquad =-(u^{n}-u^{n-1})\cdot \nabla \rho^{n}-(1+\rho^{n})\ {\rm d}v(u^{n+1}-u^{n})-(\rho^{n}-\rho^{n-1})\ \dv u^{n}.\nonumber
\end{align}
Similarly to  \eqref{v2.38}, we obtain that
\begin{align}
&\quad \frac{1}{2}\frac{\rm d}{{\rm d}t}\big\{\|f^{n+1}-f^{n}\|^{2}_{H^{3}_{x,v}}+\|\rho^{n+1}-\rho^{n}\|^{2}_{H^{3}}+\|u^{n+1}-u^{n}\|^{2}_{H^{3}}\big\}\nonumber\\
&\quad+\lambda\sum_{|\alpha|+|\beta|\leq 3}\|\partial^{\alpha}_{\beta}(f^{n+1}-f^{n})\|^{2}_{\nu}+\lambda\|\nabla(u^{n+1}-u^{n})\|^{2}_{H^{3}}\nonumber\\
&\leq C \big(1+\|\rho^{n}\|^{4}_{H^{4}}+\|u^{n}\|^{4}_{H^{4}}\big)\nonumber\\
& \quad \times \big(\|f^{n+1}-f^{n}\|^{2}_{H^{3}_{x,v}}+\|\rho^{n+1}-\rho^{n}\|^{2}_{H^{3}}+\|u^{n+1}-u^{n}\|^{2}_{H^{3}}\big)\nonumber\\
&\quad+C\|\rho^{n}\|_{H^{3}}\sum_{|\alpha|+|\beta|\leq 3}\|\partial^{\alpha}_{\beta}(f^{n+1}-f^{n})\|^{2}_{\nu}
+C(1+\|u^{n-1}\|^{2}_{H^{3}})\|f^{n}-f^{n-1}\|^{2}_{H^{3}}\nonumber\\
&\quad+C\big\{(1+\|\rho^{n}\|^{2}_{H^{4}})(1+\|f^{n}\|^{2}_{H^{3}})+\|(u^{n-1},f^{n},u^{n})\|^{2}_{H^{4}}\big\}\|u^{n}-u^{n-1}\|^{2}_{H^{3}}\nonumber\\
&\quad+C\bigg\{1+\|(u^{n},u^{n-1})\|^{2}_{H^{4}}+\|(\rho^{n},u^{n})\|^{2}_{H^{4}}\|\rho^{n}\|^{6}_{H^{4}}+(1+\|u^{n-1}\|^{2}_{H^{3}})\|f^{n}\|^{2}_{H^{3}}\nonumber\\
&\quad+\sum_{|\alpha|+|\beta|\leq 3}\|\partial^{\alpha}_{\beta}f^{n}\|^{2}_{\nu}\bigg\}\|\rho^{n}-\rho^{n-1}\|^{2}_{H^{3}}.\nonumber
\end{align}
Here we have used the Sobolev embedding $H^{2}(\Omega)\hookrightarrow L^{\infty}(\Omega)$.
Since $\mathcal{E}(0)$, $T^{*}$, and $A_{0}$ are sufficiently small, by using \eqref{v2.39}, we know that
\begin{align}
\sup_{n}\int_{0}^{T^{*}}\sum_{|\alpha|+|\beta|\leq 4}\|\partial^{\alpha}_{\beta}f^{n}\|^{2}_{\nu}\,\mathrm{d}s \nonumber
\end{align}
is also sufficiently small. Hence, there exists a constant $\kappa < 1$, such that
\begin{align}
&\sup_{0< t\leq T^{*}}\big\{\|f^{n+1}-f^{n}\|_{H^{4}}+\|\rho^{n+1}-\rho^{n}\|_{H^{4}}+\|u^{n+1}-u^{n}\|_{H^{4}}\big\}  \nonumber\\
&\qquad\leq\kappa\sup_{0< t\leq T^{*}}\big\{\|f^{n}-f^{n-1}\|_{H^{4}}+\|\rho^{n}-\rho^{n-1}\|_{H^{4}}+\|u^{n}-u^{n-1}\|_{H^{4}}\big\}.\label{v2.41}
\end{align}
According to \eqref{v2.41}, we conclude that $(f^{n},\rho^{n},u^{n})_{n\geq 0}$ is a Cauchy sequence in the Banach space
$C\big([0,T^{*}],H^{3}(\Omega\times\mathbb{R}^{3})\big)\times C\big([0,T^{*}],H^{3}(\Omega)\big)
\times C\big([0,T^{*}],H^{3}(\Omega)\big)$.
Hence, in this Banach space, there exists a limit function $(f,\rho,u)$ such that $(f,\rho,u)$ is a  solution to the Cauchy problem
 \eqref{v1.5}-\eqref{v1.8} by letting $n\rightarrow \infty$. From the fact that $F^{n}(t,x,v)\geq 0$  and the Sobolev embedding theorem, we deduce that
 $$F(t,x,v)\geq 0, \quad  \sup_{0\leq t\leq T^{*}}\|f(t)\|_{H^{4}_{x,v}}\leq A_{0}.$$
 Similarly to the proof of \eqref{v2.40}, we see  that $f\in C\big([0,T^{*}],H^{3}(\Omega\times\mathbb{R}^{3})\big).$
 Thus, we can conclude that $(f,\rho,u)\in X(0,T^{*};A_{0})$.
 Finally, let $(\bar{f},\bar{\rho},\bar{u})\in X(0,T^{*},A_{0})$  be another solution to the Cauchy problem \eqref{v1.5}-\eqref{v1.8}.
 By taking the similar process  to that of  \eqref{v2.41}, we have
 \begin{align}
 &\sup_{0< t\leq T^{*}}\big\{\|f-\bar{f}\|_{H^{4}}+\|\rho-\bar{\rho}\|_{H^{4}}+\|u-\bar{u}\|_{H^{4}}\big\}  \nonumber\\
&\qquad\leq\kappa\sup_{0< t\leq T^{*}}\big\{\|f-\bar{f}\|_{H^{4}}+\|\rho-\bar{\rho}\|_{H^{4}}+\|u-\bar{u}\|_{H^{4}}\big\}\nonumber
\end{align}
for $\kappa < 1$. Hence  we deduce that $f\equiv \bar{f}, \rho\equiv \bar{\rho}, u\equiv\bar{u}$ and uniqueness follows.
\end{proof}

Since $\mathcal{E}(0) \sim \|f_{0}\|^{2}_{H^{4}_{x,v}}+\|(\rho_{0},u_{0})\|^{2}_{H^{4}}$,   there exists $\epsilon_{0}>0$,
such that if
 $\|f_{0}\|_{H^{4}_{x,v}}+\|(\rho_{0},u_{0})\|_{H^{4}}<\epsilon_{0}$, we have $\mathcal{E}(0)\leq \frac{A_{0}}{2}$.
 Next, with the aid of Theorem \ref{vt2.1}, we obtain  the local solution on $[0,T^{*}]$  that satisfies the uniform
 a priori estimate  \eqref{v2.30} in Section 4. Finally, by taking the standard bootstrap  arguments similar to  those in \cite{DFT,YG,MN}, we
obtain the global existence and uniqueness of classical solutions  in both Theorem \ref{vt1.1} and Theorem \ref{v1.2}.


\bigskip

\section{Large time behavior of the classical solutions}
In this section, we investigate  the time-decay rates of global solutions to the problem \eqref{v1.5}-\eqref{v1.8}.
We obtain the algebraic rate of convergence of
 solution toward the equilibrium state  in the whole space case,
  while for the  periodic domain case,  the convergence rate is  exponential.

\subsection{The case of the whole space}
In this subsection we consider the large time behavior of classical solutions in the whole space $\Omega=\mathbb{R}^{3}$.
In order to obtain the desired decay rate \eqref{v1.11} in Theorem \ref{vt1.1},
we first introduce some new functionals which are similar to those \cite{DH,LY} in spirit. Then we perform the energy analysis to the temporal
derivative of the system \eqref{v1.5}-\eqref{v1.7}, instead of the original equations, to gain the one-order derivative.
Finally, we combine the energy estimates together with the uniform a priori estimates obtained in Section 4
 and  the following Gronwall-type inequality to obtain the time decay of the solutions.
For simplicity, we shall denote $\big(f_{t},\rho_{t}, u_{t}\big)=\big(\partial_{t}f, \partial_{t}\rho,\partial_{t}u\big)$.

\begin{Lemma}[\cite{DH}]\label{vl3.1}
Let $y(t)\in C^{1}([t_{0},\infty))$ satisfy $y(t)\geq 0, A=\int^{\infty}_{t_{0}} y(s)\mathrm{d}s < +\infty$  and
$y'(t)\leq a(t)y(t)$ for all $t\geq t_{0}$. If $a(t)\geq 0$ and $B=\int^{\infty}_{t_{0}}a(s)\mathrm{d}s < +\infty,$
then
$$y(t)\leq\frac{(t_{0}y(t_{0})+1)\exp(A+B)-1}{t},   \quad   \forall \, t\geq t_{0}.$$
\end{Lemma}
 First of all, we consider the   system for $(f_t,\rho_t,u_t)$:
 \begin{align}
&\partial_{t}f_{t}+v\cdot \nabla_{x}f_{t}-(1+\rho)\mathcal{L}f _{t}=\rho_{t}\mathcal{L}f -(1+\rho)u\cdot\Big(\nabla_{v}f_{t}-\frac{v}{2}f_{t}\Big)\nonumber\\
&\qquad-\big(\rho_{t}u+(1+\rho)u_{t}\big)\cdot\Big(\nabla_{v}f-\frac{v}{2}f-v\sqrt{M}\Big),\label{v3.1}\\
&\partial_{t}\rho_{t}+u\cdot\nabla\rho_{t}+\rho_{t}\ \dv u=-u_{t}\nabla \rho-(1+\rho)\ \dv u_{t},\label{v3.2}\\
&\partial_{t}u_{t}+u_{t}\cdot \nabla u + u\cdot \nabla u_{t}-\frac{1}{1+\rho}\Delta u_{t}+\frac{p'(1+\rho)}{1+\rho}\nabla \rho_{t}
+\frac{1}{(1+\rho)^{2}}\rho_{t}\Delta u\nonumber\\
&\qquad=- \gamma(\gamma-2)(1+\rho)^{\gamma-3}\rho_{t}\nabla \rho-u_{t}(1+a)-u a_{t}+b_{t}.\label{v3.3}
\end{align}
Notice that there are some  similar structures between the system \eqref{v3.1}-\eqref{v3.3} and  the system \eqref{v1.5}-\eqref{v1.7},
and we can benefit from the proofs of Lemmas \ref{vl2.2}-\ref{vl2.6}.
Thus, we will omit  some details of the proof of the following lemma.
\begin{Proposition}\label{vl3.2}
Assume that $(f,\rho,u)$ is the solution obtained in Theorem \ref{vt1.1}. Then, we have
\begin{align}
&\!\!\!\!\!\!\!\!\!\!\!\!\!\!\!\! \frac{1}{2}\frac{\rm d}{{\rm d}t}\big(\|f_{t}\|^{2}+\|\rho_{t}\|^{2}+\|u_{t}\|^{2}\big)
+\lambda\|\{\mathbf{I}-\mathbf{P}\}f_{t}\|^{2}_{\nu}+\lambda\|\nabla u_{t}\|^{2}+\lambda\|u_{t}-b_{t}\|^{2}\nonumber\\
&\leq\epsilon \|\nabla (\rho_{t},b_{t})\|^{2}
+C_{\epsilon}\big(\|\{\mathbf{I}-\mathbf{P}\}f\|^{2}_{\nu}+\|u-b\|^{2}+\|\nabla u\|^{2}_{H^{1}}\big)\|\rho_{t}\|^{2}\nonumber\\
&\quad +C(1+\|\rho\|_{H^{2}})\|u\|_{H^{2}}\|\{\mathbf{I}-\mathbf{P}\}f_{t}\|^{2}_{\nu}+ C(1+\|\rho\|^{2}_{H^{2}}+\|u\|^{2}_{H^{2}})\nonumber\\
&\quad\times\bigg\{\sum_{|\alpha|\leq 2}\|\partial^{\alpha}_{x}\{\mathbf{I}-\mathbf{P}\}f\|^{2}_{\nu}
+\|\nabla_{x}(a,b)\|^{2}_{H^{1}}
+\|\nabla_{x}(\rho,u)\|^{2}_{H^{2}}\bigg\}\nonumber\\
&\quad \times \big(\|f_{t}\|^{2}+\|(\rho_{t},u_{t})\|^{2}_{H^{2}}\big)\label{v3.4}
\end{align}
for $\epsilon >0$ sufficiently small.
\end{Proposition}
\begin{proof}
First, multiplying \eqref{v3.1} by $f_t$, \eqref{v3.2} by $\rho_t$, and \eqref{v3.3} by $u_t$, respectively, and then integrating
over $\mathbb{R}^3$ and taking  summation, we obtain that
\begin{align}
& \frac{1}{2}\frac{\rm d}{{\rm d}t}\big(\|f_{t}\|^{2}+\|\rho_{t}\|^{2}+\|u_{t}\|^{2}\big)
+\lambda\|\{\mathbf{I}-\mathbf{P}\}f_{t}\|^{2}_{\nu}+\lambda\|\nabla u_{t}\|^{2}+\|u_{t}-b_{t}\|^{2}\nonumber\\
&\leq \iint \rho_{t}\mathcal{L}f  f_{t} \, {\rm d}x  {\rm d}v + \int\rho_{t}u\cdot b_{t}\, {\rm d}x
 -\iint\big(\rho_{t}u+(1+\rho)u_{t}\big)\cdot\Big(\nabla_{v}-\frac{v}{2}\Big)f f_{t}\, {\rm d}x  {\rm d}v
\nonumber\\
&\quad-\int a u^{2}_{t}\, {\rm d}x+\frac{1}{2}\iint(1+\rho)u\cdot vf^{2}_{t}\, {\rm d}x  {\rm d}v-\int a_{t}u\cdot u_{t}\, {\rm d}x +\int\rho(u_{t}-b_{t})\cdot b_{t}\, {\rm d}x \nonumber\\
&\quad-\int u_{t}\cdot\nabla \rho \rho_{t}\, {\rm d}x +\gamma(\gamma-2)\int(1+\rho)^{\gamma-3}\rho_{t}u_{t}\cdot \nabla\rho \, {\rm d}x \nonumber\\
&\quad-\int(1+\rho)\ \dv u_{t}\rho_{t}\, {\rm d}x - \int\frac{p'(1+\rho)}{1+\rho}u_{t}\cdot \nabla \rho_{t}\, {\rm d}x \nonumber\\
&\quad-\int u\cdot \nabla \rho_{t} \rho_{t}\, {\rm d}x -\int \rho^{2}_{t}\ \dv u \, {\rm d}x
-\int u_{t}\cdot \nabla u\cdot u_{t}\, {\rm d}x -\int u\cdot \nabla u_{t}\cdot u_{t}\, {\rm d}x \nonumber\\
&\quad+\int\frac{1}{(1+\rho)^{2}}\nabla\rho \cdot \nabla u_{t}\cdot u_{t}\, {\rm d}x -\int\frac{1}{(1+\rho)^{2}}\rho_{t}\Delta u \cdot u_{t}\, {\rm d}x. \label{v3.5}
\end{align}
Next, we estimate each term on the right hand side  of \eqref{v3.5}.
For the first two terms, we rewrite them as follows:
\begin{align}
&\!\!\!\!\!\!\!\!\!\!\!\!\!\!\!\!\iint\rho_{t}\mathcal{L}f  f_{t} \, {\rm d}x  {\rm d}v + \int\rho_{t}u\cdot b_{t}\, {\rm d}x  \nonumber\\
&=-\iint\rho_{t} \Big(\nabla_{v}+\frac{v}{2}\Big)\{\mathbf{I}-\mathbf{P}\}f\Big(\nabla_{v}+\frac{v}{2}\Big)\{\mathbf{I}-\mathbf{P}\}f_{t}\, {\rm d}x  {\rm d}v\nonumber\\
&\quad-\iint\rho_{t} \Big(\nabla_{v}+\frac{v}{2}\Big)\{\mathbf{I}-\mathbf{P}\}f\Big(\nabla_{v}+\frac{v}{2}\Big)\mathbf{P}f_{t}\, {\rm d}x  {\rm d}v\nonumber\\
&\quad-\iint\rho_{t} \Big(\nabla_{v}+\frac{v}{2}\Big)\mathbf{P}f\Big(\nabla_{v}+\frac{v}{2}\Big)\{\mathbf{I}-\mathbf{P}\}f_{t}\, {\rm d}x  {\rm d}v\nonumber\\
&\quad-\iint\rho_{t} \Big(\nabla_{v}+\frac{v}{2}\Big)\mathbf{P}f\Big(\nabla_{v}+\frac{v}{2}\Big)\mathbf{P}f_{t}\, {\rm d}x  {\rm d}v
+\int\rho_{t}u\cdot b_{t}\, {\rm d}x  \nonumber\\
&:= \Pi_1+\Pi_2+\Pi_3+\Pi_4.\label{v3.555}
\end{align}
For the term $\Pi_1$, one has
\begin{align}
\Pi_1\,&\leq C\|\rho_{t}\|\int\sum_{|\alpha|\leq 2}\Big\|\Big(\nabla_{v}+\frac{v}{2}\Big)\partial^{\alpha}_{x}\{\mathbf{I}-\mathbf{P}\}f\Big\|_{L^{2}_{x}}
\Big\|\Big(\nabla_{v}+\frac{v}{2}\Big)\{\mathbf{I}-\mathbf{P}\}f_{t}\Big\|_{L^{2}_{x}} {\rm d}v\nonumber\\
&\leq C\|\rho_{t}\| \sum_{|\alpha|\leq 2}\|\partial^{\alpha}_{x}\{\mathbf{I}-\mathbf{P}\}f\|_{\nu}\|\{\mathbf{I}-\mathbf{P}\}f_{t}\|_{\nu}\nonumber\\
&\leq\epsilon \|\{\mathbf{I}-\mathbf{P}\}f_{t}\|^{2}_{\nu}+C_{\epsilon}\|\rho_{t}\|^{2} \sum_{|\alpha|\leq 2}
\|\partial^{\alpha}_{x}\{\mathbf{I}-\mathbf{P}\}f\|^{2}_{\nu}. \nonumber
\end{align}
By the definition of $\mathbf{P}$  and noticing the equality $\big(\nabla_{v}+\frac{v}{2}\big)g=b^{g}\sqrt{M}$ for any $g,$ one has
\begin{align}
\Pi_2\,&=\iint\rho_{t} \Big(\nabla_{v}+\frac{v}{2}\Big)\{\mathbf{I}-\mathbf{P}\}f\cdot b_{t}\sqrt{M}\, {\rm d}x  {\rm d}v\nonumber\\
& \leq C\|\nabla b_{t}\|_{H^{1}}\|\rho_{t}\|\|\{\mathbf{I}-\mathbf{P}\}f\|_{\nu}\nonumber\\
&\leq\epsilon \|\nabla b_{t}\|^{2}_{H^{1}}+C_{\epsilon}\|\{\mathbf{I}-\mathbf{P}\}f\|^{2}_{\nu}\|\rho_{t}\|^{2}\nonumber
\end{align}
for any $\epsilon >0.$
Similarly, for the terms $\Pi_3$ and $\Pi_4$, we have
\begin{align}
\Pi_3\,& =\iint\rho_{t} \Big(\nabla_{v}+\frac{v}{2}\Big)\{\mathbf{I}-\mathbf{P}\}f\cdot b_{t}\sqrt{M}\, {\rm d}x  {\rm d}v\nonumber\\
& \leq C\|\nabla b_{t}\|_{H^{1}}\|\rho_{t}\|\|\{\mathbf{I}-\mathbf{P}\}f\|_{\nu}\nonumber\\
\,&\leq\epsilon \|\{\mathbf{I}-\mathbf{P}\}f_{t}\|^{2}_{\nu}+C_{\epsilon}\|\nabla b\|^{2}_{H^{1}}\|\rho_{t}\|^{2},\nonumber\\
\Pi_4\, & =\int \rho_{t}(u-b)\cdot b_{t}\, {\rm d}x  \nonumber\\
& \leq C\|\nabla b_{t}\|_{H^{1}}\|u-b\| \|\rho_{t}\|\nonumber\\
& \leq \epsilon \|\nabla b_{t}\|^{2}_{H^{1}}+C_{\epsilon}\|u-b\|^{2} \|\rho_{t}\|^{2} \nonumber
\end{align}
for any $\epsilon >0.$
 Plugging the above estimates  into \eqref{v3.555}, we can rewrite it as
 \begin{align}
&\iint\rho_{t}\mathcal{L}f  f_{t} \, {\rm d}x  {\rm d}v + \int\rho_{t}u\cdot b_{t}\, {\rm d}x
\leq\epsilon\big(\|\nabla b_{t}\|^{2}_{H^{1}}+\|\{\mathbf{I}-\mathbf{P}\}f_{t}\|^{2}_{\nu}\big)\nonumber\\
&\qquad+C_{\epsilon}\big(\sum_{|\alpha|\leq 2}\|\partial^{\alpha}\{\mathbf{I}-\mathbf{P}\}f\|^{2}_{\nu}+\|u-b\|^{2}+\|\nabla b\|^{2}_{H^{1}}\big)\|\rho_{t}\|^{2}\label{v3.6}
\end{align}
for any $\epsilon >0.$

For the third and fourth terms on the right hand side of \eqref{v3.5}, we have
\begin{align}
&\!\!\!\!\!\!\!\!\!\!\!\!\!\!\!\! -\iint\big(\rho_{t}u+(1+\rho)u_{t}\big)\cdot\Big(\nabla_{v}-\frac{v}{2}\Big)f  f_{t}\, {\rm d}x  {\rm d}v-\int au^{2}_{t}\, {\rm d}x \nonumber\\
&=\int\big(\rho_{t}u+(1+\rho)u_{t}\big)f \Big(\nabla_{v}+\frac{v}{2}\Big)\{\mathbf{I}-\mathbf{P}\}f_{t}\, {\rm d}x  {\rm d}v\nonumber\\
&\quad+\int\big(\rho_{t}u+(1+\rho)u_{t}\big)f \Big(\nabla_{v}+\frac{v}{2}\Big)\mathbf{P}f_{t}\, {\rm d}x  {\rm d}v-\int au^{2}_{t}\, {\rm d}x.\label{v3.77}
\end{align}
We have the following estimates for the right hand side terms in \eqref{v3.77}:
\begin{align}
&\quad \iint\big(\rho_{t}u+(1+\rho)u_{t}\big)f \Big(\nabla_{v}+\frac{v}{2}\Big)\{\mathbf{I}-\mathbf{P}\}f_{t}\, {\rm d}x  {\rm d}v\nonumber\\
&\leq C\int \big(\|uf\|_{L^{\infty}_{x}}\|\rho_{t}\|+\|f\|_{L^{\infty}_{x}}(1+\|\rho\|_{L^{\infty}})\|u_{t}\|\big)\|\Big(\nabla_{v}+\frac{v}{2}\Big)
\{\mathbf{I}-\mathbf{P}\}f_{t}\| {\rm d}v\nonumber\\
&\leq C(1+\|\rho\|_{H^{2}}+\|u\|_{H^{2}})\|(\rho_{t},u_{t})\|\|\nabla_{x}f\|_{L^{2}_{v}(H^{1}_{x})}\|\{\mathbf{I}-\mathbf{P}\}f_{t}\|_{\nu}\nonumber\\
&\leq\epsilon\|\{\mathbf{I}-\mathbf{P}\}f_{t}\|^{2}_{\nu}
+C_{\epsilon}(1+\|(\rho,u)\|^{2}_{H^{2}})\nonumber\\
&\quad\times\big(\sum_{|\alpha|\leq 2}\|\partial^{\alpha}_{x}\{\mathbf{I}-\mathbf{P}\}f\|^{2}
+\sum_{|\alpha|\leq 1}\|\nabla_{x}(a,b)\|^{2}\big)\|(\rho_{t},u_{t})\|^{2},\nonumber\\
&\quad \iint\big(\rho_{t}u+(1+\rho)u_{t}\big)f \Big(\nabla_{v}+\frac{v}{2}\Big)\mathbf{P}f_{t}\, {\rm d}x  {\rm d}v-\int au^{2}_{t}\, {\rm d}x \nonumber\\
&=\int\big(\rho_{t}u+\rho u_{t}\big)a b_{t}\, {\rm d}x  +\int au_{t}\cdot (b_{t}-u_{t})\, {\rm d}x \nonumber\\
&\leq \epsilon\|u_{t}-b_{t}\|^{2}+C_{\epsilon}\|\nabla a\|^{2}_{H^{1}}(\|f_{t}\|^{2}+\|u_{t}\|^{2})
+C_{\epsilon}\|(\nabla \rho,\nabla u)\|^{2}_{H^{1}}(\|\rho_{t}\|^{2}+\|u_{t}\|^{2}) \nonumber
\end{align}
for any $\epsilon >0.$ Plugging   the above estimates into \eqref{v3.77} gives
 \begin{align}
&\!\!\!\!\!\!\!\!\!\!\!\!\!\!\!\! -\iint\big(\rho_{t}u+(1+\rho)u_{t}\big)\Big(\nabla_{v}-\frac{v}{2}\Big)f \cdot f_{t}\, {\rm d}x  {\rm d}v-\int au^{2}_{t}\, {\rm d}x \nonumber\\
&\leq\epsilon\big(\|u_{t}-b_{t}\|^{2}+\|\{\mathbf{I}-\mathbf{P}\}f_{t}\|^{2}_{\nu}\big)
+C_{\epsilon}(1+\|\rho\|^{2}_{H^{2}}+\|u\|^{2}_{H^{2}})\nonumber\\
&\quad\times\Big\{\sum_{|\alpha|\leq 2}\|\partial^{\alpha}\{\mathbf{I}-\mathbf{P}\}f\|^{2}
+\|\nabla (a,b,\rho,u)\|^{2}_{H^{1}}\Big\}\big(\|f_{t}\|^{2}+\|(\rho_{t},u_{t})\|^{2}\big).\label{v3.7}
\end{align}

The remaining  terms on right hand side of \eqref{v3.5}  can be treated easily. Hence we only list the bounds below and omit the details for brevity.
We have \begin{align}
&\iint\frac{1}{2}(1+\rho)u\cdot v f^{2}_{t}\, {\rm d}x  {\rm d}v-\int a_{t}u\cdot u_{t}\, {\rm d}x \nonumber\\
&\qquad\qquad\leq \epsilon \big(\|u_{t}-b_{t}\|^{2}+\|\{\mathbf{I}-\mathbf{P}\}f_{t}\|^{2}\big)
+C_{\epsilon}\big\{(1+\|\rho\|^{2}_{H^{2}})\|\nabla u\|^{2}_{H^{1}}\nonumber\\
&\qquad\qquad\quad+\|\nabla \rho\|_{H^{1}}\|\nabla u\|_{H^{1}}\big\}\|f_{t}\|^{2}
+ C(1+\|\rho\|_{H^{2}})\|u\|_{H^{2}}\|\{\mathbf{I}-\mathbf{P}\}f_{t}\|^{2}_{\nu},\label{v3.8}\\
&\int\rho (u_{t}-b_{t})\cdot b_{t} \, {\rm d}x  \leq\epsilon\|u_{t}-b_{t}\|^{2}+C_{\epsilon}\|\nabla \rho\|^{2}_{H^{1}}\|f_{t}\|^{2},\label{v3.9}\\
&\int(1+\rho)\ \dv u_{t}\rho_{t}\, {\rm d}x  +\int\frac{p'(1+\rho)}{1+\rho}u_{t}\cdot \nabla \rho_{t}\, {\rm d}x \nonumber\\
&\qquad\qquad\leq \epsilon \|\nabla u_{t}\|^{2}+C_{\epsilon} \|\nabla\rho\|^{2}_{H^{1}}\|\rho_{t}\|^{2},\label{v3.10}\\
&\int u\cdot \nabla\rho_{t} \rho_{t}\, {\rm d}x +\int \rho^{2}_{t}\ \dv u\, {\rm d}x \leq \epsilon \|\nabla \rho_{t}\|^{2}
+C_{\epsilon} \|\nabla u\|^{2}_{H^{1}}\|\rho_{t}\|^{2},\label{v3.11}\\
&\int\rho_{t}u_{t}\cdot \nabla\rho \, {\rm d}x  -\gamma(\gamma-2)\int(1+\rho)^{\gamma-3}\rho_{t}u_{t}\cdot \nabla \rho \, {\rm d}x \nonumber\\
&\qquad\qquad\leq\epsilon \|\nabla u_{t}\|^{2}+C_{\epsilon} \|\nabla\rho\|^{2}_{H^{1}}\|\rho_{t}\|^{2},\label{v3.12}\\
&\int u_{t}\cdot \nabla u\cdot u_{t}\, {\rm d}x +\int u\cdot\nabla u_{t}\cdot u_{t}\, {\rm d}x
\leq\epsilon \|\nabla u_{t}\|^{2}+C_{\epsilon} \|\nabla u\|^{2}_{H^{1}}\|u_{t}\|^{2},\label{v3.13}\\
&\int\frac{1}{(1+\rho)^{2}}\nabla \rho\cdot \nabla u_{t}\cdot u_{t}\, {\rm d}x -\int\frac{1}{(1+\rho)^{2}}\rho_{t}\Delta u_{t}\cdot u_{t}\, {\rm d}x \nonumber\\
&\qquad\qquad\leq\epsilon \|\nabla u_{t}\|^{2}+C_{\epsilon}\|(\nabla \rho,\nabla u)\|^{2}_{H^{2}}\|(\rho_{t},u_{t})\|^{2}.\label{v3.14}
\end{align}
Plugging all the above estimates \eqref{v3.6}, \eqref{v3.7}, and \eqref{v3.8}-\eqref{v3.14} into \eqref{v3.5} and then choosing $\epsilon$ sufficiently small yield   \eqref{v3.4}.
\end{proof}
\begin{Proposition}\label{vl3.3}
Assume that $(f,\rho,u)$ is the solution obtained in Theorem \ref{vt1.1}. Then, we have
\begin{align}
&\quad \frac{1}{2}\frac{\rm d}{{\rm d}t}\sum_{1\leq|\alpha|\leq3}\bigg\{\|\partial^{\alpha}f_{t}\|^{2}+\Big\|\frac{\sqrt{p'(1+\rho)}}{1+\rho}
\partial^{\alpha}\rho_{t}\Big\|^{2}+\|\partial^{\alpha}u_{t}\|^{2}\bigg\}\nonumber\\
&\quad+\lambda\sum_{1\leq|\alpha|\leq3}\big\{\|\partial^{\alpha}\{\mathbf{I}-\mathbf{P}\}f_{t}\|^{2}_{\nu}+
\|\partial^{\alpha}(u_{t}-b_{t})\|^{2}+\|\nabla\partial^{\alpha}u_{t}\|^{2}\big\}\nonumber\\
&\leq \epsilon \big(\|\nabla a_{t}\|^{2}_{H^{1}}+\|\nabla b_{t}\|^{2}_{H^{2}}\big)
+C_{\epsilon}\big(1+\|\rho\|^{^{2}}_{H^{3}}\big)\|\nabla u\|^{2}_{H^{2}}\|f_{t}\|^{2}_{L^{2}_{v}(H^{3}_{x})}\nonumber\\
&\quad+C_{\epsilon}\bigg\{\sum_{|\alpha|\leq 3}\|\partial^{\alpha}\{\mathbf{I}-\mathbf{P}\}f\|^{2}_{\nu}
+\|\nabla_{x}(a,b,\rho,u)\|^{2}_{H^{2}}\bigg\}\|(\rho_{t},u_{t})\|^{2}_{H^{3}}\nonumber\\
&\quad+C\|\rho\|_{H^{3}}\|\nabla (a_{t},b_{t})\|^{2}_{H^{2}}+C\big(1+\|\rho\|_{H^{3}}\big)\|u\|_{H^{4}}\|\nabla \rho_{t}\|^{2}_{H^{2}}\nonumber\\
&\quad+C\big\{(1+\|\rho\|_{H^{2}})\|u\|_{H^{^{2}}}+\|\rho\|_{H^{3}}\big\}\sum_{1\leq|\alpha|\leq3}\|\partial^{\alpha}\{\mathbf{I}-\mathbf{P}\}f_{t}\|^{2}_{\nu}\nonumber\\
&\quad+C\bigg\{\sum_{|\alpha|\leq 3}\|\partial^{\alpha}\{\mathbf{I}-\mathbf{P}\}f\|^{2}_{\nu}
+\|\nabla (a,b,\rho)\|^{2}_{H^{2}}+(1+\|\rho\|^{2}_{H^{3}})\|\nabla u\|^{2}_{H^{2}}\bigg\}\|f_{t}\|^{2}_{L^{2}_{v}(H^{3}_{x})}\nonumber\\
&\quad+C\bigg\{\sum_{|\alpha|\leq 3}\|\partial^{\alpha}\{\mathbf{I}-\mathbf{P}\}f\|^{2}_{\nu}+\|\nabla(a,b,\rho,u)\|^{2}_{H^{3}}
+\|\nabla u\|^{2}_{H^{4}}\nonumber\\
&\quad+\|\nabla \rho\|^{3}_{H^{3}}+\|\rho\|^{4}_{H^{3}}\|\nabla \rho\|^{2}_{H^{2}}\bigg\}\|(\rho_{t},u_{t})\|^{2}_{H^{3}}\label{v3.15}
\end{align}
for  $\epsilon >0$ sufficiently small.
\end{Proposition}
\begin{proof}
Applying  $\partial ^\alpha (1\leq|\alpha|\leq 3)$ to  \eqref{v3.1}-\eqref{v3.3},
multiplying the results by $\partial^{\alpha}f_{t}$,  $\frac{p'(1+\rho)}{(1+\rho)^{2}}\partial^{\alpha}\rho_{t}$, and $\partial^{\alpha}u_{t}$  respectively, and then adding them together, we obtain 
\begin{align}
&  \frac{1}{2}\frac{\rm d}{{\rm d}t}\Big(\|\partial^{\alpha}f_{t}\|^{2}+\|\frac{\sqrt{p'(1+\rho)}}{1+\rho}
\partial^{\alpha}\rho_{t}\|^{2}+\|\partial^{\alpha}u_{t}\|^{2}\Big)\nonumber\\
& +\int\langle -\mathcal{L}\{\mathbf{I}-\mathbf{P}\}\partial^{\alpha}f_{t},\partial^{\alpha}f_{t}\rangle   \, {\rm d}x +
\|\partial^{\alpha}(u_{t}-b_{t})\|^{2}+\int\frac{1}{1+\rho}|\nabla\partial^{\alpha}u_{t}|^{2}\, {\rm d}x \nonumber\\
=\,&\iint\partial^{\alpha}\Big((\rho_{t}u+(1+\rho)u_{t})\frac{v}{2}f\Big)\partial^{\alpha}f_{t}\, {\rm d}x  {\rm d}v
-\iint\partial^{\alpha}\big((\rho_{t}u+(1+\rho)u_{t})\cdot\nabla_{v}f\big)\partial^{\alpha}f_{t}\, {\rm d}x  {\rm d}v
\nonumber\\
&+\iint\partial^{\alpha}\Big((1+\rho)u\cdot \frac{v}{2}f_{t}\Big)\partial^{\alpha}f_{t}\, {\rm d}x  {\rm d}v
-\int\partial^{\alpha}(ua_{t})\cdot\partial^{\alpha}u_{t}\, {\rm d}x \nonumber\\
&-\int[-\partial^{\alpha},(1+\rho)u\cdot \nabla_{v}]f_{t}\partial^{\alpha}f_{t}\, {\rm d}x
+\int\partial^{\alpha}(\rho_{t}\mathcal{L}f )\partial^{\alpha}f_{t}\, {\rm d}x
\nonumber\\
&-\int\partial^{\alpha}(\rho \mathcal{L}f _{t})\partial^{\alpha}f_{t}\, {\rm d}x
-\int\frac{p'(1+\rho)}{(1+\rho)^{2}}u\cdot \nabla \partial^{\alpha}\rho_{t}\partial^{\alpha}\rho_{t}\, {\rm d}x
\nonumber\\
&+\int\partial^{\alpha}\big(\rho_{t}u+\rho u_{t}\big)\cdot\partial^{\alpha}b_{t}\, {\rm d}x
+\int \partial^{\alpha}u_{t}\cdot\nabla\frac{p'(1+\rho)}{1+\rho}\partial^{\alpha}\rho_{t}\, {\rm d}x
\nonumber\\
&+\int[-\partial^{\alpha},\rho\nabla_{x}\cdot]u_{t}\frac{p'(1+\rho)}{(1+\rho)^{2}}\partial^{\alpha}\rho_{t}\, {\rm d}x
+\int[-\partial^{\alpha},u\cdot\nabla_{x}]\rho_{t}\frac{p'(1+\rho)}{(1+\rho)^{2}}\partial^{\alpha}\rho_{t}\, {\rm d}x \nonumber\\
&+\frac{1}{2}\int\partial_{t}\big(\frac{p'(1+\rho)}{(1+\rho)^{2}}\big)|\partial^{\alpha}\rho_{t}|^{2}\, {\rm d}x
-\int\partial^{\alpha}\big(\rho_{t}\ \dv u+\nabla\rho\cdot u_{t}\big)\frac{p'(1+\rho)}{(1+\rho)^{2}}\partial^{\alpha}\rho_{t}\, {\rm d}x \nonumber\\
&-\int u\cdot \nabla \partial^{\alpha}u_{t}\cdot \partial^{\alpha}u_{t}\, {\rm d}x
-\int\nabla\frac{1}{1+\rho}\nabla\partial^{\alpha}u_{t}\cdot \partial^{\alpha}u_{t}\, {\rm d}x
-\int\partial^{\alpha}(u_{t}\cdot \nabla u)\cdot \partial^{\alpha}u_{t}\, {\rm d}x \nonumber\\
&-\sum_{1\leq\gamma\leq\alpha}C_{\alpha,\gamma}\int\partial^{\gamma}(\frac{1}{1+\rho})
\partial^{\alpha-\gamma}\Delta u_{t}\cdot\partial^{\alpha}u_{t}\, {\rm d}x
-\int\partial^{\alpha}\big(\frac{1}{(1+\rho)^{2}}\rho_{t}\Delta u\big)\cdot\partial^{\alpha}u_{t}\, {\rm d}x \nonumber\\
&+\int[-\partial^{\alpha},u\cdot\nabla]u_{t}\cdot\partial^{\alpha}u_{t}\, {\rm d}x
+\int\Big[-\partial^{\alpha},\frac{p'(1+\rho)}{1+\rho}\nabla\Big]\rho_{t}\cdot\partial^{\alpha}u_{t}\, {\rm d}x \nonumber\\
&-\gamma(\gamma-2)\int\partial^{\alpha}\big((1+\rho)^{\gamma-3}\rho_{t}\nabla\rho\big)\cdot\partial^{\alpha}u_{t}\, {\rm d}x
-\int\partial^{\alpha}(au_{t})\cdot\partial^{\alpha}u_{t}\, {\rm d}x . \label{v3.16}
\end{align}
Now we estiamte the terms on the right hand side of \eqref{v3.16}. First, we have
\begin{align}
&\iint\partial^{\alpha}\big((\rho_{t}u+\rho u_{t})\cdot \nabla_{v}f\big)\partial^{\alpha}f_{t}\, {\rm d}x  {\rm d}v
+\int\partial^{\alpha}\Big((\rho_{t}u+\rho u_{t})\cdot \frac{v}{2}f\Big)\partial^{\alpha}f_{t}\, {\rm d}x  {\rm d}v\nonumber\\
&\leq \int \big(\|\nabla u\|_{H^{2}}\|\rho_{t}\|_{H^{3}}+\|\nabla \rho\|_{H^{2}}\|u_{t}\|_{H^{3}}\big)
\Big\|\Big(\nabla_{v}+\frac{v}{2}\Big)\nabla_{x}f\Big\|_{H^{2}_{x}}\|\partial^{\alpha}f_{t}\|_{L^{2}_{x}} {\rm d}v\nonumber\\
&\leq C\big(\|\nabla u\|_{H^{2}}\|\rho_{t}\|_{H^{3}}+\|\nabla \rho\|_{H^{2}}\|u_{t}\|_{H^{3}}\big)
\Big\|\Big(\nabla_{v}+\frac{v}{2}\Big)\nabla_{x}f\Big\|_{L^{2}(H^{2}_{x})}\|\partial^{\alpha}f_{t}\|\nonumber\\
&\leq C \Big\{\sum_{|\alpha|\leq 3}\|\partial^{\alpha}\{\mathbf{I}-\mathbf{P}\}f\|^{2}_{\nu}+\|\nabla (a,b)\|^{2}_{H^{2}}\Big\}\|\partial^{\alpha}f_{t}\|^{2}
+C\|(\nabla \rho,\nabla u)\|^{2}_{H^{2}}\|(\rho_{t},u_{t})\|^{2}_{H^{3}},\nonumber\\
&\iint\partial^{\alpha}(u_{t}\Big(\nabla_{v}+\frac{v}{2}\Big)f)\partial^{\alpha}f_{t}\, {\rm d}x  {\rm d}v\nonumber\\
&\leq C\|\nabla u_{t}\|_{H^{2}}\Big\|\Big(\nabla_{v}+\frac{v}{2}\Big)\nabla_{x}f\Big\|_{L^{2}(H^{2}_{x})}\|\partial^{\alpha}f_{t}\|\nonumber\\
&\leq \eta\|\nabla u_{t}\|^{2}_{H^{2}}
+C_{\eta}\Big\{\sum_{|\alpha|\leq 3}\|\partial^{\alpha}\{\mathbf{I}-\mathbf{P}\}f\|^{2}_{\nu}+\|\nabla (a,b)\|^{2}_{H^{2}}\Big\}\|\partial^{\alpha}f_{t}\|^{2}\nonumber
\end{align}
with $\eta$ small enough. Thus the first two terms on the right hand side can be bounded by
\begin{align}
\Big\{\sum_{|\alpha|\leq 3}&\|\partial^{\alpha}\{\mathbf{I}-\mathbf{P}\}f\|^{2}_{\nu}+\|\nabla (a,b)\|^{2}_{H^{2}}\Big\}\|\partial^{\alpha}f_{t}\|^{2}
+C\|(\nabla \rho,\nabla u)\|^{2}_{H^{2}}\|(\rho_{t},u_{t})\|^{2}_{H^{3}}\nonumber\\
&+\eta\|\nabla u_{t}\|^{2}_{H^{2}}
+C_{\eta}\Big\{\sum_{|\alpha|\leq 3}\|\partial^{\alpha}\{\mathbf{I}-\mathbf{P}\}f\|^{2}_{\nu}+\|\nabla (a,b)\|^{2}_{H^{2}}\Big\}
\|\partial^{\alpha}f_{t}\|^{2}.\label{v3.17}
\end{align}

For the third and forth terms on the right hand side of \eqref{v3.16}, we can rewrite them as
\begin{align}
&\iint\partial^{\alpha}\Big((1+\rho)u\cdot \frac{v}{2}f_{t}\Big)\partial^{\alpha}f_{t}\, {\rm d}x  {\rm d}v
-\int\partial^{\alpha}(ua_{t})\cdot\partial^{\alpha}u_{t}\, {\rm d}x \nonumber\\
=\,&\iint(1+\rho)u\cdot \frac{v}{2}\partial^{\alpha}f_{t}\partial^{\alpha}f_{t}\, {\rm d}x  {\rm d}v-\int \partial^{\alpha}a_{t}u\cdot\partial^{\alpha} u_{t}\, {\rm d}x \nonumber\\
&+\sum_{1\leq\gamma\leq\alpha}C_{\alpha,\gamma}\Big(\iint\partial^{\gamma}\big((1+\rho)u\big)\cdot \frac{v}{2}\partial^{\alpha-\gamma}f_{t}\partial^{\alpha}f_{t}\, {\rm d}x  {\rm d}v
-\int\partial^{\alpha-\gamma}a_{t}\partial^{\gamma} u\cdot\partial^{\alpha}u_{t}\, {\rm d}x \Big)\nonumber\\
=\,&\frac{1}{2}\int(1+\rho)u\langle v, |\{\mathbf{I}-\mathbf{P}\}\partial^{\alpha}f_{t}|^{2}\rangle   \, {\rm d}x
+\int \partial^{\alpha}a_{t}u\cdot\partial^{\alpha}(b_{t}-u_{t})\, {\rm d}x \nonumber\\
&+\int (1+\rho)u\cdot \langle v\mathbf{P}\partial^{\alpha}f_{t},\{\mathbf{I}-\mathbf{P}\}\partial^{\alpha}f_{t}\rangle   \, {\rm d}x
 +\int\rho \partial^{\alpha}a_{t}u\cdot \partial^{\alpha}b_{t}\, {\rm d}x \nonumber\\
& +\sum_{1\leq\gamma\leq\alpha}C_{\alpha,\gamma}\Big(\iint\partial^{\gamma}\big((1+\rho)u\big)\cdot \frac{v}{2}\partial^{\alpha-\gamma}f_{t}\partial^{\alpha}f_{t}\, {\rm d}x  {\rm d}v
-\int\partial^{\alpha-\gamma}a_{t}\partial^{\gamma} u\cdot\partial^{\alpha}u_{t}\, {\rm d}x\Big).\nonumber
\end{align}
For the terms on the right hand side of the above equality, we have
\begin{align}
&\frac{1}{2}\int(1+\rho)u\cdot\langle v, |\{\mathbf{I}-\mathbf{P}\}\partial^{\alpha}f_{t}|^{2}\rangle   \, {\rm d}x
\leq C(1+\|\nabla \rho\|_{H^{1}})\|\nabla u\|_{H^{1}}\|\{\mathbf{I}-\mathbf{P}\}\partial^{\alpha}f_{t}\|^{2}_{\nu},\nonumber\\
&\int (1+\rho)u\cdot \langle v\mathbf{P}\partial^{\alpha}f_{t},\{\mathbf{I}-\mathbf{P}\}\partial^{\alpha}f_{t}\rangle   \, {\rm d}x
\leq C\|(1+\rho)u\|_{L^{\infty}}\|\partial^{\alpha}f_{t}\|\|\{\mathbf{I}-\mathbf{P}\}\partial^{\alpha}f_{t}\|\nonumber\\
&\qquad\qquad\qquad\qquad\qquad\leq \eta \|\{\mathbf{I}-\mathbf{P}\}\partial^{\alpha}f_{t}\|^{2}
+C_{\eta}(1+\|\rho\|^{2}_{H^{2}})\|\nabla u\|^{2}_{H^{1}}\|\partial^{\alpha}f_{t}\|^{2},\nonumber\\
&\int u\partial^{\alpha}a_{t}\cdot\partial^{\alpha}(b_{t}-u_{t})\, {\rm d}x \leq C\|\nabla u\|_{H^{1}}\|\partial^{\alpha}a_{t}\|\|\partial^{\alpha}(u_{t}-b_{t})\|\nonumber\\
&\qquad\qquad\qquad\qquad\qquad\leq \eta\|\partial^{\alpha}(b_{t}-u_{t})\|^{2}+C_{\eta}\|\nabla u\|^{2}_{H^{1}}\|f_{t}\|^{2}_{L^{2}_{v}(H^{3}_{x})},\nonumber\\
&\int\rho \partial^{\alpha}a_{t}u\cdot \partial^{\alpha}b_{t}\, {\rm d}x \leq C\|\rho u\|_{L^{\infty}}\|\partial^{\alpha}a_{t}\|\|\partial^{\alpha}b_{t}\|
\leq C(\|\nabla \rho\|^{2}_{H^{1}}+\|\nabla u\|^{2}_{H^{1}})\|f_{t}\|^{2}_{L^{2}_{v}(H^{3}_{x})},\nonumber\\
&\int\partial^{\gamma} u\partial^{\alpha-\gamma}a_{t}\cdot\partial^{\alpha}u_{t}\, {\rm d}x \leq C \|\nabla u\|_{H^{2}}\|\nabla a_{t}\|_{H^{1}}\|u_{t}\|_{H^{3}}\nonumber\\
&\qquad\qquad\qquad\qquad\qquad\leq\epsilon \|\nabla a_{t}\|^{2}_{H^{1}}+C_{\epsilon}\|\nabla u\|^{2}_{H^{2}}\|u_{t}\|^{2}_{H^{3}},\nonumber\\
&\iint\partial^{\gamma}\big((1+\rho)u\big)\cdot \frac{v}{2}\partial^{\alpha-\gamma}f_{t}\partial^{\alpha}f_{t}\, {\rm d}x  {\rm d}v\nonumber\\
&=\int\partial^{\gamma}\big((1+\rho)u\big)\cdot \frac{v}{2}\Big(\partial^{\alpha-\gamma}\{\mathbf{I}-\mathbf{P}\}f_{t}
+\partial^{\alpha-\gamma}\mathbf{P}f_{t}\Big)\partial^{\alpha}f_{t}\, {\rm d}x  {\rm d}v\nonumber\\
&\leq C\big(1+\|\rho\|_{H^{3}}\big)\|\nabla u\|_{H^{2}}\Big\{\sum_{|\alpha|\leq 2}\|\partial^{\alpha}\{\mathbf{I}-\mathbf{P}\}f_{t}\|_{\nu}
+\|\nabla(a_{t},b_{t})\|_{H^{1}}\Big\}\|\partial^{\alpha}f_{t}\|\nonumber\\
&\leq \eta\sum_{|\alpha|\leq 2}\|\partial^{\alpha}\{\mathbf{I}-\mathbf{P}\}f_{t}\|_{\nu}+C_{\eta}(1+\|\rho\|^{2}_{H^{3}})
\|\nabla u\|^{2}_{H^{2}}\|f_{t}\|^{2}_{L^{2}_{v}(H^{3}_{x})}\nonumber\\
&\quad+\epsilon\|\nabla (a_{t},b_{t})\|^{2}_{H^{1}}+C_{\epsilon}(1+\|\rho\|^{2}_{H^{3}})
\|\nabla u\|^{2}_{H^{2}}\|f_{t}\|^{2}_{L^{2}_{v}(H^{3}_{x})},\nonumber
\end{align}
where $\epsilon,\eta$ are sufficiently small constant.
 Thus the third and forth terms have the following bound:
\begin{align}
&\iint\partial^{\alpha}\big((1+\rho)u\cdot \frac{v}{2}f_{t}\big)\partial^{\alpha}f_{t}\, {\rm d}x  {\rm d}v
-\int\partial^{\alpha}(ua_{t})\cdot\partial^{\alpha}u_{t}\, {\rm d}x \nonumber\\
&\leq C(1+\|\nabla \rho\|_{H^{1}})\|\nabla u\|_{H^{1}}\|\{\mathbf{I}-\mathbf{P}\}\partial^{\alpha}f_{t}\|^{2}_{\nu}
+C\|\nabla (\rho,u)\|^{2}_{H^{1}}\|f_{t}\|^{2}_{L^{2}_{v}(H^{3}_{x})}\nonumber\\
&+\eta\Big\{\sum_{|\alpha|\leq 2}\|\partial^{\alpha}\{\mathbf{I}-\mathbf{P}\}f_{t}\|_{\nu}+\|\partial^{\alpha}(u_{t}-b_{t})\|^{2}\Big\}
+C_{\eta}(1+\|\rho\|^{2}_{H^{3}})
\|\nabla u\|^{2}_{H^{2}}\|f_{t}\|^{2}_{L^{2}_{v}(H^{3}_{x})}\nonumber\\
&+\epsilon\|\nabla (a_{t},b_{t})\|^{2}_{H^{1}}+C_{\epsilon}(1+\|\rho\|^{2}_{H^{3}})
\|\nabla u\|^{2}_{H^{2}}\|f_{t}\|^{2}_{L^{2}_{v}(H^{3}_{x})}+C_{\epsilon}\|\nabla u\|^{2}_{H^{2}}\|u_{t}\|^{2}_{H^{3}}.\label{v3.18}
\end{align}
Similarly, we have
\begin{align}
&\!\!\!\!\!\!\!\!\!\!\!\!\!\!\!\!\iint[\partial^{\alpha},(1+\rho)u\cdot \nabla_{v}]f_{t}\partial^{\alpha}f_{t}\, {\rm d}x  {\rm d}v\nonumber\\
\leq &\eta\sum_{|\alpha|\leq 2}\|\partial^{\alpha}\{\mathbf{I}-\mathbf{P}\}f_{t}\|_{\nu}+C_{\eta}(1+\|\rho\|^{2}_{H^{3}})
\|\nabla u\|^{2}_{H^{2}}\|f_{t}\|^{2}_{L^{2}_{v}(H^{3}_{x})}\nonumber\\
&+\epsilon\|\nabla (a_{t},b_{t})\|^{2}_{H^{1}}+C_{\epsilon}(1+\|\rho\|^{2}_{H^{3}})
\|\nabla u\|^{2}_{H^{2}}\|f_{t}\|^{2}_{L^{2}_{v}(H^{3}_{x})}.\label{v3.19}
\end{align}

Next, we deal with $\iint\partial^{\alpha}(\rho_{t}\mathcal{L}f )\partial^{\alpha}f_{t}\, {\rm d}x  {\rm d}v$. We have
\begin{align}
\iint\partial^{\alpha}&(\rho_{t}\mathcal{L}f )\partial^{\alpha}_{x}f_{t}\, {\rm d}x  {\rm d}v\nonumber\\
&=-\iint\partial^{\alpha}\Big(\rho_{t}\big(\nabla_{v}+\frac{v}{2}\big)f\Big)\cdot\partial^{\alpha}b_{t}\sqrt{M}\, {\rm d}x  {\rm d}v\nonumber\\
&\quad-\iint\partial^{\alpha}\Big(\rho_{t}\big(\nabla_{v}+\frac{v}{2}\big)f\Big)\cdot\Big(\nabla_{v}+\frac{v}{2}\Big)\partial^{\alpha}_{x}\{\mathbf{I}-\mathbf{P}\}f_{t}\, {\rm d}x  {\rm d}v
\nonumber\\
&\leq C\|\nabla \rho_{t}\|_{H^{2}}\Big\|\Big(\nabla_{v}+\frac{v}{2}\Big)\nabla_{x}f\Big\|_{L^{2}_{v}(H^{2}_{x})}
\Big(\|\partial^{\alpha}\{\mathbf{I}-\mathbf{P}\}f_{t}\|_{\nu}+\|\partial^{\alpha}b_{t}\|\Big)\nonumber\\
&\leq \eta\|\partial^{\alpha}\{\mathbf{I}-\mathbf{P}\}f_{t}\|^{2}_{\nu}+
C_{\eta}\Big\{\sum_{|\alpha|\leq 3}\|\partial^{\alpha}\{\mathbf{I}-\mathbf{P}\}f\|^{2}_{\nu}+\|\nabla (a,b)\|^{2}_{H^{2}}\Big\}\|\rho_{t}\|^{2}_{H^{3}}\nonumber\\
&\quad+\epsilon\|\nabla b_{t}\|^{2}_{H^{2}}+C_{\epsilon}\Big\{\sum_{|\alpha|\leq 3}\|\partial^{\alpha}\{\mathbf{I}-\mathbf{P}\}f\|^{2}_{\nu}
+\|\nabla (a,b)\|^{2}_{H^{2}}\Big\}\|\rho_{t}\|^{2}_{H^{3}}.\label{v3.20}
\end{align}

Similarly, we have
\begin{align}
-\iint\partial^{\alpha}(\rho \mathcal{L}f _{t})&\partial^{\alpha}_{x}f_{t}\, {\rm d}x  {\rm d}v\nonumber\\
=&\iint \partial^{\alpha}\Big(\rho\big(\nabla_{v}+\frac{v}{2}\big)f_{t}\Big)\cdot\Big(\nabla_{v}+\frac{v}{2}\Big)\partial^{\alpha}_{x}f_{t}\, {\rm d}x  {\rm d}v\nonumber\\
\leq\,& C \|\nabla \rho\|_{H^{^{2}}}\Big\{\sum_{|\alpha|\leq 3}\|\partial^{\alpha}\{\mathbf{I}-\mathbf{P}\}f_{t}\|^{2}_{\nu}+\|\nabla (a_{t},b_{t})\|^{2}_{H^{2}}\Big\}.\label{v3.21}
\end{align}

For the  sake of brevity, we only give the bound of the remaining terms on the right hand side of \eqref{v3.16} as follows:
\begin{align}
&\int\partial^{\alpha}(\rho_{t}u+\rho u_{t})\cdot\partial^{\alpha}b_{t}\, {\rm d}x  \leq \epsilon\|\nabla b_{t}\|^{2}_{H^{2}}
+C_{\epsilon}\|\nabla (\rho,u)\|^{2}_{H^{2}}\|(\rho_{t},u_{t})\|^{2}_{H^{3}},\label{v3.22}\\
&\int\partial^{\alpha} u_{t}\cdot\nabla\frac{p'(1+\rho)}{1+\rho}\partial^{\alpha}\rho_{t}\, {\rm d}x
+\int[\partial^{\alpha},\rho\nabla\cdot]u_{t}\frac{p'(1+\rho)}{(1+\rho)^{2}}\partial^{\alpha}\rho_{t}\, {\rm d}x \nonumber\\
&\qquad\qquad\leq\eta\|\nabla u_{t}\|^{2}_{H^{^{2}}}+  C_{\eta} \|\nabla \rho\|^{2}_{H^{^{2}}}\|\rho_{t}\|^{2}_{H^{^{3}}},\label{v3.23}\\
&\int\frac{p'(1+\rho)}{(1+\rho)^{2}}u\cdot\nabla\partial^{\alpha}\rho_{t}\partial^{\alpha}\rho_{t}\, {\rm d}x \nonumber\\
&\qquad\qquad\leq  C\|\nabla \rho\|_{H^{2}}\|\nabla u\|_{H^{1}}\|\partial^{\alpha}\rho_{t}\|^{2}+C\| u\|_{H^{3}}\| \rho_{t}\|^{2}_{H^{3}},\label{v3.24}\\
&\int[-\partial^{\alpha},u\cdot \nabla]\rho_{t}\frac{p'(1+\rho)}{(1+\rho)^{2}}\partial^{\alpha}\rho_{t}\, {\rm d}x
-\int\partial^{\alpha}\big(\rho_{t}\ \dv u+u_{t}\cdot \nabla \rho\big)\frac{p'(1+\rho)}{(1+\rho)^{2}}\partial^{\alpha}\rho_{t}\, {\rm d}x \nonumber\\
&\qquad\qquad\leq \eta\|\nabla u_{t}\|^{2}_{H^{2}}+C_{\eta}\|\nabla\rho\|^{2}_{H^{3}}\|\nabla \rho_{t}\|^{2}_{H^{2}}+C\|u\|_{H^{4}}\|\nabla \rho_{t}\|^{2}_{H^{2}},\label{v3.25}\\
&\int u\cdot \nabla \partial^{\alpha}u_{t}\cdot \partial^{\alpha}u_{t}\, {\rm d}x +\int\nabla\frac{1}{1+\rho}\cdot\nabla\partial^{\alpha}u_{t}
\cdot\partial^{\alpha}u_{t}\, {\rm d}x \nonumber\\
&\qquad\qquad\quad+\sum_{1\leq\gamma \leq\alpha}C_{\alpha,\gamma}\int\partial^{\gamma}\Big(\frac{1}{1+\rho}\Big)
\partial^{\alpha-\gamma}\Delta u_{t}\partial^{\alpha}u_{t}\, {\rm d}x \nonumber\\
&\qquad\qquad\leq \eta\|\nabla \partial^{\alpha}u_{t}\|^{2} +C_{\eta} \big(\|\nabla(\rho,u)\|^{2}_{H^{2}}+\|\nabla\rho\|^{3}_{H^{2}}\big)\|u_{t}\|^{2}_{H^{3}},\label{v3.26}\\
&\int-\partial^{\alpha}(u_{t}\cdot \nabla u)\partial^{\alpha}\, {\rm d}x +\int[-\partial^{\alpha},u\cdot\nabla]u_{t}\partial^{\alpha}u_{t}\, {\rm d}x\nonumber\\
&\qquad\qquad\quad +\int\Big[-\partial^{\alpha},\frac{p'(1+\rho)}{1+\rho}\nabla\Big]\rho_{t}\partial^{\alpha}u_{t}\, {\rm d}x \nonumber\\
&\qquad\qquad\leq\eta\|\nabla u_{t}\|^{2}_{H^{2}}+C_{\eta}\big(\|(\nabla\rho,\nabla u)\|^{2}_{H^{3}}+\|\nabla \rho\|^{3}_{H^{2}}\big)\|(\rho_{t},u_{t})\|^{2}_{H^{3}},\label{v3.27}\\
&\int\partial_{t}\Big(\frac{p'(1+\rho)}{(1+\rho)^{2}}\Big)|\partial^{\alpha}\rho_{t}|^{2}\, {\rm d}x
\leq C(1+\|\rho\|_{H^{3}})\|u\|_{H^{3}}\|\nabla \rho_{t}\|^{2}_{H^{2}},\label{v3.28}\\
&\int\partial^{\alpha}\Big(\frac{1}{(1+\rho)^{2}}\rho_{t}\Delta u\Big)\cdot\partial^{\alpha}u_{t}\, {\rm d}x \nonumber\\
&\qquad\qquad\leq C(1+\|\rho\|^{4}_{H^{3}})\|\nabla \rho\|^{2}_{H^{2}}\|\rho_{t}\|^{2}_{H^{3}}+C\|\nabla u\|^{2}_{H^{4}}\|u_{t}\|^{2}_{H^{3}},\label{v3.29}\\
&\int\partial^{\alpha}\big((1+\rho)^{\gamma-3}\rho_{t}\nabla \rho\big)\cdot\partial^{\alpha}u_{t}\, {\rm d}x
\leq C(1+\|\rho\|^{2}_{H^{3}})\|\nabla \rho\|^{2}_{H^{3}}\|(\rho_{t},u_{t})\|^{2}_{H^{3}},\label{v3.30}\\
&\int\partial^{\alpha}(au_{t})\cdot\partial^{\alpha}u_{t}\, {\rm d}x \leq \eta\|\nabla u_{t}\|^{2}_{H^{2}}+C_{\eta}\|\nabla a\|^{2}_{H^{2}}\|u_{t}\|^{2}_{H^{3}}.\label{v3.31}
\end{align}
Finally, plugging the above estimates \eqref{v3.17}-\eqref{v3.31} into \eqref{v3.16}, using\eqref{v2.1}, and choosing $\eta$ sufficiently small,
we obtain \eqref{v3.15}.
\end{proof}

\begin{Remark}\label{vr3.1}
Combining Proposition\,\ref{vl3.2} and Proposition\,\ref{vl3.3},
for the solution  $(f,\rho,u)$   obtained in Theorem \ref{vt1.1},   we have
\begin{align}
&\frac{1}{2}\frac{\rm d}{{\rm d}t}\bigg\{\sum_{|\alpha|\leq 3}\|\partial^{\alpha}f_{t}\|^{2}+\sum_{|\alpha|\leq 3}\|\partial^{\alpha}u_{t}\|^{2}
+\|\rho_{t}\|^{2}+\sum_{1\leq |\alpha|\leq 3}\Big\|\frac{\sqrt{p'(1+\rho)}}{1+\rho}\partial^{\alpha}\rho_{t}\Big\|^{2}\bigg\}\nonumber\\
&+\lambda\sum_{|\alpha|\leq3}\big\{\|\partial^{\alpha}\{\mathbf{I}-\mathbf{P}\}f_{t}\|^{2}_{\nu}
+\|\nabla\partial^{\alpha}u_{t}\|^{2}+\|\partial^{\alpha}(u_{t}-b_{t})\|^{2}\big\}\nonumber\\
\leq\, &\epsilon \big\{\|\nabla a_{t}\|^{2}_{H^{1}}+\|\nabla b_{t}\|^{2}_{H^{2}}+\|\nabla \rho_{t}\|^{2}\big\}
+C_{\epsilon}\big(1+\|\rho\|^{^{2}}_{H^{3}}\big)\|\nabla u\|^{2}_{H^{2}}\|f_{t}\|^{2}_{L^{2}_{v}(H^{3}_{x})}\nonumber\\
&+C_{\epsilon}\bigg\{\sum_{|\alpha|\leq 3}\|\partial^{\alpha}\{\mathbf{I}-\mathbf{P}\}f\|^{2}_{\nu}
+\|\nabla_{x}(a,b,\rho,u)\|^{2}_{H^{2}}+\|u-b\|^{2}\bigg\}\|(\rho_{t},u_{t})\|^{2}_{H^{3}}\nonumber\\
&+C\|\rho\|_{H^{3}}\|\nabla (a_{t},b_{t})\|^{2}_{H^{2}}+C\big(1+\|\rho\|_{H^{3}}\big)\|u\|_{H^{4}}\|\nabla \rho_{t}\|^{2}_{H^{2}}\nonumber\\
&+C\big((1+\|\rho\|_{H^{2}})\|u\|_{H^{^{2}}}+\|\rho\|_{H^{3}}\big)\sum_{1\leq|\alpha|\leq3}\|\partial^{\alpha}\{\mathbf{I}-\mathbf{P}\}f_{t}\|^{2}_{\nu}\nonumber\\
&+C(1+\|\rho\|^{4}_{H^{4}}+\|u\|^{2}_{H^{2}})
\bigg\{\sum_{|\alpha|\leq 3}\|\partial^{\alpha}\{\mathbf{I}-\mathbf{P}\}f\|^{2}_{\nu}
+\|\nabla (a,b)\|^{2}_{H^{2}}\nonumber\\
&+\|\nabla \rho\|^{2}_{H^{3}}+\|\nabla u\|^{2}_{H^{4}}\bigg\}
\big(\|f_{t}\|^{2}_{L^{2}_{v}(H^{3}_{x})}+\|\rho_{t},u_{t}\|^{2}_{H^{3}}\big).\label{v3.32}
\end{align}
\end{Remark}
Next, we need to estimate $\|\nabla (a_{t},b_{t})\|^{2}_{H^{2}}$.
 From \eqref{v2.13}-\eqref{v2.15}, we deduce that $(a_{t},b_{t})$ satisfy the following system
\begin{align}
&\partial_{t}a_{t}+\dv b_{t}=0,\label{v3.33}\\
&\partial_{t}b_{t,i}+\partial_{i}a_{t}+\sum_{j}\partial_{j}\Gamma_{i,j}(\{\mathbf{I}-\mathbf{P}\}f_{t})\nonumber\\
&\qquad=-(1+\rho) b_{t,i}-\rho_{t}b_{i}+\rho_{t}u_{i}(1+a)+(1+\rho)u_{t,i}(1+a)+(1+\rho)u_{i}a_{t},\label{v3.34}\\
&\partial_{i}b_{t,j}+\partial_{j}b_{t,i}-(1+\rho)(u_{t,i}b_{j}+u_{t,j}b_{i}+u_{i}b_{t,j}+u_{j}b_{t,i})-\rho_{t}(u_{i}b_{j}+u_{j}b_{i})\nonumber\\
&\qquad=-\partial_{t}\Gamma(\{\mathbf{I}-\mathbf{P}\}f_{t})+\Gamma_{i,j}(l_{t}+r_{t}+s_{t}),\label{v3.35}
\end{align}
where
\begin{align}
&l_{t}:=-v\cdot \nabla_{x}\{\mathbf{I}-\mathbf{P}\}f_{t}+L\{\mathbf{I}-\mathbf{P}\}f_{t},\nonumber\\
&r_{t}:=\partial_{t}\Big(-u\cdot\nabla_{v}\{\mathbf{I}-\mathbf{P}\}f+\frac{1}{2}u\cdot v\{\mathbf{I}-\mathbf{P}\}f\Big),\nonumber\\
&s_{t}:=\partial_{t}\Big(\rho M^{-\frac{1}{2}}\nabla_{v}\cdot\Big(\frac{v}{2}\sqrt{M}\{\mathbf{I}-\mathbf{P}\}f
+\sqrt{M}\{\mathbf{I}-\mathbf{P}\}f-u\sqrt{M}\{\mathbf{I}-\mathbf{P}\}f\Big)\Big).\nonumber
\end{align}
Define the following functional
\begin{align}
\bar{\mathcal{E}}_{0}(f_{t}(t)):=\,&\sum_{|\alpha|\leq 2}\sum_{i,j}\int_{\mathbb{R}^{3}}\partial^{\alpha}_{x}
(\partial_{x_{j}}b_{t,i}+\partial_{_{x_{i}}}b_{t,j})\partial^{\alpha}_{x}\Gamma_{i,j}(\{\mathbf{I}-\mathbf{P}\}f_{t})\, {\rm d}x \nonumber\\
&-\sum_{|\alpha|\leq 2}\int_{\mathbb{R}^{3}}\partial^{\alpha}_{x}a_{t}\partial^{\alpha}_{x}\nabla_{x}\cdot b_{t} \, {\rm d}x.\label{v3.36}
\end{align}
Similarly to the proofs of Proposition \ref{vl2.4} and Proposition \ref{vl2.5}, we have the following two results:
\begin{Proposition}\label{vl3.4}
Assume that $(f,\rho,u)$ is the solution obtained in Theorem \ref{vt1.1}. Then, we have
\begin{align}
&\frac{\rm d}{{\rm d}t}\bar{\mathcal{E}}_{0}(f_{t}(t))+\frac{3}{2}\|\nabla b_{t}\|^{2}_{H^{2}}+\frac{1}{2}\|\nabla\cdot b_{t}\|^{2}_{H^{2}}
+\frac{1}{4}\|\nabla a_{t}\|^{2}_{H^{2}}\nonumber\\
\leq\,& C\big(\|u_{t}-b_{t}\|^{2}_{H^{2}}+\|\{\mathbf{I}-\mathbf{P}\}f_{t}\|^{2}_{L^{2}_{x}(H^{3}_{x})}\big)
+C\big(1+\|\rho\|^{2}_{H^{3}}+\|u\|^{2}_{H^{3}}\big)\nonumber\\
&\times\big(\|\nabla(a,b,u)\|^{2}_{H^{2}}
+\|\{\mathbf{I}-\mathbf{P}\}f\|^{2}_{L^{2}_{x}(H^{3}_{x})}+\|u-b\|^{2}_{H^{2}}\big)
\big(\|f_{t}\|^{2}_{L^{2}_{x}(H^{3}_{x})}+\|(\rho_{t},u_{t})\|^{2}_{H^{3}}\big)\nonumber\\
&+C\big(\|(\rho,u)\|^{2}_{H^{3}}+\|\rho\|^{2}_{H^{3}}\|u\|^{2}_{H^{3}}\big)
\big(\|u_{t}-b_{t}\|^{2}_{H^{2}}+\|\{\mathbf{I}-\mathbf{P}\}f_{t}\|^{2}_{L^{2}_{x}(H^{3}_{x})}\big).\label{v3.37}
\end{align}
\end{Proposition}
\begin{Proposition}\label{vl3.5}
Assume that $(f,\rho,u)$ is the solution obtained in Theorem \ref{vt1.1}. Then, we have
\begin{align}
&\frac{\rm d}{{\rm d}t}\int\partial^{\alpha}u_{t}\partial^{\alpha}\nabla \rho_{t}\, {\rm d}x  +\frac{1}{2}p'(1)\|\nabla \partial^{\alpha}\rho_{t}\|^{2}\nonumber\\
\leq\,& C\big(\|\nabla u_{t}\|^{2}_{H^{3}}+\|u_{t}-b_{t}\|^{2}_{H^{2}}\big)+C\big(1+\|\rho\|^{4}_{H^{3}}\big)\nonumber\\
&\big(\|\nabla(a,\rho)\|^{2}_{H^{2}}+\|\nabla u\|^{2}_{H^{3}}\big)\big(\|f_{t}\|^{2}_{L^{2}_{x}(H^{3}_{x})}+\|(\rho_{t},u_{t})\|^{2}_{H^{3}}\big)\nonumber\\
&+C \|(\rho,u)\|_{H^{3}}\|\nabla(\rho_{t},u_{t})\|^{2}_{H^{2}}.\label{v3.38}
\end{align}
\end{Proposition}
Now, we define a temporal energy functional $\bar{\mathcal{E}}_{1}(t)$ and the corresponding dissipation rate $\bar{\mathcal{D}}_{1}(t)$ by
\begin{align}
\bar{\mathcal{E}}_{1}(t)& :=\|f_{t}\|^{2} +\|\rho_{t}\|^{2}+\|u_{t}\|^{2}+\sum_{1\leq |\alpha|\leq 3}\bigg\{\|\partial^{\alpha}f_{t}\|^{2}
+\Big\|\frac{\sqrt{p'(1+\rho)}}{1+\rho}\partial^{\alpha}\rho_{t}\Big\|^{2}
\nonumber\\
&\quad+\|\partial^{\alpha}u_{t}\|^{2}\bigg\}
+\kappa_{1}\bar{\mathcal{E}}_{0}(t)+\kappa_{2}\sum_{|\alpha|\leq 2}\int_{\mathbb{R}^{3}}\partial^{\alpha}u_{t}\cdot\partial^{\alpha}\nabla \rho_{t} \, {\rm d}x  ,\nonumber\\
\bar{\mathcal{D}}_{1}(t)& :=\|\nabla (a_{t},b_{t},\rho_{t},u_{t})\|^{2}_{H^{2}}+\|b_{t}-u_{t}\|^{2}_{H^{3}}
+\sum_{|\alpha|\leq 3}\Big\{\|\partial^{\alpha}\{\mathbf{I}-\mathbf{P}\}f_{t}\|^{2}_{\nu}+\|\partial^{\alpha}\nabla u_{t}\|^{2}\Big\},\nonumber
\end{align}
 where $\kappa_{1},\kappa_{2}>0$ are  small constants to be chosen later. Obviously, we have
$$\bar{\mathcal{E}}_{1}(t)\sim \|f_{t}(t)\|^{2}_{L^{2}_{v}(H^{3}_{x})}+\|(\rho_{t},u_{t})(t)\|^{2}_{H^{3}}.$$
According to \eqref{v3.32}, \eqref{v3.37}, \eqref{v3.38}, and by choosing $\kappa_{1},\, \kappa_{2}\,\text{and}\, \epsilon$\, sufficiently small,  we finally obtain
that
\begin{align}
&\frac{\rm d}{{\rm d}t}\bar{\mathcal{E}}_{1}(t)+\lambda \sum _{|\alpha|\leq 3}\|\partial^{\alpha}\{\mathbf{I}-\mathbf{P}\}f_{t}\|^{2}_{\nu}\nonumber\\
&+\lambda \Big(\|\nabla (a_{t},b_{t},\rho_{t},u_{t})\|^{2}_{H^{2}}+\|u_{t}-b_{t}\|^{2}_{H^{3}}+\|\nabla u_{t}\|^{2}_{H^{3}}\Big)\nonumber\\
\leq & C \Big\{\|\rho\|_{H^{3}}+\|\rho\|^{2}_{H^{3}}+\|u\|_{H^{4}}+\|u\|^{2}_{H^{3}}+\|\rho\|_{H^{3}}\|u\|_{H^{4}}+\|\rho\|^{2}_{H^{3}}\|u\|^{2}_{H^{3}}\Big\}\nonumber\\
&\times \Big\{\sum _{|\alpha|\leq 3}\|\partial^{\alpha}\{\mathbf{I}-\mathbf{P}\}f_{t}\|^{2}_{\nu}
+\|\nabla (a_{t},b_{t},\rho_{t},u_{t})\|^{2}_{H^{2}}+\|u_{t}-b_{t}\|^{2}_{H^{2}}\Big\}\nonumber\\
&+ C(1+\|\rho\|^{4}_{H^{4}}+\|u\|^{2}_{H^{3}})\Big\{\sum _{|\alpha|\leq 3}\|\partial^{\alpha}\{\mathbf{I}-\mathbf{P}\}f\|^{2}_{\nu}+\|\nabla (a,b)\|^{2}_{H^{2}}\nonumber\\
&+\|\nabla \rho\|^{2}_{H^{3}}+\|\nabla u\|^{2}_{H^{4}}+\|u-b\|^{2}_{H^{2}}\Big\}
\big(\|f_{t}\|^{2}_{L^{2}_{v}(H^{3}_{x})}+\|(\rho_{t},u_{t})\|^{2}_{H^{3}}\big).\nonumber
\end{align}
For simplicity, by the definition of  $\mathcal{E}_{1}(t),\mathcal{D}_{1}(t)$ in \eqref{vz1},\eqref{vz2}, we can rewrite the above inequality as
\begin{align}
\frac{\rm d}{{\rm d}t}\bar{\mathcal{E}}_{1}(t)+\lambda \bar{\mathcal{D}}_{1}(t)\leq C \big(\mathcal{E}^{\frac{1}{2}}_{1}(t)+\mathcal{E}^{2}_{1}(t)\big)\bar{\mathcal{D}}_{1}(t)
+C(1+\mathcal{E}^{2}_{1}(t))\mathcal{D}_{1}(t)\bar{\mathcal{E}}_{1}(t),\label{v3.39}
\end{align}

Next, we need to estimate  the mixed space-velocity derivatives of $f_{t}$.
Since $\|\partial^{\alpha}_{\beta}\mathbf{P}f_{t}\|\leq C \|\partial^{\alpha}f_{t}\|$ for any $\alpha,\beta$,
 we only estimate $\|\partial^{\alpha}_{\beta}\{\mathbf{I}-\mathbf{P}\}f_{t}\|$.
From \eqref{v2.25}, we easily deduce that
\begin{align}
&\partial_{t}\{\mathbf{I}-\mathbf{P}\}f_{t}+v\cdot \nabla_{x}\{\mathbf{I}-\mathbf{P}\}f_{t}+u\cdot \nabla_{v}\{\mathbf{I}-\mathbf{P}\}f_{t}
-\frac{1}{2}u\cdot v\{\mathbf{I}-\mathbf{P}\}f_{t}\nonumber\\
=\,&\mathcal{L}\{\mathbf{I}-\mathbf{P}\}f_{t}
+\mathbf{P}\big(v\cdot \nabla_{x}\{\mathbf{I}-\mathbf{P}\}f_{t}+u\cdot \nabla_{v}\{\mathbf{I}-\mathbf{P}\}f_{t}-\frac{1}{2}u\cdot v
\{\mathbf{I}-\mathbf{P}\}f_{t}\big)\nonumber\\
 &-\{\mathbf{I}-\mathbf{P}\}\Big(v\cdot \nabla_{x}\mathbf{P}f_{t}+u\cdot\nabla_{v}\mathbf{P}f_{t}-\frac{1}{2}u\cdot v\mathbf{P}f_{t}\Big)
+\{\mathbf{I}-\mathbf{P}\}G_{t},\label{v3.40}
\end{align}
where $\{\mathbf{I}-\mathbf{P}\}G_{t}$ is defined by
\begin{align}
& \{\mathbf{I}-\mathbf{P}\}G_{t}\nonumber\\
&=\rho\mathcal{ L}\{\mathbf{I}-\mathbf{P}\}f_{t}+\rho_{t} \mathcal{L}\{\mathbf{I}-\mathbf{P}\}f
-\rho u\nabla _{v}\{\mathbf{I}-\mathbf{P}\}f_{t}+\frac{1}{2}\rho u\cdot v\{\mathbf{I}-\mathbf{P}\}f_{t}\nonumber\\
&\quad-(\rho_{t}u+(1+\rho)u_{t})\cdot \Big(\nabla_{v}-\frac{v}{2}\Big)\{\mathbf{I}-\mathbf{P}\}f\nonumber\\
&\quad+\mathbf{P}\Big\{\rho u\cdot \Big(\nabla_{v}-\frac{v}{2}\Big)
\{\mathbf{I}-\mathbf{P}\}f_{t}
+(\rho_{t}u+(1+\rho)u_{t})\cdot \Big(\nabla_{v}-\frac{v}{2}\Big)\{\mathbf{I}-\mathbf{P}\}f\Big\}\nonumber\\
&\quad-\{\mathbf{I}-\mathbf{P}\}\Big\{\rho u\cdot \Big(\nabla_{v}-\frac{v}{2}\Big) \mathbf{P}f_{t}
+(\rho_{t}u+(1+\rho)u_{t})\cdot \Big(\nabla_{v}-\frac{v}{2}\Big)\mathbf{P}f\Big\}.
\nonumber
\end{align}
Similarly to the proof of Proposition \ref{vl2.6}, we can prove that
\begin{Proposition}\label{vl3.6}
Assume that $(f,\rho,u)$ is the solution obtained in Theorem \ref{vt1.1}. Then we have
\begin{align}
&\frac{\rm d}{{\rm d}t}\sum_{\substack{|\beta|=k \\ |\alpha|+|\beta|\leq 3}}\|\partial^{\alpha}_{\beta}\{\mathbf{I}-\mathbf{P}\}f_{t}\|^{2}
+\lambda \sum_{\substack{|\beta|=k \\ |\alpha|+|\beta|\leq 3}}\|\partial^{\alpha}_{\beta}\{\mathbf{I}-\mathbf{P}\}f_{t}\|^{2}_{\nu}\nonumber\\
\leq& C\Big\{\sum_{|\alpha'|\leq 4-k}\|\partial^{\alpha'}\{\mathbf{I}-\mathbf{P}\}f_{t}\|^{2}_{\nu}+\|\nabla b_{t}\|^{2}_{H^{3-k}} \Big\}
 +C\chi_{2\leq k\leq 4}\sum_{\substack{1\leq |\beta'|\leq k-1\\
|\alpha'|+|\beta'|\leq 3 }}\|\partial^{\alpha'}_{\beta'}\{\mathbf{I}-\mathbf{P}\}f_{t}\|^{2}_{\nu}\nonumber\\
&+ C(1+\|\rho\|^{2}_{H^{3}})\|\nabla u\|^{2}_{H^{2}}
\bigg\{\sum_{|\alpha'|\leq3-k}\|\partial^{\alpha'}\{\mathbf{I}-\mathbf{P}\}f_{t}\|^{2} +
\sum_{\substack{1\leq |\beta'|\leq 3 \\ |\alpha'|+|\beta'|\leq 3}}\|\partial^{\alpha'}_{\beta'}\{\mathbf{I}-\mathbf{P}\}f_{t}\|^{2}\bigg\}\nonumber\\
&+C(1+\|\rho\|^{2}_{H^{3}}+\|u\|^{2}_{H^{3}})
\Big\{\sum_{ |\alpha'|+|\beta'|\leq 3}\|\partial^{\alpha'}_{\beta'}\{\mathbf{I}-\mathbf{P}\}f\|^{2}_{\nu}
+\|\nabla b\|^{2}_{H^{2}}\Big\}\|(\rho_{t},u_{t})\|^{2}_{H^{3}}\nonumber\\
&+C(\|\rho\|^{2}_{H^{3}}+\|u\|^{2}_{H^{3}}+\|\rho\|^{2}_{H^{3}}\|u\|^{2}_{H^{3}})
\Big\{\sum_{|\alpha'|\leq3-k}\|\partial^{\alpha'}\{\mathbf{I}-\mathbf{P}\}f_{t}\|^{2}_{\nu}\nonumber\\
&+\sum_{\substack{1\leq|\beta'|\leq 3 \\ |\alpha'|+|\beta'|\leq 3}}\|\partial^{\alpha'}_{\beta'}\{\mathbf{I}-\mathbf{P}\}f_{t}\|^{2}_{\nu}
+\|\nabla b_{t}\|^{2}_{H^{2}}\Big\}.
 \label{v3.41}
\end{align}
\end{Proposition}
Multiplying \eqref{v3.41} by suitable constants $R_{k}$ and taking summation over $k$, we obtain  the following inequality:
\begin{align}
&\frac{\rm d}{{\rm d}t}\sum_{1\leq k\leq 3}R_{k}\sum_{\substack{|\beta|=k \\ |\alpha|+|\beta|\leq 3}}\|\partial^{\alpha}_{\beta}\{\mathbf{I}-\mathbf{P}\}f_{t}\|^{2}
+\lambda \sum_{\substack{1\leq|\beta|\leq 3 \\ |\alpha|+|\beta|\leq 3}}\|\partial^{\alpha}_{\beta}\{\mathbf{I}-\mathbf{P}\}f_{t}\|^{2}_{\nu}\nonumber\\
\leq\,& C\Big\{\sum_{|\alpha|\leq 3}\|\partial^{\alpha}\{\mathbf{I}-\mathbf{P}\}f_{t}\|^{2}_{\nu}+\|\nabla b_{t}\|^{2}_{H^{2}} \Big\}
\nonumber\\
&+ C(1+\|\rho\|^{2}_{H^{3}})\|\nabla u\|^{2}_{H^{2}}
\Big\{\sum_{|\alpha|\leq 2}\|\partial^{\alpha}\{\mathbf{I}-\mathbf{P}\}f_{t}\|^{2}+
\sum_{\substack{1\leq |\beta|\leq 3 \\ |\alpha|+|\beta|\leq 3}}\|\partial^{\alpha}_{\beta}\{\mathbf{I}-\mathbf{P}\}f_{t}\|^{2}\Big\}\nonumber\\
&+C(1+\|\rho\|^{2}_{H^{3}}+\|u\|^{2}_{H^{3}})
\Big\{\sum_{ |\alpha'|+|\beta'|\leq 3}\|\partial^{\alpha'}_{\beta'}\{\mathbf{I}-\mathbf{P}\}f\|^{2}_{\nu}
+\|\nabla b\|^{2}_{H^{2}}\Big\}\|(\rho_{t},u_{t})\|^{2}_{H^{3}}\nonumber\\
&+C(\|\rho\|^{2}_{H^{3}}+\|u\|^{2}_{H^{3}}+\|\rho\|^{2}_{H^{3}}\|u\|^{2}_{H^{3}})
\Big\{\sum_{|\alpha|\leq3}\|\partial^{\alpha}\{\mathbf{I}-\mathbf{P}\}f_{t}\|^{2}_{\nu}\nonumber\\
&+\sum_{\substack{1\leq|\beta|\leq 3 \\ |\alpha|+|\beta|\leq 3}}\|\partial^{\alpha}_{\beta}\{\mathbf{I}-\mathbf{P}\}f_{t}\|^{2}_{\nu}
+\|\nabla b_{t}\|^{2}_{H^{2}}\Big\}.
 \label{v3.42}
\end{align}

Now, we define $\bar{\mathcal{E}}_{2}(t)\,\text{and}\,\bar{\mathcal{D}}_{2}(t)$ as
\begin{align}
\bar{\mathcal{E}}_{2}(t) &:=\sum_{1\leq k\leq 3}R_{k}\sum_{\substack{|\beta|=k \\ |\alpha|+|\beta|\leq 3}}\|\partial^{\alpha}_{\beta}\{\mathbf{I}-\mathbf{P}\}f_{t}\|^{2},\nonumber\\
\bar{\mathcal{D}}_{2}(t) &:=\sum_{\substack{1\leq|\beta|\leq 3 \\ |\alpha|+|\beta|\leq 3}}\|\partial^{\alpha}_{\beta}\{\mathbf{I}-\mathbf{P}\}f_{t}\|^{2}_{\nu}.\nonumber
\end{align}
Therefore, according to the definition of $\mathcal{D}_{2}(t)$ in \eqref{vz4}, \eqref{v3.42} can be rewrited as
 \begin{align}
 \frac{\rm d}{{\rm d}t}\bar{\mathcal{E}}_{2}(t)+\lambda \bar{\mathcal{D}}_{2}(t)&\leq C \bar{\mathcal{D}}_{1}(t)
 +C\big(\mathcal{E}_{1}(t)+\mathcal{E}^{2}_{1}(t)\big)\big(\bar{\mathcal{D}}_{1}(t)+\bar{\mathcal{D}}_{2}(t)\big)\nonumber\\
 &\quad+C\big(1+\mathcal{E}_{1}(t)\big)\big(\mathcal{D}_{1}(t)+\mathcal{D}_{2}(t)\big)\big(\bar{\mathcal{E}}_{1}(t)+\bar{\mathcal{E}}_{2}(t)\big).\label{v3.43}
 \end{align}
Thus, we define a total  energy functional $\bar{\mathcal{E}}(t)$ and the corresponding dissipation rate $\bar{\mathcal{D}}(t)$ by
\begin{align}
\bar{\mathcal{E}}(t) &:=\bar{\mathcal{E}}_{1}(t)+\kappa_{3}\bar{\mathcal{E}}_{2}(t),\nonumber\\
\bar{\mathcal{D}}(t) &:=\bar{\mathcal{D}}_{1}(t)+\kappa_{3}\bar{\mathcal{D}}_{2}(t),\nonumber
\end{align}
where $\kappa_{3}>  0$  is a small constant to be chosen later.\\

With the aid of \eqref{vz5},\eqref{vz6},\eqref{v3.39} and \eqref{v3.43}, we conclude that
\begin{align}
\frac{\rm d}{{\rm d}t}\bar{\mathcal{E}}(t)+\lambda \bar{\mathcal{D}}
\leq  C\big(\mathcal{E}^{\frac{1}{2}}(t)+\mathcal{E}^{2}(t)\big)\bar{\mathcal{D}}
+C \big(1+\mathcal{E}^{2}\big)\mathcal{D}\bar{\mathcal{E}}.\label{v3.44}
\end{align}
By the uniform a priori estimates obtained in Section 4, we deduce that $\mathcal{E}(t)$ is sufficiently small. Therefore, we obtain that
\begin{align}\frac{\rm d}{{\rm d}t}\bar{\mathcal{E}}(t)+\lambda \bar{\mathcal{D}}
\leq C (1+\mathcal{E}^{2})\mathcal{D}\bar{\mathcal{E}}.\label{v3.444}
\end{align}
Applying Grownwall's inequality to \eqref{v3.444}, we obtain
\begin{align}
\sup_{ 0\leq t<\infty}\bar{\mathcal{E}}(t)+\lambda\int_{0}^{+\infty}\bar{\mathcal{D}}{\rm d}t \leq C \bar{\mathcal{E}}(0)< +\infty. \nonumber
\end{align}
Meanwhile, according to equations \eqref{v1.5}-\eqref{v1.8} and the uniform a priori estimates obtained above, we deduce that,
$$\int_{0}^{+\infty}(\|f_{t}\|^{2}+\|\rho_{t}\|^{2}+\|u_{t}\|^{2})\,{\rm d}t<  +\infty.$$
Thus,
$$\int_{0}^{+\infty}(\|f_{t}\|^{2}_{H^{3}_{x,v}}+\|\rho_{t}\|^{2}_{H^{3}}+\|u_{t}\|^{2}_{H^{3}})\,{\rm d}t< +\infty.$$

On the other hand, we can apply Lemma \ref{vl3.1} to obtain
\begin{align}
\|f_{t}\|^{2}_{H^{3}_{x,v}}+\|\rho_{t}\|^{2}_{H^{3}}+\|u_{t}\|^{2}_{H^{3}}\leq \frac{C}{1+t}. \label{v3.45}
\end{align}

Now, we define the following functionals
\begin{align}
&\tilde{\mathcal{E}}_{0}(f_{t}(t)) :=\sum_{|\alpha|\leq 2}\sum_{i,j}\int \partial^{\alpha}_{x}
(\partial_{x_{j}}b_{i}+\partial_{_{x_{i}}}b_{j})\partial^{\alpha}_{x}\Gamma_{i,j}(\{\mathbf{I}-\mathbf{P}\}f)\, {\rm d}x \nonumber\\
&\qquad\qquad\quad-\sum_{|\alpha|\leq 2}\int \partial^{\alpha}_{x}a\partial^{\alpha}_{x}\nabla_{x}\cdot b \, {\rm d}x ,\nonumber\\
&\tilde{\mathcal{E}}_{1}(t) :=\|f\|^{2} +\|\rho\|^{2}+\|u\|^{2}+\sum_{1\leq |\alpha|\leq 3}\bigg\{\|\partial^{\alpha}f\|^{2} +\Big\|\frac{\sqrt{p'(1+\rho)}}{1+\rho}\partial^{\alpha}\rho\Big\|^{2}
+\|\partial^{\alpha}u\|^{2}\bigg\}\nonumber\\
&\qquad\qquad+\tau_{1}\tilde{\mathcal{E}}_{0}(t)+\tau_{2}\sum_{|\alpha|\leq 2}\int \partial^{\alpha}u\cdot\partial^{\alpha}\nabla \rho \, {\rm d}x , \nonumber\\
&\tilde{\mathcal{D}}_{1}(t) :=\|\nabla (a,b,\rho,u)\|^{2}_{H^{2}}+\|b-u\|^{2}_{H^{3}}
+\sum_{|\alpha|\leq 3}\big\{\|\partial^{\alpha}\{\mathbf{I}-\mathbf{P}\}f\|^{2}_{\nu}+\|\partial^{\alpha}\nabla u\|^{2}\big\},\nonumber\\
&\tilde{\mathcal{E}}_{2}(t) :=\sum_{1\leq k\leq 3}C_{k}\sum_{\substack{|\beta|:=k \\ |\alpha|+|\beta|\leq 3}}\|\partial^{\alpha}_{\beta}\{\mathbf{I}-\mathbf{P}\}f\|^{2},\nonumber\\
&\tilde{\mathcal{D}}_{2}(t) :=\sum_{\substack{1\leq|\beta|\leq 3 \\ |\alpha|+|\beta|\leq 3}}\|\partial^{\alpha}_{\beta}\{\mathbf{I}-\mathbf{P}\}f\|^{2}_{\nu},\nonumber\\
&\tilde{\mathcal{E}}(t) :=\tilde{\mathcal{E}}_{1}(t)+ \tau_{3}\tilde{\mathcal{E}}_{2}(t),\nonumber\\
&\tilde{\mathcal{D}}(t) :=\tilde{\mathcal{D}}_{1}(t)+ \tau_{3}\tilde{\mathcal{D}}_{2}(t).\nonumber
\end{align}
According to the results obtained in Section 4, we finally obtain that
\begin{align}
&\sup_{x\in \mathbb{R}^{3}}\Big\{\sum_{|\alpha|+|\beta|\leq 1}\|\partial^{\alpha}_{\beta}f\|^{2}_{L^{2}_{v}}\Big\}
+\|\rho\|^{2}_{L^{\infty}}+\|u\|^{2}_{L^{\infty}}\nonumber\\
\leq &\sum_{|\alpha|+|\beta|\leq 2}\|\nabla\partial^{\alpha}_{\beta}f\|^{2}+\|\nabla \rho\|^{2}_{H^{2}}+\|\nabla u\|^{2}_{H^{2}}\leq
\tilde{\mathcal{D}}(t)\leq -\frac{\rm d}{{\rm d}t}\tilde{\mathcal{E}}(t) \nonumber\\
\leq& \, C (\|f\|_{H^{3}_{x,v}}+\|(\rho,u)\|_{H^{3}})(\|f_{t}\|_{H^{3}_{x,v}}+\|(\rho_{t},u_{t})\|_{H^{3}})\leq C (1+t)^{-\frac{1}{2}}.\label{v3.46}
\end{align}
This completes the proof of Theorem \ref{vt1.1}.

\subsection{The case of periodic domain}
In this subsection we study the large time behavior of  classical solutions
when $\Omega $ is a spatial periodic domain $\mathbb{T}^{3}$.
We shall use the  uniform a priori  estimates  obtained in Section 4 below.
It follows from \eqref{v4.5} together with \eqref{v4.7} that
$$\frac{\rm d}{{\rm d}t}\mathcal{E}(t) +\lambda \mathcal{D}_{\mathbb{T}}(t) \leq C(\mathcal{E}^{\frac{1}{2}}(t)+\mathcal{E}^{2}(t))\mathcal{D}_{\mathbb{T}}(t).$$
Using the fact that $\mathcal{E}(t)$ is small enough and uniformly in time, and $\mathcal{E}(t)\leq C \mathcal{D}_{\mathbb{T}}(t)$,
 we  then obtain
$$\frac{\rm d}{{\rm d}t}\mathcal{E}(t) +\lambda \mathcal{E}(t) \leq 0$$
for all $t\geq 0.$  This gives the desired exponential decay by applying Gronwall's inequality, and hence
  completes the proof of Theorem \ref{vt1.2}.

\bigskip

\section{A priori estimates of the classical solutions}
In this section, we shall establish  the uniform-in-time a priori estimates in the spatial domain $\Omega=\mathbb{R}^3$ or $\mathbb{T}^3$
which
 have been used in Sections 2 and 3.
We need the following two assumptions:

\begin{enumerate}[(1)]
\leftmargin=12mm
\item  $(f,\rho,u)$ is the smooth solution to the Cauchy problem \eqref{v1.5}-\eqref{v1.8} for $0< t<T$ with a fixed $T>0$;

\item $(f,\rho,u)$ satisfies
\begin{align}
\sup_{0<t<T}\{\|f(t)\|_{H^{4}_{x,v}}+\|(\rho,u)(t)\|_{H^{4}_{x}}\}\leq\delta, \label{v2.2}
\end{align}
where $0<\delta<1 $ is a sufficiently small generic constant.
\end{enumerate}

\smallskip

First, we introduce a  lemma which is useful in the subsequent estimates:
\begin{Lemma}[see \cite{CDM}]\label{vl2.1}
There exists a positive constant $C$, such that for any $f,g\in H^{4 }(\Omega)$ and
any multi-index $\gamma$ with $1\leq |\gamma|\leq 4$,
\begin{align}
\|f\|_{L^{\infty}(\Omega)}&\leq C\|\nabla_{x}f\|_{L^{2}(\Omega)}^{1/2}\|\nabla^{2}_{x}f\|_{L^{2}(\Omega)}^{1/2},\label{v2.3}\\
\|fg\|_{H^{2}(\Omega)}&\leq C \|f\|_{H^{2}(\Omega)}\|\nabla _{x}g\|_{H^{2}(\Omega)},\label{v2.4}\\
\|\partial^{\gamma}_{x}(fg)\|_{L^{2}(\Omega)}&\leq C \|\nabla _{x}f\|_{H^{3}(\Omega)}\|\nabla _{x}g\|_{H^{3}(\Omega)}.\label{v2.5}
\end{align}
\end{Lemma}

\subsection{The case of the whole space}
In this subsection we deal with the uniform-in-time a priori estimates in the whole space $\Omega=\mathbb{R}^{3}$.
\begin{Proposition}\label{vl2.2}
For smooth solutions of the problem \eqref{v1.5}-\eqref{v1.8}, we have
\begin{align}
&\frac{1}{2}\frac{\rm d}{{\rm d}t}\|(f,\rho,u)(t)\|^{2}+\lambda (\|\{\mathbf{I}-\mathbf{P}\}f\|_{\nu}^{2}+
\|b-u\|^{2}+\|\nabla u\|^{2})
\nonumber\\
\!\!\!\!\!\!\!\!\leq \,&C(\|(\rho,u)\|_{H^{2}}+\|\rho\|_{H^{2}}\|u\|_{H^{2}})  (\|\nabla_{x}(a,b,\rho,u)\|^{2}+\|u-b\|^{2}+\|\{\mathbf{I}-\mathbf{P}\}f\|^{2}_{\nu})\label{v2.6}
\end{align}
for all $0\leq t<T$. 
\end{Proposition}

\begin{proof}
 Multiplying \eqref{v1.5}-\eqref{v1.7} by $f,\rho,$ and $u$ respectively and then taking integration and summation, we get
 \begin{align}
 &\frac{1}{2}\frac{\rm d}{{\rm d}t}(\|f\|^{2} +\|\rho\|^{2}+\|u\|^{2})+\int\langle -\mathcal{L}\{\mathbf{I}-\mathbf{P}\}f,f\rangle   \, {\rm d}x
 +\int\frac{|\nabla u|^{2}}{1+\rho}\, {\rm d}x +\|b-u\|^{2}\nonumber\\
=&\int u\Big\langle \frac{1}{2}vf,f\Big\rangle   \, {\rm d}x -\int a|u|^{2}\, {\rm d}x -\int \rho\dv u\, {\rm d}x -\int p'(1)\nabla \rho\cdot u\, {\rm d}x \nonumber\\
 &-\int(u\cdot\nabla u)\cdot u\, {\rm d}x -\int\nabla\frac{1}{1+\rho}\nabla u\cdot u\, {\rm d}x -\int\Big(\frac{p'(1+\rho)}{1+\rho}-p'(1)\Big)\nabla\rho \cdot u\, {\rm d}x  \nonumber\\
 &-\frac{1}{2}\int \rho^{2}\dv u\, {\rm d}x +\iint\rho\Big(\mathcal{L}f -u\cdot \nabla_{v}f+\frac{1}{2}u\cdot vf+u\cdot v\sqrt{M}\Big)f\, {\rm d}x  {\rm d}v. \label{v2.7}
\end{align}
By \eqref{v2.1}, we have
$$
\langle -\mathcal{L}\{\mathbf{I}-\mathbf{P}\}f, f\rangle \geq \lambda_{0}|\{\mathbf{I}-\mathbf{P}\}f|^{2}_{\nu}.
$$
Thus,  we only need   to estimate the terms on the right hand side of the equality \eqref{v2.7}.
For the first two terms, by taking the same computation as that in  \cite{CDM}, we get
\begin{align}
&\int u\Big\langle \frac{1}{2}vf,f\Big\rangle   \, {\rm d}x -\int a|u|^{2}\, {\rm d}x \nonumber\\
&\quad \leq C(\|\nabla_{v}u\|_{H^{1}}+\|u\|_{H^{1}})\|\{\mathbf{I}-\mathbf{P}f\}\|_{\nu}^{2}+C\|u\|_{H^{1}}(\|\nabla_{x}(a,b)\|^{2}+\|u-b\|^{2}). \nonumber
\end{align}
Without loss of generality, we can assume that $p'(1)=1$, then, one has
\begin{align}
\int \rho\ \dv u\, {\rm d}x +\int p'(1)\nabla \rho\cdot u\, {\rm d}x =0. \nonumber
\end{align}
For the next four terms, using H\"{o}lder's, Sobolev's inequalities and Lemma \ref{vl2.1}, we have
\begin{align}
\int(u\cdot\nabla u)\cdot u\, {\rm d}x &\leq C\|u\|_{L^{3}}\|\nabla_{x}u\|_{L^{2}}\|u\|_{L^{6}}\leq C\|u\|_{H^{1}}\|\nabla u\|^{2}_{L^{2}},\nonumber\\
\int\nabla\frac{1}{1+\rho}\nabla u\cdot u\, {\rm d}x &\leq C \|u\|_{H^{2}}(\|\nabla \rho\|^{2}+\|\nabla u\|^{2}),\nonumber\\
\int\Big(\frac{p'(1+\rho)}{1+\rho}-p'(1)\Big)\nabla\rho \cdot u\, {\rm d}x & \leq C \|\rho\|_{L^{3}}\|\nabla_{x}\rho\|_{L^{2}}\|u\|_{L^{6}}\nonumber\\
&\leq C\|\rho\|_{H^{1}}(\|\nabla \rho\|^{2}+\|\nabla u\|^{2}),\nonumber\\
\frac{1}{2}\int \rho^{2}\ {\rm div} u\, {\rm d}x &\leq C\|\rho\|_{L^{3}}\|\nabla_{x}u\|_{L^{2}}\|\rho\|_{L^{6}}\nonumber\\
&\leq C\|\rho\|_{H^{1}}(\|\nabla \rho\|^{2}+\|\nabla u\|^{2}).\nonumber
\end{align}
Here, we have used the facts that $\|u\|_{L^{6}}\leq C \|\nabla u\|_{L^{2}}$ and $\|u\|_{L^{3}}\leq C \|u\|_{H^{1}}$.

Using the macro-micro decomposition \eqref{decomp}, we rewrite  the last term as
\begin{align}
&\iint\rho\Big(\mathcal{L}f -u\cdot \nabla_{v}f+\frac{1}{2}u\cdot vf+u\cdot v\sqrt{M}\Big)f\, {\rm d}x  {\rm d}v\nonumber\\
&\qquad=\int\rho\langle -\mathcal{L}\{\mathbf{I}-\mathbf{P}\}f,f\rangle   \, {\rm d}x +\int\rho(u-b)\cdot b\, {\rm d}x  +\frac{1}{2}\iint\rho u\cdot vf^{2}\, {\rm d}x  {\rm d}v.\nonumber
\end{align}
It is easy to see that
\begin{align}
&\int\rho(u-b)\cdot b\, {\rm d}x \leq C\|\rho\|_{L^{3}}\|u-b\|_{L^{2}}\|b\|_{L^{6}}\leq C\|\rho\|_{H^{1}}(\|\nabla b\|^{2}_{L^{2}}+\|u-b\|^{2}_{L^{2}}).\nonumber
\end{align}
Noticing that
\begin{align}
&\Big\langle \frac{1}{2}v_{i}f,f\Big\rangle   =ab_i+\langle v_{i}\mathbf{P}f,\{\mathbf{I}-\mathbf{P}\}f\rangle
  +\Big\langle \frac{1}{2}v_{i},|\{\mathbf{I}-\mathbf{P}\}f|^{2}\Big\rangle,\quad (i=1,2,3),\nonumber
\end{align}
the term $\frac{1}{2}\iint\rho u\cdot vf^{2}\, {\rm d}x  {\rm d}v$ can be estimated as follows:
\begin{align}
\int\rho u \cdot b a \, {\rm d}x &\leq C\|\rho\|_{L^{3}}\|u\|_{L^{3}}\|a\|_{L^{6}}\|b\|_{L^{6}}\nonumber\\
&\leq C\|\rho\|_{H^{1}}\|u\|_{H^{1}}(\|\nabla a\|^{2}+\|\nabla b\|^{2}),\nonumber\\
\int\rho\langle v_{i}\mathbf{P}f,\{\mathbf{I}-\mathbf{P}\}f\rangle   \, {\rm d}x
&\leq C\|\rho u\|_{L^{3}}\|(a,b)\|_{L^{6}}\|\{\mathbf{I}-\mathbf{P}\}f\|\nonumber\\
&\leq C\|\rho\|_{H^{1}}\|u\|_{H^{1}}(\|\nabla (a,b)\|^{2}+\|\{\mathbf{I}-\mathbf{P}\}f\|_{\nu}^{2}),\nonumber\\
\int\rho\big\langle \frac{1}{2}v_{i},|\{\mathbf{I}-\mathbf{P}\}f|^{2}\big\rangle   \, {\rm d}x
&\leq C\|\rho u\|_{L^{\infty}}\|\{\mathbf{I}-\mathbf{P}\}f\|_{\nu}^{2}\nonumber\\
&  \leq C\|\rho\|_{H^{2}}\|u\|_{H^{2}}\|\{\mathbf{I}-\mathbf{P}\}f\|_{\nu}^{2}.\nonumber
\end{align}
Therefore, the last term is bounded by
\begin{align}
&\int\rho\langle -\mathcal{L}\{\mathbf{I}-\mathbf{P}\}f,f\rangle   \, {\rm d}x \nonumber\\
&\qquad +C\big(\|\rho\|_{H^{1}}+\|\rho\|_{H^{2}}\|u\|_{H^{2}}\big)
\big(\|\nabla (a,b)\|^{2}+\|u-b\|^{2}+\|\{\mathbf{I}-\mathbf{P}\}f\|_{\nu}^{2}\big).\nonumber
\end{align}
Plugging all the above estimates   into \eqref{v2.7} and using \eqref{v2.2}, we obtain \eqref{v2.6}.
\end{proof}
\begin{Proposition}\label{vl2.3}
For smooth solutions of the problem \eqref{v1.5}-\eqref{v1.8}, we have
\begin{align}
&\frac{1}{2}\frac{\rm d}{{\rm d}t}\sum_{1\leq |\alpha|\leq 4}\Big\{\|\partial^{\alpha}f\|^{2} +\Big\|\frac{\sqrt{p'(1+\rho)}}{1+\rho}\partial^{\alpha}\rho\Big\|^{2}
+\|\partial^{\alpha}u\|^{2}\Big\}\nonumber\\
&+\lambda \sum_{1\leq |\alpha|\leq 4}\big\{\|\{\mathbf{I}-\mathbf{P}\}\partial^{\alpha}f\|_{\nu}^{2}+
\|\partial^{\alpha}(b-u)\|^{2}+\|\nabla \partial^{\alpha} u\|^{2}\big\}\nonumber\\
\leq \,&C\big\{\|u\|_{H^{4}}+\|\rho\|_{H^{4}}\|u\|_{H^{4}}(1+\|\rho\|^{2}_{H^{4}})+\|\rho\|_{H^{4}}(1+\|\rho\|^{3}_{H^{4}})\big\}
\nonumber\\
&\times \Big\{\|\nabla_{x}(a,b,\rho,u)\|^{2}_{H^{3}}+\sum_{1\leq |\alpha'|\leq 4}\|\{\mathbf{I}-\mathbf{P}\}\partial^{\alpha'}f\|^{2}_{\nu}\Big\}\label{v2.8}
\end{align}
for all $0\leq t<T $.
\end{Proposition}
\begin{proof}
Applying differentiation $\partial^{\alpha} (1\leq |\alpha|\leq 4)$ to the system \eqref{v1.5}-\eqref{v1.7}, we have
\begin{align}
&\partial_{t}(\partial^{\alpha}f)+v\cdot \nabla_{x}(\partial^{\alpha}f) +u\cdot \nabla_{v}(\partial^{\alpha}f)
-\partial^{\alpha}u\cdot v\sqrt{M}-\mathcal{L}\partial^{\alpha}f\nonumber\\
 &\qquad =\frac{1}{2}\partial^{\alpha}[(1+\rho)u\cdot vf]
 +[-\partial^{\alpha},u\cdot \nabla_{v}]f+\partial^{\alpha}\{\rho(\mathcal{L}f -u\cdot \nabla_{v}f+u\cdot v\sqrt{M})\},\label{v2.9}\\
 &\partial_{t}(\partial^{\alpha}\rho)+u\cdot \nabla \partial^{\alpha}\rho+(1+\rho)\dv \partial^{\alpha}u
 =[-\partial^{\alpha},\rho \nabla_{x}\cdot]u+[-\partial^{\alpha}, u \cdot\nabla_{x}]\rho,\label{v2.10}\\
 &\partial_{t}(\partial^{\alpha}u)+u\cdot \nabla (\partial^{\alpha} u)+\frac{p'(1+\rho)}{1+\rho}\nabla\partial^{\alpha}\rho
 -\partial^{\alpha}\Big(\frac{1}{1+\rho}\Delta u\Big)-\partial^{\alpha}(b-u)\nonumber\\
 &\qquad=[-\partial^{\alpha},u\cdot\nabla_{x}]u+\Big[-\partial^{\alpha}, \frac{p'(1+\rho)}{1+\rho}\nabla_{x}\Big]\rho
 -\partial^{\alpha}(ua),\label{v2.11}
\end{align}
where $[A,B]$ denotes the commutator $AB-BA$\, for two operators $A$ and $B.$
Now, multiplying \eqref{v2.9}-\eqref{v2.11} by $\partial^{\alpha}f,\frac{p'(1+\rho)}{(1+\rho)^{2}}\partial^{\alpha}\rho$, and $\partial^{\alpha}u$
 respectively and then taking integration and summation, we have
\begin{align}
&\frac{1}{2}\frac{\rm d}{{\rm d}t}\bigg\{\|\partial^{\alpha}f\|^{2} +\Big\|\frac{\sqrt{p'(1+\rho)}}{1+\rho}\partial^{\alpha}\rho\Big\|^{2}
+\|\partial^{\alpha}u\|^{2}\bigg\}\nonumber\\
&+\int\langle -\mathcal{L}\{\mathbf{I}-\mathbf{P}\}\partial^{\alpha}f,\partial^{\alpha}f\rangle   \, {\rm d}x
 +\int\frac{1} {1+\rho}|\nabla(\partial^{\alpha}u)|^{2}\, {\rm d}x +\|\partial^{\alpha}(b-u)\|^{2}\nonumber\\
=\,&\int\langle [-\partial^{\alpha},u\cdot \nabla_{v}]f,\partial^{\alpha}f\rangle
 \, {\rm d}x +\int\frac{1}{2}\langle \partial^{\alpha}[(1+\rho)u\cdot vf],\partial^{\alpha}f\rangle   \, {\rm d}x
\nonumber\\
&+\int[-\partial^{\alpha},\rho \nabla_{x}\cdot]u\frac{p'(1+\rho)}{(1+\rho)^{2}}\partial^{\alpha}\rho \, {\rm d}x
+\int[-\partial^{\alpha}, u \cdot\nabla_{x}]\rho\frac{p'(1+\rho)}{(1+\rho)^{2}}\partial^{\alpha}\rho \, {\rm d}x
\nonumber\\
&+\int[-\partial^{\alpha},u\cdot\nabla_{x}]u\partial^{\alpha}u\, {\rm d}x -\frac{1}{2}\int|\partial^{\alpha}u|^{2}\ \dv u\, {\rm d}x \nonumber\\
&+\int[-\partial^{\alpha}, \frac{p'(1+\rho)}{1+\rho}\nabla_{x}]\rho\partial^{\alpha}u\, {\rm d}x
+\int\partial^{\alpha}\rho\partial^{\alpha}u\cdot\nabla\frac{p'(1+\rho)}{1+\rho}\, {\rm d}x\nonumber\\
& -
\frac{1}{2}\int|\partial^{\alpha}\rho|^{2}\dv\Big(\frac{p'(1+\rho)}{(1+\rho)^{2}}u\Big)\, {\rm d}x -\int\partial^{\alpha}(ua)\partial^{\alpha}u\, {\rm d}x \nonumber\\
&-\sum_{i=1}^{3}\int\nabla\frac{1}{1+\rho}\nabla\partial^{\alpha}u_{i}\partial^{\alpha}u_{i}\, {\rm d}x -
\sum_{1\leq\beta\leq\alpha}c_{\alpha,\beta}\int\partial^{\beta}\Big(\frac{1}{1+\rho}\Big)\partial^{\alpha-\beta}\Delta u\partial^{\alpha}u\, {\rm d}x \nonumber\\
&+\int\langle \partial^{\alpha}\{\rho(\mathcal{L}f -u\cdot \nabla_{v}f+u\cdot v\sqrt{M})\},\partial^{\alpha}f\rangle
  \, {\rm d}x +\frac{1}{2}\int\partial_{t}\frac{p'(1+\rho)}{(1+\rho)^{2}}|\partial^{\alpha}\rho|^{2}\, {\rm d}x \nonumber\\
:=\,&\sum_{j=1}^{14}I_{j}, \label{v2.12}
\end{align}
where $C_{\alpha,\beta}$ are constants depending only on $\alpha$ and $\beta$. Each term in \eqref{v2.12} can be estimated as follows. For $I_{1},I_{2},I_{5}$, and $I_{10}$,
we can carry out similar arguments to the proof of Lemma 2.3 in \cite{CDM} to obtain (the details is omitted here)
\begin{align}
I_{1}\leq\,& C\|\nabla u\|_{H^{3}}\|\nabla_{x}f\|_{L^{2}_{v}(H^{3}_{x})}\|\nabla_{v}\partial^{\alpha}f\|,\nonumber\\
I_{2}\leq\,& C(1+\|\rho\|_{H^{4}})\|\nabla u\|_{H^{3}}\|\nabla_{x}f\|_{L^{2}_{v}(H^{3}_{x})}\|v\partial^{\alpha}f\|,\nonumber\\
I_{5}\leq\,& C\|\nabla u\|_{H^{3}}^{2}\|\partial^{\alpha}u\|,\nonumber\\
I_{10}\leq\,& C\|\nabla u\|_{H^{3}}\|\nabla a\|_{H^{3}}\|\partial^{\alpha}u\|.\nonumber
\end{align}
Using H\"{o}lder's, Sobolev's, and Young's inequalities, we easily get the following bounds:
\begin{align}
I_{3}+I_{4}&\leq C\|\nabla u\|_{H^{3}}\|\nabla \rho\|_{H^{3}}\|\partial^{\alpha}\rho\|\leq C\|\nabla u\|_{H^{3}}\|\nabla \rho\|^{2}_{H^{3}},\nonumber\\
I_{6}&\leq C\|\ \dv u\|_{L^{\infty}}\|\partial^{\alpha}u\|^{2}\leq C\|\nabla u\|_{H^{3}}\|\partial^{\alpha}u\|^{2},\nonumber\\
I_{7}&\leq C(\|\nabla \rho\|^{2}_{H^{3}}+\|\nabla \rho\|^{5}_{H^{3}})\|\nabla u\|_{H^{3}},\nonumber\\
I_{8}&\leq C\|u\|_{H^{4}}\|\nabla \rho\|^{2}_{H^{3}},\nonumber\\
I_{9}&\leq C\big(\|\ \dv u\|_{L^{\infty}}+\| u\|_{L^{\infty}}\| \nabla \rho\|_{L^{\infty}}\big)\|\partial^{\alpha}\rho\|^{2}\nonumber\\
&\leq C(1+\|\rho\|_{H^{3}})\|u\|_{H^{3}}\|\partial^{\alpha}\rho\|^{2},\nonumber\\
I_{11}&\leq C\|\nabla \rho\|_{H^{2}}\|\nabla\partial^{\alpha}u\|\|\partial^{\alpha}u\|
\leq C_{\epsilon}\|\nabla \rho\|_{H^{2}}^{2}\|\partial^{\alpha}u\|^{2}+\epsilon\|\nabla\partial^{\alpha}u\|^{2},\nonumber\\
I_{12}&\leq C\|\nabla \rho\|_{H^{2}}\|\nabla\partial^{\alpha}u\|\|\partial^{\alpha}u\|+
 C\big(\|\nabla \rho\|_{H^{3}}+\|\nabla \rho\|^{4}_{H^{3}}\big)\|\nabla u\|^{2}_{H^{3}}\nonumber\\
 &\leq C_{\epsilon}\big(\|\nabla \rho\|_{H^{3}}+\|\nabla \rho\|^{4}_{H^{3}}\big)\|\nabla u\|_{H^{3}}^{2}+\epsilon\|\nabla\partial^{\alpha}u\|^{2}\nonumber
\end{align}
with $\epsilon>0$ a small constant.

By means of \eqref{v1.5} and \eqref{v2.2}, one has
$$\sup_{0\leq t<T , x \in\mathbb{R}^{3}}|\partial_{t}\rho(t,x)|\leq(1+\|\rho\|_{L^{\infty}})\|\ \dv u\|_{L^{\infty}}
+\| u\|_{L^{\infty}}\|\nabla \rho\|_{L^{\infty}}\leq C\|u\|_{H^{3}}.$$
Then,  the term $I_{14}$  can be controlled by the following bound
$$C\|\rho_{t}\|_{L^{\infty}}\|\partial^{\alpha}\rho\|^{2}\leq C\|u\|_{H^{3}}\|\partial^{\alpha}\rho\|^{2}.$$

Now we   estimate the term $I_{13}$. We have
\begin{align}
\iint-\partial^{\alpha}(\rho u\cdot \nabla_{v}f)\partial^{\alpha}f\, {\rm d}x  {\rm d}v
&=\iint\partial^{\alpha}(\rho u f)\cdot\nabla_{v}\partial^{\alpha}f\, {\rm d}x  {\rm d}v\nonumber\\
& \leq C \|\nabla\rho\|_{H^{3}}\|\nabla u\|_{H^{3}}\|\nabla_{x} f\|_{L^{2}_{v}(H^{3}_{x})}\|\nabla_{v} \partial^{\alpha}f\|,\nonumber\\
\iint\partial^{\alpha}(\rho u)\cdot v\sqrt{M}\partial^{\alpha}f\, {\rm d}x  {\rm d}v
&\leq C\|\partial^{\alpha}(\rho u)\|\|v\sqrt{M}\partial^{\alpha}f\|\nonumber\\
& \leq C \|\nabla\rho\|_{H^{3}}\|\nabla u\|_{H^{3}}\|\partial^{\alpha}f\|_{\nu},\nonumber\\
\iint\partial^{\alpha}(\rho \mathcal{L}f )\partial^{\alpha}f\, {\rm d}x  {\rm d}v
&=-\iint\partial^{\alpha}(\rho g)\partial^{\alpha}g\, {\rm d}x  {\rm d}v\nonumber\\
&\leq C\|\nabla \rho\|_{H^{3}}\|\nabla g\|^{2}_{L^{2}_{v}(H^{3}_{x})}\nonumber\\
&\leq C\|\nabla \rho\|_{H^{3}} \sum_{1\leq |\alpha'|\leq 4}
\|\partial^{\alpha'}f\|_{\nu}^{2}\nonumber
\end{align}
with $g=\sqrt{M}\nabla_{v}(M^{-\frac{1}{2}}f)$. Thus, $I_{13}$ can be bounded by
$$C \big(\|\rho\|_{H^{4}}+\|\rho\|_{H^{4}}\|u\|_{H^{4}}\big)
\Big\{\|\nabla_{x}(a,b,u)\|^{2}_{H^{3}}+\sum_{1\leq |\alpha'|\leq 4}\|\{\mathbf{I}-\mathbf{P}\}\partial^{\alpha'}f\|^{2}_{\nu}\Big\}.$$

Plugging   the estimates on $I_{i}\,(1\leq i\leq14)$ into \eqref{v2.12} and taking summation over $1\leq |\alpha|\leq 4$, we obtain \eqref{v2.8}.
\end{proof}

In order to get the energy dissipation rate $\|\nabla_{x}(a,b)\|_{H^{3}}$, we need to  study the equations satisfied by $a$ and $b$.
We follow some ideas developed in \cite{CDM} where the lower order estimates similar to Proposition \ref{vl2.4} below were obtained.
It is easy to verify that $a$ and $b$ satisfy the following equations:
\begin{align}
&\partial_{t}a+\dv b=0,  \label{v2.13}  \\
&\partial_{t}b_{i}+\partial_{x_{i}}a +\sum_{j}\partial_{x_{j}}\Gamma_{i,j}(\{\mathbf{I}-\mathbf{P}\}f)
=-(1+\rho)b_{i}+(1+\rho)u_{i}(1+a),\label{v2.14}  \\
&\partial_{x_{j}}b_{i}+\partial_{x_{i}}b_{j}-(1+\rho)(u_{i}b_{j}+u_{j}b_{i})
=-\partial_{t}\Gamma_{i,j}\{\mathbf{I}-\mathbf{P}\}f +\Gamma_{i,j}(l+r+s) \label{v2.15}
\end{align}
for $1\leq i,j\leq3$, where $\Gamma_{i,j}$ is the moment functional defined by $$\Gamma_{i,j}=\langle (v_{i}v_{j}-1)\sqrt{M},g\rangle   $$
for any $g=g(v)$, and $l,r,s$ are defined respectively by
\begin{align}
&l:=-v\cdot \nabla_{x}\{\mathbf{I}-\mathbf{P}\}f+\mathcal{L}\{\mathbf{I}-\mathbf{P}\}f,\nonumber\\
&r:=-u\cdot\nabla_{v}\{\mathbf{I}-\mathbf{P}\}f+\frac{1}{2}u\cdot v\{\mathbf{I}-\mathbf{P}\}f,\nonumber\\
&s:=\rho M^{-\frac{1}{2}}\nabla_{v}\cdot\Big(\frac{v}{2}\sqrt{M}\{\mathbf{I}-\mathbf{P}\}f
+\sqrt{M}\{\mathbf{I}-\mathbf{P}\}f-u\sqrt{M}\{\mathbf{I}-\mathbf{P}\}f\Big).\nonumber
\end{align}
In fact, \eqref{v2.13} and \eqref{v2.14} can be obtained straightforwardly by  multiplying \eqref{v1.5} by $\sqrt{M}$ and
$v_i\sqrt{M}\,(1\leq i\leq 3)$ respectively and then taking the velocity integration over $\mathbb{R}^3$. To obtain \eqref{v2.15},
we can  rewrite \eqref{v1.5} as
\begin{align*}
   &\partial_{t}\mathbf{P}f+v\cdot \nabla_{x}\mathbf{P}f +u\cdot \nabla_{v}\mathbf{P}f-\frac{1}{2}u\cdot v\mathbf{P}f
   +\mathbf{P}_1f=-\partial_t\{\mathbf{I}-\mathbf{P}\}f+l+r+s,
 \end{align*}
then apply $\Gamma_{ij}$ to it and use \eqref{v2.13}.

We have the following estimate. 

\begin{Proposition}\label{vl2.4}
For smooth solutions of the problem \eqref{v1.5}-\eqref{v1.8}, we have
\begin{align}
\frac{\rm d}{{\rm d}t}\mathcal{E}_{0}(t)+\lambda \|\nabla_{x}(a,b)\|^{2}_{H^{3}}
\leq\, & C\big(\|\{\mathbf{I}-\mathbf{P}\}f\|^{2}_{L^{2}_{v}(H^{4}_{x})}
+\|u-b\|^{2}_{H^{3}}\big)
\nonumber\\
& +C\big(\|\rho,u\|^{2}_{H^{3}}+\|\rho\|^{2}_{H^{3}}\|u\|^{2}_{H^{3}}\big)\nonumber\\
&\times\big(\|\nabla_{x}\{\mathbf{I}-\mathbf{P}\}f\|^{2}_{L^{2}_{v}(H^{3}_{x})}+\|u-b\|^{2}_{H^{3}}+\|\nabla_{x}(a,b)\|^{2}_{H^{3}}\big)\label{v2.17}
\end{align}
for all $0\leq t<T $.
\end{Proposition}
\begin{proof}
It follows from \eqref{v2.15} that
\begin{align}
&\sum_{i,j}\|\partial^{\alpha}(\partial_{i}b_{j}+\partial_{j}b_{i})\|^{2}\nonumber\\
=\,&\sum_{i,j}\int\partial^{\alpha}(\partial_{i}b_{j}+\partial_{j}b_{i})
\nonumber\\
& \times\partial^{\alpha} [(1+\rho)(u_{i}b_{j}+u_{j}b_{i})
-\partial_{t}\Gamma_{i,j}(\{\mathbf{I}-\mathbf{P}\}f)+\Gamma_{i,j}(l+r+s)]\, {\rm d}x \nonumber\\
\,=&-\frac{\rm d}{{\rm d}t}\sum_{i,j}\int\partial^{\alpha}(\partial_{i}b_{j}+\partial_{j}b_{i})\partial^{\alpha}\Gamma_{i,j}(\{\mathbf{I}-\mathbf{P}\}f)\, {\rm d}x \nonumber\\
&+\sum_{i,j}\int\partial^{\alpha}(\partial_{i}\partial_{t}b_{j}+\partial_{j}\partial_{t}b_{i})\partial^{\alpha}\Gamma_{i,j}(\{\mathbf{I}-\mathbf{P}\}f)\, {\rm d}x \nonumber\\
&+\sum_{i,j}\int\partial^{\alpha}(\partial_{i}b_{j}+\partial_{j}b_{i})\partial^{\alpha}
[(1+\rho)(u_{i}b_{j}+u_{j}b_{i})+\Gamma_{i,j}(l+r+s)]\, {\rm d}x.\label{v2.18}
\end{align}
Applying \eqref{v2.14}, Lemma \ref{vl2.1} and Young's inequality, we obtain
\begin{align}
&\!\!\!\!\!\!\!\!\!\!\!\!\!\!\!\!\!\!\!\!\!\!\!\!
\sum_{i,j}\int\partial^{\alpha}(\partial_{i}\partial_{t}b_{j}+\partial_{j}\partial_{t}b_{i})\partial^{\alpha}\Gamma_{i,j}(\{\mathbf{I}-\mathbf{P}\}f)\, {\rm d}x\nonumber\\
\leq\, & \epsilon \|\nabla_{x}a\|^{2}_{H^{3}}
+C_{\epsilon}\|\nabla_{x}\{\mathbf{I}-\mathbf{P}\}f\|^{2}_{L^{2}_{v}(H^{3}_{x})}\nonumber\\
&+C(1+\|\rho\|^{2}_{H^{3}})(\|u-b\|^{2}_{H^{3}}+\|u\|^{2}_{H^{2}}\|\nabla_{x}a\|^{2}_{H^{2}}) \nonumber
\end{align}
with $\epsilon >0$  sufficiently small. For the final term on the right hand side of \eqref{v2.18}, we have the following estimate:
\begin{align}
&\!\!\!\!\!\!\!\!\!\!\!\!\!\!\!\!\!\!\!\!\!\!\!\!
\sum_{i,j}\int\partial^{\alpha}(\partial_{i}b_{j}+\partial_{j}b_{i})\partial^{\alpha}
[(1+\rho)(u_{i}b_{j}+u_{j}b_{i})+\Gamma_{i,j}(l+r+s)]\, {\rm d}x\nonumber\\
\leq\,& \frac{1}{2}\sum_{i,j}\|\partial^{\alpha}(\partial_{i}b_{j}
+\partial_{j}b_{i})\|^{2}\nonumber\\
&
+C\sum_{i,j}(\|\partial^{\alpha}(1+\rho)(u_{i}b_{j}+u_{j}b_{i})\|^{2}+\|\partial^{\alpha}\Gamma_{i,j}(l+r+s)\|^{2}).\nonumber
\end{align}
According to Lemma \ref{vl2.1}, the definition of $\Gamma_{i,j}$, and the expressions of $l$ and $r$, we get
\begin{align}
\sum_{i,j}\|\partial^{\alpha}(u_{i}b_{j}+u_{j}b_{i})\|^{2}&\leq C(1+\|\rho\|^{2}_{H^{3}})\|u\|^{2}_{H^{3}}\|\nabla_{x}b\|^{2}_{H^{3}},\nonumber\\
\sum_{i,j}\|\partial^{\alpha}\Gamma_{i,j}(l)\|^{2}&\leq C\|\{\mathbf{I}-\mathbf{P}\}f\|^{2}_{L^{2}_{v}(H^{4})},\nonumber\\
\sum_{i,j}\|\partial^{\alpha}\Gamma_{i,j}(r)\|^{2}&\leq C\|u\|^{2}_{H^{3}}\|\nabla_{x}\{\mathbf{I}-\mathbf{P}\}f\|^{2}_{L^{2}_{v}(H^{3})}.\nonumber
\end{align}
For $\Gamma_{i,j}(s)$, we have the following estimate:
\begin{align}
\Gamma_{i,j}(s)=-\rho\int\{v_{i}v_{j}+(v_{i}\partial_{v_{j}}+v_{i}\partial_{v_{j}})-(u_{i}v_{j}+u_{j}v_{i})\}\sqrt{M}\{\mathbf{I}-\mathbf{P}\}f {\rm d}v. \label{v2.19}
\end{align}
Now, we deal with the terms on the right hand side  of \eqref{v2.19}. First, we have
\begin{align}
\Big\|\partial^{\alpha}\int\rho u_{j}v_{i}\sqrt{M}\{\mathbf{I}-\mathbf{P}\}f {\rm d}v\Big\|^{2}
&=\int\Big|\int v_{i}\sqrt{M}\partial^{\alpha}(\rho u_{j}\{\mathbf{I}-\mathbf{P}\}f) {\rm d}v\Big|^{2}\, {\rm d}x \nonumber\\
&\leq C\int\int |\partial^{\alpha}(\rho u_{j}\{\mathbf{I}-\mathbf{P}\}f)|^{2}\, {\rm d}x  {\rm d}v\nonumber\\
&\leq C\|\rho\|^{2}_{H^{3}}\|u\|^{2}_{H^{3}}\|\nabla_{x}\{\mathbf{I}-\mathbf{P}\}f\|^{2}_{L^{2}_{v}(H_{x}^{3})}.\nonumber
\end{align}
Similarly we obtain that
\begin{align}
\Big\|\partial^{\alpha}\int\rho u_{i}v_{j}\sqrt{M}\{\mathbf{I}-\mathbf{P}\}f {\rm d}v\Big\|^{2}
&\leq C\|\rho\|^{2}_{H^{3}}\|u\|^{2}_{H^{3}}\|\nabla_{x}\{\mathbf{I}-\mathbf{P}\}f\|_{L^{2}_{v}(H_{x}^{3})},\nonumber\\
\Big\|\partial^{\alpha}\int\rho v_{i}v_{j}\sqrt{M}\{\mathbf{I}-\mathbf{P}\}f {\rm d}v\Big\|^{2}
&\leq C\|\rho\|^{2}_{H^{3}}\|\nabla_{x}\{\mathbf{I}-\mathbf{P}\}f\|^{2}_{L^{2}_{v}(H_{x}^{3})},\nonumber\\
\Big\|\partial^{\alpha}\int\rho (v_{i}\partial_{v_{j}}+v_{i}\partial_{v_{j}})\sqrt{M}\{\mathbf{I}-\mathbf{P}\}f {\rm d}v\Big\|^{2}
&\leq C\|\rho\|^{2}_{H^{3}}\|\nabla_{x}\{\mathbf{I}-\mathbf{P}\}f\|^{2}_{L^{2}_{v}(H_{x}^{3})}.\nonumber
\end{align}
Notice that
$$\sum_{i,j}\|\partial^{\alpha}(\partial_{i}b_{j}+\partial_{j}b_{i})\|^{2}=2\|\nabla_{x}\partial^{\alpha}b\|^{2}+2\|\nabla_{x}\cdot\partial^{\alpha}b\|^{2}.$$
Using the above equality and plugging the above estimates into \eqref{v2.18}, and then taking summation over $|\alpha|\leq 3$, one gets
\begin{align}
&\frac{\rm d}{{\rm d}t}\sum_{|\alpha|\leq3}\sum_{i,j}\int\partial^{\alpha}(\partial_{i}b_{j}+\partial_{j}b_{i})\partial^{\alpha}\Gamma_{i,j}(\{\mathbf{I}-\mathbf{P}\}f)\, {\rm d}x
+\|\nabla_{x}\partial^{\alpha}b\|^{2}+\|\nabla_{x}\cdot\partial^{\alpha}b\|^{2}\nonumber\\
\leq\,&\epsilon \|\nabla a\|^{2}_{H^{3}}+C_{\epsilon}\|\{\mathbf{I}-\mathbf{P}\}f\|^{2}_{L^{2}_{x}(H^{4}_{x})}
+C\|u-b\|^{2}_{H^{3}}
+C\big(\|(\rho,u)\|^{2}_{H^{3}}+\|\rho\|^{2}_{H^{3}}\|u\|^{2}_{H^{3}}\big)\nonumber\\
&\times\big(\|\nabla_{x}\{\mathbf{I}-\mathbf{P}\}f\|^{2}_{L^{2}_{v}(H^{3}_{x})}+\|u-b\|^{2}_{H^{3}}+\|\nabla_{x}(a,b)\|^{2}_{H^{3}}\big).\label{v2.20}
\end{align}
On the other hand, by means of \eqref{v2.13} and \eqref{v2.14}, it follows that
\begin{align}
&\!\!\!\!\!\!\!\!\!\|\partial^{\alpha}\nabla a\|^{2}
=-\frac{\rm d}{{\rm d}t}\int\partial^{\alpha}a\partial^{\alpha}\dv b\, {\rm d}x +\|\partial^{\alpha}\dv b\|^{2}\nonumber\\
&\qquad+\sum_{i}\int\partial^{\alpha}\partial_{i}a
 \partial^{\alpha}\Big\{(1+\rho)(u_{i}-b_{i})-\sum_{j}\partial_{j}\Gamma_{i,j}\{\mathbf{I}-\mathbf{P}\}f
+(1+\rho)u_{i}a\Big\}\, {\rm d}x \nonumber\\
\leq& -\frac{\rm d}{{\rm d}t}\int\partial^{\alpha}a\partial^{\alpha}\dv b\, {\rm d}x +\|\partial^{\alpha}\dv b\|^{2}
+\frac{1}{2}\|\partial^{\alpha}\nabla a\|^{2}
+C\|\nabla_{x}\{\mathbf{I}-\mathbf{P}\}f\|^{2}_{L^{2}_{v}(H^{3}_{x})}\nonumber\\
&  +C(1+\|\rho\|^{2}_{H^{3}})\big\{\|u-b\|^{2}_{H^{3}}+\|u\|^{2}_{H^{3}}\|\nabla a\|^{2}_{H^{3}}\big\}.\label{v2.21}
\end{align}
Summing \eqref{v2.21} over $|\alpha|\leq3$ and taking $\epsilon=\frac{1}{4}$, and then adding the results into \eqref{v2.20} implies  \eqref{v2.17}.
\end{proof}
\begin{Proposition}\label{vl2.5}
For smooth solutions of the problem \eqref{v1.5}-\eqref{v1.8}, we have
\begin{align}
&\frac{\rm d}{{\rm d}t}\sum_{|\alpha|\leq 3}\int_{\mathbb{R}^{3}}\partial^{\alpha}u\cdot\partial^{\alpha}\nabla \rho \, {\rm d}x + \lambda \|\nabla \rho\|^{2}_{H^{3}}\nonumber\\
&\quad \leq C(\|u-b\|^{2}_{H^{3}}+\|\nabla u\|^{2}_{H^{4}})+C(\|\rho\|_{H^{4}}+\|\rho\|^{8}_{H^{4}}+\|u\|^{2}_{H^{3}})\|\nabla(\rho,u)\|^{2}_{H^{3}}\label{v2.22}
\end{align}
for all $0\leq t<T $. 
\end{Proposition}
\begin{proof}
 Taking differentiation $\partial^\alpha \,(|\alpha|\leq 3)$ to \eqref{v1.7}, and by carrying an direct calculation, we get
\begin{align}
p'(1)\|\nabla \partial^{\alpha} \rho\|^{2}
\,=&-\int\nabla \partial^{\alpha} \rho \partial^{\alpha}\partial_{t} u\, {\rm d}x  + \int\nabla \partial^{\alpha} \rho \partial^{\alpha}(b-u)\, {\rm d}x \nonumber\\
& +\int\nabla \partial^{\alpha} \rho \partial^{\alpha}\Big\{-u\cdot \nabla u+\frac{1}{1+\rho}\Delta u-\Big[\frac{p'(1+\rho)}{1+\rho}-p'(1)\Big]\nabla \rho\Big\}\, {\rm d}x \nonumber\\
:=\,& Q_{1}+Q_{2}+Q_{3}. \label{v2.23}
\end{align}
For $Q_{i}\, (i=1,2,3)$, applying \eqref{v1.6}, H\"{o}lder's, Sobolev's and Young's inequalities, we have
\begin{align}
Q_{1}=\,&-\frac{\rm d}{{\rm d}t}\int\nabla \partial^{\alpha} \rho \partial^{\alpha}u \, {\rm d}x
+\int \partial^{\alpha}\ \dv u \partial^{\alpha}[(1+\rho)\ \dv u+u\cdot \nabla \rho]\, {\rm d}x \nonumber\\
\quad\leq\,& -\frac{\rm d}{{\rm d}t}\int\nabla \partial^{\alpha} \rho \partial^{\alpha}u \, {\rm d}x
+C\|\partial^{\alpha}\ \dv u\|^{2}+C\|\rho\|_{H^{4}}\|\nabla u\|^{2}_{H^{3}},\nonumber\\
Q_{2}\leq\, &\frac{p'(1)}{4}\|\nabla\partial^{\alpha}\rho\|^{2}+ C\|u-b\|^{2}_{H^{3}},\nonumber\\
Q_{3}\leq\,& \frac{p'(1)}{8}\|\nabla\partial^{\alpha}\rho\|^{2}+C\|u\cdot \nabla u\|^{2}_{H^{3}}
+C\Big\|\frac{1}{1+\rho}\Delta u\Big\|^{2}_{H^{3}}\nonumber\\
& +C\Big\|\Big[\frac{p'(1+\rho)}{1+\rho}-p'(1)\Big]\nabla \rho\Big\|^{2}_{H^{3}}\nonumber\\
\leq\,& \frac{p'(1)}{4}\|\nabla\partial^{\alpha}\rho\|^{2} +C\|u\|^{2}_{H^{3}}\|\nabla u\|^{2}_{H^{3}}
+C\|\rho\|_{H^{3}}\|\nabla\partial^{\alpha}\rho\|^{2}+C\|\nabla u\|^{2}_{H^{4}}\nonumber\\
&+C(\|\rho\|^{2}_{H^{3}}+\|\rho\|^{8}_{H^{3}})\|(\nabla\rho,\nabla u)\|^{2}_{H^{3}}.\nonumber
\end{align}
With the help of \eqref{v2.2}, plugging the above estimates into \eqref{v2.23}, we obtain \eqref{v2.22}.
\end{proof}

 Reorganizing the estimates obtained above in the Propositions \ref{vl2.2}-\ref{vl2.5}, one gets
 \begin{align}
&\frac{\rm d}{{\rm d}t}\mathcal{E}_{1}(t)+\lambda \mathcal{D}_{1}(t)\nonumber\\
\leq\, &C\big\{\|(\rho,u)\|_{H^{4}}+\|\rho\|_{H^{4}}\|u\|_{H^{4}}(1+\|\rho\|^{2}_{H^{4}}+\|\rho\|_{H^{4}}\|u\|_{H^{4}})
+\|\rho\|_{H^{4}}(1+\|\rho\|^{7}_{H^{4}})\big\}\nonumber\\
& \times\bigg\{\|\nabla (a,b,\rho,u)\|^{2}_{H^{3}}+\|b-u\|^{2}_{H^{4}}+\sum_{|\alpha|\leq 4}\|\partial^{\alpha}\{\mathbf{I}-\mathbf{P}\}f\|^{2}_{\nu}\bigg\}.\label{v2.24}
\end{align}

Next, we need to estimate  the mixed space-velocity derivatives of $f$, i.e., $\partial^\alpha_\beta f$.
Since $\|\partial^{\alpha}_{\beta}\mathbf{P}f\|\leq C \|\partial^{\alpha}f\|$ for any $\alpha$ and $\beta$,
 we only need to  estimate $\|\partial^{\alpha}_{\beta}\{\mathbf{I}-\mathbf{P}\}f\|$ below.
Let us apply $\mathbf{I}-\mathbf{P}$ to both sides of \eqref{v1.5} to get
\begin{align}
&\!\!\!\!\!\!\!\!\!\!\!\!\!\!\!\!
\partial_{t}\{\mathbf{I}-\mathbf{P}\}f+v\cdot \nabla_{x}\{\mathbf{I}-\mathbf{P}\}f+u\cdot \nabla_{v}\{\mathbf{I}-\mathbf{P}\}f
-\frac{1}{2}u\cdot v\{\mathbf{I}-\mathbf{P}\}f\nonumber\\
=\,&\mathcal{L}\{\mathbf{I}-\mathbf{P}\}f
+\mathbf{P}\Big\{v\cdot \nabla_{x}\{\mathbf{I}-\mathbf{P}\}f+u\cdot \nabla_{v}\{\mathbf{I}-\mathbf{P}\}f-\frac{1}{2}u\cdot v\{\mathbf{I}-\mathbf{P}\}f\Big\}\nonumber\\
&-\{\mathbf{I}-\mathbf{P}\}\Big\{v\cdot \nabla_{x}\mathbf{P}f+u\cdot\nabla_{v}\mathbf{P}f-\frac{1}{2}u\cdot v\mathbf{P}f\Big\} +\{\mathbf{I}-\mathbf{P}\}G.\label{v2.25}
\end{align}
where $\{\mathbf{I}-\mathbf{P}\}G$ is defined by
\begin{align}
\{\mathbf{I}-\mathbf{P}\}G:=\,&\rho \Big\{\mathcal{L}\{\mathbf{I}-\mathbf{P}\}f+\frac{1}{2}u\cdot v\{\mathbf{I}-\mathbf{P}\}f-u\cdot \nabla_{v}\{\mathbf{I}-\mathbf{P}\}f\nonumber\\
&\quad+\mathbf{P}\Big(u\cdot \nabla_{v}\{\mathbf{I}-\mathbf{P}\}f-\frac{1}{2}u\cdot v\{\mathbf{I}-\mathbf{P}\}f\Big)\nonumber\\
& \quad-\{\mathbf{I}-\mathbf{P}\}\Big(u\cdot\nabla_{v}\mathbf{P}f-\frac{1}{2}u\cdot v\mathbf{P}f\Big)\Big\}.\nonumber
\end{align}
Now we give  a simple  derivation of the equality \eqref{v2.25}.
First,
\begin{align}
\{\mathbf{I}-\mathbf{P}\}(v\cdot \nabla_{x}f)&= v\cdot \nabla_{x}\{\mathbf{I}-\mathbf{P}\}f +v\cdot \mathbf{P}f
-\mathbf{P}(v\cdot \nabla_{x}\{\mathbf{I}-\mathbf{P}\}f)-\mathbf{P}(v\cdot \nabla_{x}\mathbf{P}f)\nonumber\\
&=v\cdot \nabla_{x}\{\mathbf{I}-\mathbf{P}\}f+\{\mathbf{I}-\mathbf{P}\}(v\cdot \nabla_{x}\mathbf{P}f)-\mathbf{P}(v\cdot \nabla_{x}\{\mathbf{I}-\mathbf{P}\}f).\nonumber
\end{align}
We can deal with the terms $\{\mathbf{I}-\mathbf{P}\}(u\cdot \nabla_{v}f),\{\mathbf{I}-\mathbf{P}\}(u\cdot v f)$ in the same way.
Meanwhile, by the definitions of $\mathcal{L}\,\text{and}\,\mathbf{P}$, one has
\begin{align}
&\{\mathbf{I}-\mathbf{P}\}(u\cdot v\sqrt{M})\equiv 0 ,\ \quad
 \{\mathbf{I}-\mathbf{P}\}\mathcal{L}f =\mathcal{L}\{\mathbf{I}-\mathbf{P}\}f.\nonumber
\end{align}
Then  \eqref{v2.25} follows.

In the proof of the following Proposition, we adopt some ideas from \cite[Lemma 4.3]{DFT}.

\begin{Proposition}\label{vl2.6}
Let $1\leq k\leq4$. For smooth solutions of the problem \eqref{v1.5}-\eqref{v1.8}, we have
\begin{align}
&\!\!\!\!\!\!\!\!\!\!\!\!\!\!\!\!\frac{\rm d}{{\rm d}t}\sum_{\substack{|\beta|=k \\ |\alpha|+|\beta|\leq 4}}\|\partial^{\alpha}_{\beta}\{\mathbf{I}-\mathbf{P}\}f\|^{2}
+\lambda \sum_{\substack{|\beta|=k \\ |\alpha|+|\beta|\leq 4}}\|\partial^{\alpha}_{\beta}\{\mathbf{I}-\mathbf{P}\}f\|^{2}_{\nu}\nonumber\\
\leq \,&C(1+\|(\rho,u)\|^{2}_{H^{3}}+\|\rho\|^{2}_{H^{3}}\|u\|^{2}_{H^{3}})
\sum_{|\alpha'|\leq4-k+1}\|\partial^{\alpha'}\{\mathbf{I}-\mathbf{P}\}f\|^{2}_{\nu}\nonumber\\
&+C(\|\rho\|_{H^{3}}+\|u\|^{2}_{H^{3}}+\|\rho\|^{2}_{H^{3}}\|u\|^{2}_{H^{3}})
\sum_{\substack{1\leq |\beta'|\leq 4 \\|\alpha'|+|\beta'|\leq 4}}\|\partial^{\alpha'}_{\beta'}\{\mathbf{I}-\mathbf{P}\}f\|^{2}_{\nu}\nonumber\\
&+C\|\nabla b\|^{2}_{H^{4-k}}+C(1+\|\rho\|^{2}_{H^{3}})\|u\|^{2}_{H^{4-k}}\|\nabla b\|^{2}_{H^{3}}\nonumber\\
&+C\chi_{2\leq k\leq 4}(1+\|\rho\|^{2}_{H^{3}})\sum_{\substack{1\leq |\beta'|\leq k-1\\
|\alpha'|+|\beta'|\leq 4 }}\|\partial^{\alpha'}_{\beta'}\{\mathbf{I}-\mathbf{P}\}f\|^{2}_{\nu} \label{v2.26}
\end{align}
for all $0\leq t<T $. 
Here  $\chi_{E}$ denotes the characteristic function of the  set $E$.
\end{Proposition}

\begin{proof}
Fix $k\,(1\leq k\leq 4)$.  Choosing $\alpha$ and $\beta$ such that  $|\beta|=k$ and $|\alpha|+|\beta|\leq 4$,
  multiplying \eqref{v2.25} by $\partial^{\alpha}_{\beta}\{\mathbf{I}-\mathbf{P}\}f$ and then taking integration, one has
\begin{align}
\frac{1}{2}\frac{\rm d}{{\rm d}t}\|\partial^{\alpha}_{\beta}\{\mathbf{I}-\mathbf{P}\}f\|^{2}+\int\langle -L\partial^{\alpha}_{\beta}
\{\mathbf{I}-\mathbf{P}\}f,\{\mathbf{I}-\mathbf{P}\}f\rangle   \, {\rm d}x :=\sum_{i=1}^{7}J_{i}\label{v2.27}
\end{align}
with
\begin{align}
&J_{1}=\int\langle -\partial^{\alpha}_{x}[\partial^{\beta}_{v},v\cdot \nabla_{x}]\{\mathbf{I}-\mathbf{P}\}f,\partial^{\alpha}_{\beta}\{\mathbf{I}-\mathbf{P}\}f\rangle   \, {\rm d}x ,
\nonumber\\
&J_{2}=\int\langle \partial^{\alpha}_{x}[\partial^{\beta}_{v},-|v|^{2}]
\{\mathbf{I}-\mathbf{P}\}f,\partial^{\alpha}_{\beta}\{\mathbf{I}-\mathbf{P}\}f\rangle   \, {\rm d}x ,\nonumber\\
&J_{3}=\int\langle -\partial^{\alpha}_{\beta}(u\cdot \nabla_{v}\{\mathbf{I}-\mathbf{P}\}f),\partial^{\alpha}_{\beta}\{\mathbf{I}-\mathbf{P}\}f\rangle   \, {\rm d}x ,\nonumber\\
&J_{4}=\int\langle \frac{1}{2}\partial^{\alpha}_{\beta}(u\cdot v\{\mathbf{I}-\mathbf{P}\}f),\partial^{\alpha}_{\beta}\{\mathbf{I}-\mathbf{P}\}f\rangle   \, {\rm d}x ,\nonumber\\
&J_{5}=\int\Big\langle \partial^{\alpha}_{\beta}\mathbf{P}\Big(v\cdot \nabla_{x}\{\mathbf{I}-\mathbf{P}\}f+u\cdot \nabla_{v}\{\mathbf{I}-\mathbf{P}\}f\nonumber\\
&\qquad \qquad\qquad  -\frac{1}{2}u\cdot v\{\mathbf{I}-\mathbf{P}\}f\Big),\partial^{\alpha}_{\beta}\{\mathbf{I}-\mathbf{P}\}f\Big\rangle   \, {\rm d}x ,\nonumber\\
&J_{6}=\int\Big\langle -\partial^{\alpha}_{\beta}\{\mathbf{I}-\mathbf{P}\}\Big(v\cdot \nabla_{x}\mathbf{P}f+u\cdot\nabla_{v}\mathbf{P}f-\frac{1}{2}u\cdot v\mathbf{P}f\Big),
\partial^{\alpha}_{\beta}\{\mathbf{I}-\mathbf{P}\}f\Big\rangle   \, {\rm d}x ,\nonumber\\
&J_{7}=\int\langle \partial^{\alpha}_{\beta}\{\mathbf{I}-\mathbf{P}\}G,\partial^{\alpha}_{\beta}\{\mathbf{I}-\mathbf{P}\}f\rangle   \, {\rm d}x .\nonumber
\end{align}
Here the fact that $[\partial^{\beta}_{v},\mathcal{L}]=[\partial^{\beta}_{v},-|v|^{2}]$ has been used.

Now  we start estimating each term $J_{i}$ in \eqref{v2.27}. For the  terms $J_i\,(i=1,\cdots, 6)$, we have
\begin{align}
J_{1}&\leq \eta \|\partial^{\alpha}_{\beta}\{\mathbf{I}-\mathbf{P}\}f\|^{2}
+C_{\eta}\|[\partial^{\beta}_{v},v\cdot\nabla_{v}]\partial^{\alpha}_{x}\{\mathbf{I}-\mathbf{P}\}f\|^{2}\nonumber\\
&\leq \eta \|\partial^{\alpha}_{\beta}\{\mathbf{I}-\mathbf{P}\}f\|^{2}
+C_{\eta}\sum_{|\alpha'|\leq 4-k}\|\partial^{\alpha'}\nabla_{x}\{\mathbf{I}-\mathbf{P}\}f\|^{2}\nonumber\\
&\quad +\chi_{2\leq k\leq4} C_{\eta}\sum_{\substack{1\leq |\beta'|\leq k-1 \\ |\alpha'|+|\beta'|\leq 4}}\|\partial^{\alpha'}_{\beta'}\{\mathbf{I}-\mathbf{P}\}f\|^{2},\nonumber\\
J_{2}&\leq \eta \|\partial^{\alpha}_{\beta}\{\mathbf{I}-\mathbf{P}\}f\|^{2}
+C_{\eta}\|[\partial^{\beta}_{v},-|v|^{2}]\partial^{\alpha}_{x}\{\mathbf{I}-\mathbf{P}\}f\|^{2}\nonumber\\
&\leq \eta \|\partial^{\alpha}_{\beta}\{\mathbf{I}-\mathbf{P}\}f\|^{2}
+C_{\eta}\sum_{|\alpha'|\leq 4-k}\|\partial^{\alpha'}\{\mathbf{I}-\mathbf{P}\}f\|^{2}_{\nu}\nonumber\\
&\quad  +\chi_{2\leq k\leq4} C_{\eta}\sum_{\substack{1\leq |\beta'|\leq k-1 \\ |\alpha'|+|\beta'|\leq 4}}\|\partial^{\alpha'}_{\beta'}\{\mathbf{I}-\mathbf{P}\}f\|^{2}_{\nu},\nonumber\\
J_{3}&\leq \eta \|\partial^{\alpha}_{\beta}\{\mathbf{I}-\mathbf{P}\}f\|^{2}
+C_{\eta}\|\partial^{\alpha}_{x}(u\cdot\nabla_{v}\partial^{\beta}_{v}\{\mathbf{I}-\mathbf{P}\}f)\|^{2}\nonumber\\
&\leq \eta \|\partial^{\alpha}_{\beta}\{\mathbf{I}-\mathbf{P}\}f\|^{2}
+C_{\eta}\|u\|^{2}_{H^{3}}\sum_{\substack{1\leq |\beta'|\leq 4 \\  |\alpha'|+|\beta'|\leq 4 }}\|\partial^{\alpha'}_{\beta'}\{\mathbf{I}-\mathbf{P}\}f\|^{2},\nonumber\\
J_{4}&\leq \eta \|\partial^{\alpha}_{\beta}\{\mathbf{I}-\mathbf{P}\}f\|^{2}
+C_{\eta}\|\partial^{\alpha}_{x}\big(u\cdot\partial^{\beta}_{v}(v\{\mathbf{I}-\mathbf{P}\}f)\big)\|^{2}\nonumber\\
&\leq \eta \|\partial^{\alpha}_{\beta}\{\mathbf{I}-\mathbf{P}\}f\|^{2}
+C_{\eta}\|u\|^{2}_{H^{3}}\sum_{\substack{1\leq |\beta'|\leq 4 \\  |\alpha'|+|\beta'|\leq 4 }}\|\partial^{\alpha'}_{\beta'}\{\mathbf{I}-\mathbf{P}\}f\|^{2}_{\nu}\nonumber\\
&\quad +C_{\eta}\|u\|^{2}_{H^{3}}\sum_{  |\alpha'|\leq 4-k }\|\partial^{\alpha'}\{\mathbf{I}-\mathbf{P}\}f\|^{2},\nonumber\\
J_{5}&\leq \eta \|\partial^{\alpha}_{\beta}\{\mathbf{I}-\mathbf{P}\}f\|^{2}
+C_{\eta}\|\partial^{\alpha}_{\beta}\mathbf{P}(v\cdot \nabla_{x}\{\mathbf{I}-\mathbf{P}\}f)\|^{2}\nonumber\\
&\quad +C_{\eta}\|\partial^{\alpha}_{\beta}\mathbf{P}(u\cdot \nabla_{v}\{\mathbf{I}-\mathbf{P}\}f)\|^{2}
+C_{\eta}\|\partial^{\alpha}_{\beta}\mathbf{P}(u\cdot v\{\mathbf{I}-\mathbf{P}\}f)\|^{2}\nonumber\\
&\leq \eta \|\partial^{\alpha}_{\beta}\{\mathbf{I}-\mathbf{P}\}f\|^{2}
+C_{\eta}\sum_{|\alpha'|\leq 4-k}\|\nabla_{x}\partial^{\alpha'}_{x}\{\mathbf{I}-\mathbf{P}\}f\|^{2}\nonumber\\
&\quad +C_{\eta}\|u\|^{2}_{H^{3}}\sum_{|\alpha'|\leq 4-k}\|\partial^{\alpha'}_{x}\{\mathbf{I}-\mathbf{P}\}f\|^{2},\nonumber\\
J_{6}&\leq \eta \|\partial^{\alpha}_{\beta}\{\mathbf{I}-\mathbf{P}\}f\|^{2}
+C_{\eta}\|\partial^{\alpha}_{\beta}\{\mathbf{I}-\mathbf{P}\}(v\cdot \nabla_{x}\mathbf{P}f)\|^{2}\nonumber\\
&\quad +C_{\eta}\|\partial^{\alpha}_{\beta}\{\mathbf{I}-\mathbf{P}\}(u\cdot \nabla_{v}\mathbf{P}f)\|^{2}
+C_{\eta}\|\partial^{\alpha}_{\beta}\{\mathbf{I}-\mathbf{P}\}(u\cdot v \mathbf{P}f)\|^{2}\nonumber\\
&\leq \eta \|\partial^{\alpha}_{\beta}\{\mathbf{I}-\mathbf{P}\}f\|^{2}
+C_{\eta}\big(\|\nabla b\|^{2}_{H^{4-k}}+\|\nabla b\|^{2}_{H^{2}}\|u\|^{2}_{H^{4-k}}\big).\nonumber
\end{align}

For the term $J_{7}$, we have the following calculation and estimates:
\begin{align}
J_{7}=\,&\int \langle \partial^{\alpha}_{\beta}\big(\rho L\{\mathbf{I}-\mathbf{P}\}f\big),\partial^{\alpha}_{\beta}\{\mathbf{I}-\mathbf{P}\}f\rangle \, {\rm d}x  \nonumber\\
&+\frac{1}{2}\int\langle\partial^{\alpha}_{\beta}\big(\rho u\cdot v\{\mathbf{I}-\mathbf{P}\}f\big),\partial^{\alpha}_{\beta}\{\mathbf{I}-\mathbf{P}\}f\rangle \, {\rm d}x\nonumber\\
&-\int\langle\partial^{\alpha}_{\beta}\big(\rho u\cdot \nabla_{v}\{\mathbf{I}-\mathbf{P}\}f\big),\partial^{\alpha}_{\beta}\{\mathbf{I}-\mathbf{P}\}f\rangle \, {\rm d}x\nonumber\\
&+\int\Big\langle\partial^{\alpha}_{\beta}\mathbf{P}\Big(\rho u\cdot \nabla_{v}\{\mathbf{I}-\mathbf{P}\}f-\frac{1}{2}\rho u\cdot v\{\mathbf{I}-\mathbf{P}\}f\Big),
\partial^{\alpha}_{\beta}\{\mathbf{I}-\mathbf{P}\}f\Big\rangle \, {\rm d}x\nonumber\\
&-\int\Big\langle\partial^{\alpha}_{\beta}\{\mathbf{I}-\mathbf{P}\}\Big(\rho u\cdot \nabla_{v}\mathbf{P}f-\frac{1}{2}\rho u\cdot v\mathbf{P}f\Big),
\partial^{\alpha}_{\beta}\{\mathbf{I}-\mathbf{P}\}f\Big\rangle \, {\rm d}x\nonumber\\
:=\,& \sum_{i=1}^{5}Y_{i}.\nonumber
\end{align}
We can adopt the above similar estimates to deal with  $Y_{i}\,(2\leq i\leq5)$. Thus, we only give the following bounds:\begin{align}
Y_{2}&\leq \eta\|\partial^{\alpha}_{\beta}\{\mathbf{I}-\mathbf{P}\}f\|^{2}
\nonumber\\
& \quad  +C_{\eta}\|\rho\|^{2}_{H^{3}}\|u\|^{2}_{H^{3}}
\bigg\{\sum_{|\alpha'|\leq 4-k}\|\partial^{\alpha'}\{\mathbf{I}-\mathbf{P}\}f\|^{2}
+\sum_{\substack{1\leq |\beta'|\leq 4 \\ |\alpha'|+|\beta'|\leq 4}}\|\partial^{\alpha'}_{\beta'}\{\mathbf{I}-\mathbf{P}\}f\|^{2}_{\nu}\bigg\},\nonumber\\
Y_{3}&\leq \eta\|\partial^{\alpha}_{\beta}\{\mathbf{I}-\mathbf{P}\}f\|^{2}
+C_{\eta}\|\rho\|^{2}_{H^{3}}\|u\|^{2}_{H^{3}}\sum_{\substack{1\leq |\beta'|\leq 4 \\ |\alpha'|+|\beta'|\leq 4}}\|\partial^{\alpha'}_{\beta'}\{\mathbf{I}-\mathbf{P}\}f\|^{2},\nonumber\\
Y_{4}&\leq \eta\|\partial^{\alpha}_{\beta}\{\mathbf{I}-\mathbf{P}\}f\|^{2}
+C_{\eta}\|\rho\|^{2}_{H^{3}}\|u\|^{2}_{H^{3}}\sum_{|\alpha'|\leq 4-k}\|\partial^{\alpha'}\{\mathbf{I}-\mathbf{P}\}f\|^{2},\nonumber\\
Y_{5}&\leq \eta\|\partial^{\alpha}_{\beta}\{\mathbf{I}-\mathbf{P}\}f\|^{2}
+C_{\eta}\|\rho\|^{2}_{H^{3}}\|u\|^{2}_{H^{4-k}}\|\nabla _{x}b\|^{2}_{H^{2}}.\nonumber
\end{align}
For $Y_{1}$, we give a detailed calculation:
\begin{align}
Y_{1}
=\,&\iint\partial_{x}^{\alpha}\big(\rho L\partial_{v}^{\beta}\{\mathbf{I}-\mathbf{P}\}f\big)\partial^{\alpha}_{\beta}\{\mathbf{I}-\mathbf{P}\}f\, {\rm d}x  {\rm d}v \nonumber\\
 &+\iint\partial_{x}^{\alpha}\big(\rho[\partial_{v}^{\beta},-|v|^{2}]\{\mathbf{I}-\mathbf{P}\}f\big)\partial^{\alpha}_{\beta}
\{\mathbf{I}-\mathbf{P}\}f\, {\rm d}x  {\rm d}v,\nonumber\\
=\,&\iint\rho L\partial^{\alpha}_{\beta}\{\mathbf{I}-\mathbf{P}\}f\partial^{\alpha}_{\beta}\{\mathbf{I}-\mathbf{P}\}f\, {\rm d}x  {\rm d}v \nonumber\\
&+\sum_{1\leq \gamma \leq \alpha}C_{\alpha,\gamma}
\iint \partial^{\gamma}\rho L\partial^{\alpha-\gamma}_{\beta}\{\mathbf{I}-\mathbf{P}\}f\partial^{\alpha}_{\beta}\{\mathbf{I}-\mathbf{P}\}f\, {\rm d}x  {\rm d}v\nonumber\\
&+\iint\partial_{x}^{\alpha}\big(\rho[\partial_{v}^{\beta},-|v|^{2}]\{\mathbf{I}-\mathbf{P}\}f\big)\partial^{\alpha}_{\beta}
\{\mathbf{I}-\mathbf{P}\}f\, {\rm d}x  {\rm d}v,\nonumber\\
:=\,& Y_{11}+Y_{12}+Y_{13}.\nonumber
\end{align}
For $Y_{11}$, we can move it to the left hand side of the equality \eqref{v2.27}. Thus, we only need to  deal with $Y_{12}$ and $Y_{13}$.  We have
\begin{align}
Y_{12}&=-\iint\partial^{\gamma}\rho \sqrt{M}\nabla_{v}\big(M^{-\frac{1}{2}}
\partial^{\alpha-\gamma}_{\beta}\{\mathbf{I}-\mathbf{P}\}f\big)\sqrt{M}\nabla_{v}
\big(\partial^{\alpha}_{\beta}\{\mathbf{I}-\mathbf{P}\}f\big)\, {\rm d}x  {\rm d}v\nonumber\\
&\leq C \|\partial^{\alpha}_{\beta}\{\mathbf{I}-\mathbf{P}\}f\|_{\nu}\|\rho\|_{H^{3}}
\sum_{ \substack{|\beta'|=k \\ |\alpha'|+|\beta'|\leq 4}}\|\partial^{\alpha'}_{\beta'}\{\mathbf{I}-\mathbf{P}\}f\|_{\nu}\nonumber\\
&\leq C\|\rho\|_{H^{3}}
\sum_{\substack{ |\beta'|=k \\ |\alpha'|+|\beta'|\leq 4}}\|\partial^{\alpha'}_{\beta'}\{\mathbf{I}-\mathbf{P}\}f\|^{2}_{\nu},\nonumber\\
Y_{13}&\leq \eta \|\partial^{\alpha}_{\beta}\{\mathbf{I}-\mathbf{P}\}f\|^{2}
+C_{\eta}\|[\partial^{\beta}_{v},-|v|^{2}]\partial^{\alpha}_{x}\rho\{\mathbf{I}-\mathbf{P}\}f\|^{2}\nonumber\\
&\leq \eta \|\partial^{\alpha}_{\beta}\{\mathbf{I}-\mathbf{P}\}f\|^{2}
+C_{\eta}\|\rho\|^{2}_{H^{3}}\Big\{\sum_{|\alpha'|\leq 4-k}\|\partial^{\alpha'}\{\mathbf{I}-\mathbf{P}\}f\|^{2}_{\nu}\nonumber\\
&\quad +\chi_{2\leq k\leq4}  \sum_{\substack{1\leq |\beta'|\leq k-1 \\ |\alpha'|+|\beta'|\leq 4}}\|\partial^{\alpha'}_{\beta'}\{\mathbf{I}-\mathbf{P}\}f\|^{2}_{\nu}\Big\}.\nonumber
\end{align}
For the second term in the left hand side of the equality \eqref{v2.27}, one gets
\begin{align}
&\quad \int(1+\rho)\langle -L\partial^{\alpha}_{\beta}\{\mathbf{I}-\mathbf{P}\}f, \partial^{\alpha}_{\beta}\{\mathbf{I}-\mathbf{P}\}f\rangle   \, {\rm d}x\nonumber\\
&\geq\lambda_{1}\|\{\mathbf{I}-\mathbf{P}_{0}\}f\|^{2}_{\nu}\nonumber\\
&\geq\frac{\lambda_{1}}{2}\|\partial^{\alpha}_{\beta}\{\mathbf{I}-\mathbf{P}\}f\|^{2}_{\nu}
-\lambda_{1}\|\mathbf{P}_{0}\partial^{\alpha}_{\beta}\{\mathbf{I}-\mathbf{P}\}f\|^{2}_{\nu}\nonumber\\
&\geq\frac{\lambda_{1}}{2}\|\partial^{\alpha}_{\beta}\{\mathbf{I}-\mathbf{P}\}f\|^{2}_{\nu}
-C\|\partial^{\alpha}\{\mathbf{I}-\mathbf{P}\}f\|^{2}.\nonumber
\end{align}
Plugging all the above estimates into \eqref{v2.27} and choosing $\eta$ sufficiently small, we obtain \eqref{v2.26}.
\end{proof}
\begin{Remark}\label{vr2.1}
According to the above Lemma, we can choose some suitable constants $C_{k}$, such that
\begin{align}
&\frac{\rm d}{{\rm d}t}\sum_{1\leq k\leq4}C_{k}\sum_{\substack{|\beta|=k \\ |\alpha|+|\beta|\leq 4}}\|\partial^{\alpha}_{\beta}\{\mathbf{I}-\mathbf{P}\}f\|^{2}
+\lambda\sum_{\substack{1\leq |\beta|\leq 4 \\ |\alpha|+|\beta|\leq 4}}\|\partial^{\alpha}_{\beta}\{\mathbf{I}-\mathbf{P}\}f\|^{2}_{\nu}\nonumber\\
\leq\, & C(\|(\rho,u)\|^{2}_{H^{3}}+\|\rho\|^{2}_{H^{3}}\|u\|^{2}_{H^{3}})\sum_{|\alpha|\leq 4}\|\partial^{\alpha}\{\mathbf{I}-\mathbf{P}\}f\|^{2}_{\nu}\nonumber\\
&+C(\|\rho\|_{H^{3}}+\|u\|^{2}_{H^{3}}+\|\rho\|^{2}_{H^{3}}\|u\|^{2}_{H^{3}})
\sum_{\substack{1\leq||\beta|\leq 4  \\ \alpha|+|\beta|\leq 4}}\|\partial^{\alpha}_{\beta}\{\mathbf{I}-\mathbf{P}\}f\|^{2}_{\nu}\nonumber\\
&+C\|u\|^{2}_{H^{3}}(1+\|\rho\|^{2}_{H^{3}})\|\nabla b\|^{2}_{H^{3}}
+C\Big(\|\nabla b\|^{2}_{H^{3}}+\sum_{|\alpha|\leq 4}\|\partial^{\alpha}\{\mathbf{I}-\mathbf{P}\}f\|^{2}_{\nu}\Big).\label{v2.28}
\end{align}
\end{Remark}

With the aid of the inequalities \eqref{v2.24}, \eqref{v2.28} and \eqref{v2.2}, we have
\begin{align}
\frac{\rm d}{{\rm d}t}\mathcal{E}(t)+\lambda \mathcal{D}(t)\leq C\big(\mathcal{E}^{\frac{1}{2}}(t)+\mathcal{E}^{2}(t)\big)\mathcal{D}(t)
\leq C(\delta+\delta^{2}) \mathcal{D}(t). \label{v2.29}
\end{align}
So, as long as $0<\delta < 1$ is sufficiently small, the  integration in time of \eqref{v2.29} yields
\begin{align}
\mathcal{E}(t)+\lambda \int^{t}_{0}\mathcal{D}(s){\rm d}s\leq \mathcal{E}(0) \label{v2.30}
\end{align}
for all $0\leq t<T $.  Besides, \eqref{v2.2} can be justified by choosing
\begin{align}
\mathcal{E}(0) \sim \|f_{0}\|^{2}_{H^{4}_{x,v}}+\|(\rho_{0},u_{0})\|^{2}_{H^{4}}\nonumber
\end{align}
sufficiently small.

\subsection{The case of periodic domain}

In this subsection we deal with the uniform a priori  estimate
when $\Omega $ is a spatial periodic domain $\mathbb{T}^{3}$.
Using the following conservation laws in the case of torus,
\begin{align}
&\frac{\rm d}{{\rm d}t}\iint  F \, {\rm d}x  {\rm d}v=0, \quad
\frac{\rm d}{{\rm d}t}\int  n \, {\rm d}x =0,\nonumber\\
&\frac{\rm d}{{\rm d}t}\bigg\{\int   nu\, {\rm d}x  +\iint  v F\, {\rm d}x  {\rm d}v\bigg\}=0,\nonumber
\end{align}
and by the assumptions of Theorem \ref{vt1.2}, it follows that
\begin{align}
 \int a\, {\rm d}x =0, \quad  \int \rho \, {\rm d}x =0, \quad \int \big(b+(1+\rho)u\big)\, {\rm d}x =0 \label{v4.1}
 \end{align}
for all $t\geq 0.$

Thanks to the Poincar\'{e} inequality and the conservation laws \eqref{v4.1}, we have
\begin{align}
\|a\|_{L^{2}}&\leq C\|\nabla a\|_{L^{2}}, \quad \|\rho\|_{L^{2}}\leq C\|\nabla \rho\|_{L^{2}}, \label{v4.2}\\
\|u+b\|_{L^{2}}&\leq \|b+u+\rho u\|_{L^{2}}+\|\rho u\|_{L^{2}}\nonumber\\
&\leq C \|\nabla (b+u+\rho u)\|_{L^{2}}+ \|u\|_{L^{\infty}}\|\rho\|_{L^{2}}\nonumber\\
&\leq C \|\nabla (b,u)\|_{L^{2}}+ C\|u\|_{H^{2}}\|\nabla \rho\|_{L^{2}}+C\|\rho\|_{H^{2}}\|\nabla u\|_{L^{2}}.\label{v4.3}
\end{align}
Similarly to the whole space  case, we  can    obtain the  following estimates:
\begin{align}
&\frac{\rm d}{{\rm d}t}\mathcal{E}_{1}(t)+\lambda\mathcal{D}_{1}(t) \leq C(\mathcal{E}^{\frac{1}{2}}_{1}+\mathcal{E}^{2}_{1})\mathcal{D}_{1}(t),\label{v4.4}\\
&\frac{\rm d}{{\rm d}t}\mathcal{E}_{2}(t)+\lambda\mathcal{D}_{2}(t) \leq  C\mathcal{D}_{1}(t)
+C(\mathcal{E}_{1}+\mathcal{E}^{2}_{1})\mathcal{D}_{1}(t)
+C(\mathcal{E}^{\frac{1}{2}}_{1}+\mathcal{E}^{2}_{1})\mathcal{D}_{2}(t).\label{v4.5}
\end{align}
According to the definition of $\mathcal{D}_{\mathbb{T},1}(t)$, we have
\begin{align}
\mathcal{D}_{\mathbb{T},1}(t)\sim \sum_{|\alpha|\leq 4}\|\{\mathbf{I}-\mathbf{P}\}\partial^{\alpha}f\|^{2}_{\nu}+\|(a,b,\rho,u)\|^{2}_{H^{4}}.\label{v4.6}
\end{align}
Combining \eqref{v4.2},\,\eqref{v4.3} and \eqref{v4.4} together, we conclude that
\begin{align}
\frac{\rm d}{{\rm d}t}\mathcal{E}_{1}(t)+\lambda\mathcal{D}_{\mathbb{T},1}(t) \leq C(\mathcal{E}^{\frac{1}{2}}_{1}+\mathcal{E}^{2}_{1})\mathcal{D}_{\mathbb{T},1}(t).\label{v4.7}
\end{align}

\medskip\medskip
{\bf Acknowledgements:}
The authors are very grateful to the referees for their constructive
comments and valuable suggestions, which considerably improved the earlier
version of this paper. Part of this work was  done when the second author  visited the Department of Mathematics,
University of Pittsburgh. She thanks the department for its hospitality. F.  Li was supported in part by NSFC (Grant Nos. 11271184, 11671193), PAPD and China Scholarship Council.
Y. Mu was supported by China Scholarship Council and the Doctoral Starting up Foundation of Nanjing University of Finance \& Economics (MYMXW16001).
D. Wang's research was supported in part by the National Science Foundation under grants DMS-1312800 and DMS-1613213.

%

 \medskip

\end{document}